\documentclass[12pt]{amsart}
\usepackage[english]{babel}

\usepackage[
pdftex,hyperfootnotes]{hyperref}
\hypersetup{
	colorlinks=true,
	linkcolor=NavyBlue, 
	urlcolor=RoyalPurple,
	citecolor=OliveGreen,
	pdftitle={Multiple orthogonal polynomials and random walks},
	bookmarks=true,
	pdfpagemode=FullScreen,
}
\usepackage[usenames,dvipsnames,svgnames,table,x11names]{xcolor}
\usepackage{pgfplots}
\usepackage{tikz}
\usepackage{tikz-3dplot}
\usetikzlibrary{automata,quotes, chains,matrix,calc,shadows,shapes.callouts,shapes.geometric,shapes.misc,positioning,patterns,decorations.shapes,
decorations.pathmorphing,decorations.markings,decorations.fractals,decorations.pathreplacing,shadings,fadings,arrows.meta,bending}

\usepackage{nicematrix}
\usepackage[utf8]{inputenc}
\usepackage{comment}
\usepackage[textwidth=17.5cm,textheight=22.275cm,
height=
24.275cm,
width=18cm]{geometry}

\usepackage{rotating,multirow,pdflscape}
\usepackage{amssymb,latexsym,amsmath,amsthm,bm}
\usepackage{mathrsfs}
\usepackage{mathtools,arydshln,mathdots}
\usepackage{bigints}
\makeatletter
\renewcommand*\env@matrix[1][*\c@MaxMatrixCols c]{%
\hskip -\arraycolsep
\let\@ifnextchar\new@ifnextchar
\array{#1}}
\makeatother

\mathtoolsset{centercolon}
\usepackage[table]{xcolor}
\usepackage[toc,page]{appendix}

\usepackage{tocvsec2}

\usepackage{Baskervaldx}
\usepackage{newtxmath }
\newtheorem{coro}{Corollary}

\newtheorem{defi}{Definition}
\newtheorem{teo}{Theorem}

\newtheorem{pro}{Proposition}
\newtheorem{lemma}{Lemma}

\newtheorem{rem}{Remark}
\newtheorem{con}{Conjecture}

\renewcommand{\d}{\operatorname{d}}

\newcommand{\diag}{\operatorname{diag}}

\newcommand{\I}{\mathbb{I}}
\newcommand{\C}{\mathbb{C}}

\newcommand{\N}{\mathbb{N}}
\newcommand{\R}{\mathbb{R}}
\newcommand{\Z}{\mathbb{Z}}

\newcommand{\Beta}{\boldsymbol{\beta}}
\newcommand{\Alpha}{\boldsymbol{\alpha}}
\newcommand{\GGamma}{\boldsymbol{\gamma}}
\newcommand{\1}{\text{\fontseries{bx}\selectfont \textup 1}}

\makeatletter
\def\@settitle{\begin{center}%
\baselineskip14\p@\relax
\bfseries
\uppercasenonmath\@title
\@title
\ifx\@subtitle\@empty\else
\\begin{align*}1ex]\uppercasenonmath\@subtitle
\footnotesize\mdseries\@subtitle
\fi
\end{center}%
}
\def\subtitle#1{\gdef\@subtitle{#1}}
\def\@subtitle{}
\makeatother

\usepackage{longtable}
\usepackage{enumerate}
\allowdisplaybreaks

\DeclareRobustCommand{\rchi}{{\mathpalette\irchi\relax}}
\newcommand{\irchi}[2]{\raisebox{\depth}{$#1\chi$}} 

\title{Multiple Orthogonal Polynomials and Random Walks}

\author[A Branquinho]{Amílcar Branquinho$^{1,\flat}$}
\address{$^1$Departamento de Matemática,
Universidade de Coimbra, 3001-454 Coimbra, Portugal}
\email{ajplb@mat.uc.pt}

\author[A Foulquié]{Ana Foulquié-Moreno$^{2,\natural}$}
\address{$^2$Departamento de Matemática, Universidade de Aveiro, 3810-193 Aveiro, Portugal}
\email{foulquie@ua.pt}

\author[M Mañas]{Manuel Mañas$^{3,\clubsuit}$}
\address{$^3$Departamento de Física Teórica, Universidad Complutense de Madrid, Plaza Ciencias 1, 28040-Madrid, Spain \&
Instituto de Ciencias Matematicas (ICMAT), Campus de Cantoblanco UAM, 28049-Madrid, Spain}
\email{manuel.manas@ucm.es}

\author[C Álvarez-Fernández]{Carlos Álvarez-Fernández$^4$}
\address{$^4$Departamento de Métodos Cuantitativos, Universidad Pontificia Comillas \\
28015-Madrid,~Spain}
\email{calvarez@cee.upcomillas.es}

\author[JE Fernández-Díaz]{Juan E. Fernández-Díaz$^5$}
\email{juanen01@ucm.es}

\thanks{$^\flat$Acknowledges Centro de Matemática da Universidade de Coimbra (CMUC) -- UID/MAT/00324/2019, funded by the Portuguese Government through FCT/MEC and co-funded by the European Regional Development Fund through the Partnership Agreement PT2020}

\thanks{$^\natural$Acknowledges CIDMA Center for Research and Development in Mathematics and Applications (University of Aveiro) and the Portuguese Foundation for Science and Technology (FCT) within project UIDB/MAT/UID/04106/2020 and UIDP/MAT/04106/2020}

\thanks{$^\clubsuit$Thanks financial support from the Spanish ``Agencia Estatal de Investigación'' research project [PGC2018-096504-B-C33], \emph{Ortogonalidad y Aproximación: Teoría y Aplicaciones en Física Matemática}.}



\begin{document}

\maketitle

\begin{abstract}
Given a non-negative Jacobi matrix describing higher order recurrence relations for multiple orthogonal polynomials of type~II and corresponding linear forms of type~I, a general strategy for constructing a pair of stochastic matrices, dual to each other, is provided. The corresponding Markov chains (or 1D random walks) allow, in one transition, to reach for
the $N$-th previous states, to remain in the state or reach for the immediately next state.
The dual Markov chains allow, in one transition, to reach for the $N$-th next states, to remain in the state or reach for immediately previous state.
The connection between both dual Markov chains is discussed at the light of the Poincaré's theorem on ratio asymptotics for homogeneous linear recurrence relations and the Christoffel--Darboux formula within the sequence of multiple orthogonal polynomials and linear forms of type~I. 

The Karlin--McGregor representation formula is extended to both dual random walks, and applied to the discussion of the corresponding generating functions and first-passage distributions.
Recurrent or transient character of the Markov chain is discussed. Steady state and some conjectures on its existence and the relation with mass points are also given.

The Jacobi--Piñeiro multiple orthogonal polynomials are taken as a case 
study of the described results. For the first time in the literature, an explicit formula for the type~I Jacobi--Piñeiro polynomials is determined. Then, the region of parameters where the Jacobi matrix is non-negative is given.
Moreover, two stochastic matrices, describing two dual random walks, one allowing to reach for the two previous states, remain or reach for the next, and the other allowing to reach for the two next states, remain or reach for the previous, are given in terms of the Jacobi matrix and the values of the multiple orthogonal polynomials of type~II and corresponding linear forms of type~I at the point~$1$.

The region of parameters where the Markov chains are recurrent or transient is given, and it is conjectured that when recurrent, the Markov chains are null recurrent and, consequently, the expected return times are infinity. Examples of recurrent and transient Jacobi--Piñeiro random walks are constructed explicitly.
\end{abstract}

\thispagestyle{empty}

\clearpage

\tableofcontents

\section{Introduction}

In this paper we will show that multiple orthogonality, of both types~I and~II,  is related to Markov chains beyond birth and death chains.

The interplay of orthogonal polynomials and stochastic processes is quite old. We can refer the role played by Hermite polynomials in the theory of stochastic processes and the integration with respect to the Wiener process~\cite{wiener, Ito}.

The 1950s witnessed important advances in the understanding of the links between orthogonal polynomials and stochastic processes. Indeed, highly influential papers on the role of spectral representation of birth and death processes probabilities appeared in those years. Let us mention the celebrated papers by Kendal, Ledermman and Reuter~\cite{Kendall,Kendall2,Ledermann} and also the seminal works by Samuel Karlin and James McGregor; in particular,~\cite{KmcG1957-1,KmcG1957-2} were devoted to birth and death Markov processes, touching differential and classification aspects,
and where the integral representation of the transient probability matrix reveals the intimate relation between the theory of birth and death processes and the theory of the Stieltjes moment problem, while in~\cite{KmcG} was devoted to random walks, that is uncountable Markov chains. One of the key findings was the Karlin--MacGregor representation formula, that allow for a integral representation in terms of the orthogonal polynomials, of some relevant probabilistic objects of the corresponding stochastic processes, analyzing also the usual recurrence and absorption characteristics of them.
For an account on this see~\cite{Schoutens} as well as~\cite{Grunbaum1,mirta}.

Nowadays, \emph{random walk polynomial sequence} is defined as a standard orthogonal polynomial sequence which is orthogonal with respect to a measure on $[- 1, 1]$ and which is such that the three term recurrence relation coefficient are nonnegative, and any measure with respect to which a random walk polynomial sequence is orthogonal is a random walk measure. Random walk polynomials have become a classical subject in the literature on orthogonal polynomials, for example, see Chapter~4 in~\cite{Ismail} for a recent  account of some of its relevant aspects.

A revival of this circle of ideas took place at the 1970-80's. The work~\cite{Whitehurst} discussed simple random walks and integral representation in terms of orthogonal polynomials and the support of the spectral measure. In the papers~\cite{Ogura,Engel} an integral relation between the Poisson process and the discrete Charlier orthogonal polynomials was given and in~\cite{Diaconis} the Stein equations for some well-known distributions, including Pearson's class, were linked with corresponding orthogonal polynomials.
 
Orthogonal polynomials and Karlin--McGregor ideas found applications in queuing problems, see~\cite{Van Doorn0,Chiara} and~\cite{Charris} as well. In~\cite{Kijima} the authors represented the conditional limiting distribution of a birth and death process in terms of the birth-death polynomials. Van Doorn and Schrinjer~\cite{Van Doorn} studied \emph{random walk polynomials} and random walk measures relevant in the analysis of random walks, giving properties of random walk measures and polynomials, and obtaining a limit theorem for random walk measures which is of interest in the study of random walks. 
Later, the same authors~\cite{Coolen} studied some aspects of discrete-time birth-death processes or random walks, highlighting the role played by orthogonal polynomials. In particular, it was shown how one can establish whether a given sequence of orthogonal polynomials is a sequence of random walk polynomials, and whether a given random walk measure corresponds to a unique random walk.


More recently, in this~XXI century there has been several advances in this rich network of relations between probability and orthogonal polynomials. For example, in~\cite{Schoutens} the Krawtchouk polynomials were shown to play an important role, similar to the Hermite--Brownian and the Charlier--Poisson cases, in stochastic integration theory with respect to the binomial process, relating the very classical orthogonal polynomials with Stein's method for Pearson's and Ord's classes of distributions, and extending Karlin--McGregor results by considering doubly limiting conditional distributions, giving a probabilistic interpretation to many orthogonal families in the Askey scheme.


In~\cite{Kovchegov} proposals that go beyond near neighbors were given, the main idea was to study the~spectrum of a polynomial of a given transition matrix.
In~\cite{Obata} the Karlin--McGregor formula was reformulated in terms of one-mode interacting Fock spaces and an integral expression for the moments of an associated operator is given. This integral expression leads to an extension of the Karlin--McGregor formula to the graph of paths connected with a {clique}.

Our interest in random walks and orthogonal polynomials was inspired by the talks and works of Alberto Grünbaum and collaborators~\cite{Grunbaum1, Grunbaum11,Grunbaum2,Grunbaum3}. In particular,~\cite{Grunbaum1} (see also the already quoted~\cite{Kovchegov}) proposed the search for generalized orthogonal polynomials describing Markov chains beyond birth and death chains, i.e. allowing for jumps not only among nearest neighbors. The proposal of Grünbaum for such an extension were matrix orthogonal polynomials. These discussions and results posed us a question:
\begin{quote}
\emph{Do multiple orthogonal polynomials have some role in the description of Markov chains beyond birth and death chains?}
\end{quote}
We hope with this paper give some light to this fascinating subject.

\smallskip

The contents of the paper are as follows. Within this introduction we include a brief subsection to describe the basics from stochastic processes required for the reading of this paper. We follow the excellent book by Gallager~\cite{Gallager} as well as the classical texts~\cite{Karlin-Taylor1,Karlin-Taylor2} and~\cite{feller}. Then,~\S\ref{Section:multiple} recollects the essential features we need in the sequel sections about multiple orthogonality. The point of view here is that of Gaussian factorization of the matrix of moments. This classical technique used initially to solve linear systems of equations has proven useful in this field in general and for multiple orthogonal polynomials in particular, that is why we will follow~\cite{afm} closely when we introduce factorization tools and notation.

Given a particular sequence of weights, based on the Euclidean division, we construct a semi-infinite matrix of moments, and a Gauss--Borel factorization delivers multiple orthogonal polynomials of types~I and~II, and also directly leads to biorthogonality. This approach also gives the homogeneous linear recurrence relations satisfied by both types of polynomials. i.e. multi-diagonal (more than three) Jacobi matrices, as well as the reproducing kernel and the corresponding Christoffel--Darboux formula within the sequence of multiple orthogonal polynomials.

The next two sections contain the main results and theorems of this paper. In~\S\ref{Section:Stochastic} we give a general strategy for constructing stochastic matrices once a non negative Jacobi matrix is known. For that aim we assume the zeros of the orthogonal polynomials of type~II and linear forms of type~I belong to a bounded set, which is satisfied for example for~AT-systems.

The basic technique here is to normalize the Jacobi matrix using the values of the multiple orthogonal polynomials of type~II at some value $\lambda$, not in the bounded set containing the zeros, giving an stochastic matrix that we say of type~II, and also to normalize the transposed Jacobi matrix by the linear forms evaluated at $\lambda$ giving an stochastic matrix of type~I, see Theorems~\ref{pro:sigma_spectral} and~\ref{pro:sigma_spectral_I}, respectively. 
These matrices describe Markov chains in where in each transition the $N$ previous states can be reached as well as the next state, in the type~II situation; for the type~I matrix, the $N$ next states can be reached in one transition and also the previous state. Then, using Poincaré's theorem~\cite{poincare} on the ratio asymptotics for homogeneous linear recurrences we get the connection between the previous mentioned dual stochastic matrices, when all the roots of the corresponding characteristic polynomial have distinct absolute value, see Theorem~\ref{teo:I_and_II}.

Then, in Theorem~\ref{teo:removing_roots} and Corollary~\ref{cor:removing_roots}, using the Christoffel--Darboux formula within the sequence described in~\S\ref{Section:multiple} we see that we can depress the order of the homogeneous linear recurrences to better characterize the ratio asymptotics of the type~II multiple orthogonal polynomials and linear forms, respectively, and, consequently, the connection of the dual stochastic matrices.

In Theorem~\ref{teo:steady} we provide a candidate for a steady state constructed in terms of the orthogonal polynomials of type~II and linear forms of type I evaluated at $\lambda$. We conjecture that when $\lambda$ is a mass point this candidate belongs to $\ell_1$ and therefore we can normalize it to a probability vector. Then, in Theorem~\ref{teo:KMcG}, we extend to multiple orthogonal polynomials of type~I and~II the representation formula of Karlin and McGregor, and in Theorem~\ref{teo:KMcG2} the generating functions for transition probabilities and first passage are given. Theorem~\ref{teo:recurrent_state} gives the integral formula characterizing recurrent and transient Markov chains.

Finally, in~\S\ref{Section:JP} we use the Jacobi--Piñeiro multiple orthogonal polynomials as a case study. We recall the explicit expressions for Jacobi--Piñeiro multiple orthogonal polynomials of type~II, and in Theorem~\ref{theorem:JPI} we give the explicit expressions for the two families of Jacobi--Piñeiro multiple orthogonal polynomials of type~I. To our best knowledge this is the first time such an expression is found. It is also expressed in terms of a generalized hypergeometric functions ${}_3F_2$.

In Theorem~\ref{teo:positive_JP} the region of parameters in which the Jacobi--Piñeiro's Jacobi matrix is nonnegative is determined. In Theorem~\ref{teo:JPII_stochastic} and Corollary~\ref{coro:stochastic_JPII} the stochastic matrix of type~II is given and the coefficients are explicitly determined as rational functions in the Jacobi--Piñeiro parameters
 $\alpha,\beta,\gamma$ and~$n$. Moreover, in Corollary~\ref{coro:large_n_limit_II} we give the large~$n$ limit of this stochastic matrix, this leads to a splitting in terms of a Toeplitz semi-stochastic matrix and a compact matrix. Then, in Theorem~\ref{teo:JPI_stochastic} the Jacobi--Piñeiro type~I stochastic matrix is given. These matrices describe Markov chains in where in each transition the two previous states can be reached as well as the next state, in the type~II situation; for the type~I matrix, the two next states can be reached in one transition and also the previous state. 
 
In Proposition~\ref{pro:JP_stochastic_dual} it is proven, using Poincaré's theorem and the Christoffel--Darboux formula that,  in the large $n$-limit,  the  dual stochastic matrices are transposed to each  other. In Proposition~\ref{pro:recurrentJP} the region of parameters classifying recurrent and transient Markov chains is determined. Finally, some particular examples, of types~I and~II, recurrent and transient, are discussed in more~detail.

We also include an Appendix for a general theorem regarding the type~II case.


\subsection{Markov chains, random walks and all that}

This is a brief review of some basic concepts about stochastic processes used in this work, for a deeper account see for instance~\cite{Karlin-Taylor1,Karlin-Taylor2,feller, Gallager}.

According to~\cite{Gallager}, see also~\cite{Karlin-Taylor1}, a \emph{countable Markov chain} is an integer-time process $\{X_n\}_{n\in\N_0}$, $\N_0=\{0,1,2,\ldots\}$, where the random variables $X_n$ lie in a countable set, that for simplicity we take as $\N_0$, and depend on the past only through the most recent random variable $X_{n-1}$. That is, for all choices $\{j,i=i_1,i_2,i_3,\ldots\}\subset\N_0$ we have the following condition for the corresponding conditional probabilities,
\begin{align*}
\operatorname{Pr}[X_n=j|X_{n-1}=i_1,X_{n-2}=i_2,X_{n-3}=i_3,\ldots]
=\operatorname{Pr}[X_n=j|X_{n-1}=i] .
\end{align*}
Moreover, $\operatorname{Pr}[X_n=j|X_{n-1}=i]$ depends only on $i$ and $j$ and not on $n$,
\begin{align*}
\operatorname{Pr}[X_n=j|X_{n-1}=i]=P_{ij}.
\end{align*}
Here $P_{ij}$ is the probability of reaching to state $j$ given that the previous state is $i$. Therefore the matrix $P=(P_{ij})_{i,j\in \N}$ of transition probabilities is a semi-infinite stochastic matrix. That~is, $0\leq P_{ij}\leq 1$ and for every $i\in\N_0$ we have $\sum_{i=0}^\infty P_{ij}=1$. In terms of the vector
\begin{align*}
\1= \begin{pNiceMatrix}
1\\1\\\Vdots
\end{pNiceMatrix},
\end{align*}
we have that a non-negative matrix $P$ (i.e., with non-negative coefficients) is a stochastic or Markov matrix if it satisfies
\begin{align}\label{eq:stochastic_condition}
P \, \1 =\1.
\end{align}
If we only know that $\sum_{i=0}^\infty P_{ij}\leq1$ we say that $P$ is semi-stochastic.
Given a sequence $\{X_i\}_{i\in\N_0}$ of independent identically distributed random variables the integer-time stochastic process~$\{S_n=X_0+X_1+\cdots+X_n\}_{n\in\N_0}$ is called~\emph{one-dimensional random walk} based on~$\{X_i\}_{i\in\N_0}$. If the random variables are integer-valued the 1D random walk can be represented as a countable Markov chain. In this paper we indistinctly use both terms, Markov chain and random walk (that we assume to be 1D).

The \emph{first-passage-time probability} $f_{ij}^n$ is the probability that, if initially at state $i$, the state~$j$ is visited for the first time after $n$ transitions. That is, for $n=1$, $f^1_{ij}=f_{ij}=P_{ij}$ and, for~$n>1$
\begin{align*}
f_{ij}^n=\operatorname{Pr}[X_n=j,X_{n-1}\neq j,\ldots, X_1\neq j|X_0=i].
\end{align*}
Notice the difference with $P^n_{ij}=(P^n)_{i,j} =\operatorname{Pr}[X_n=j|X_0=i]$.
In fact, $f^n_{ij}$ is the probability that the \emph{first} entry at $j$ occurs at time $n$, while $P_{ij}^n$ is the probability that \emph{any} entry occurs at time $n$.

We then define, for $n\in\N$, $F^n_{ij}$ as the probability that, being initially at $X_0=i$, the state~$j$ occurs at some time between $1$ and $n$ inclusive, i.e. $F^n_{ij}=\sum_{m=1}^nf_{ij}^m$. Notice that~$F^n_{ij}$ is a non-decreasing function in $n$ and bounded by $1$, so that $F^\infty_{ij}=\lim_{n\to \infty}F_{ij}^n$ exists and represents the probability that if initially $X_0=i$ the state $j$ will ever happen.

If $F^\infty_{ij}=1$, then if we are initially at the state $i$ the state $j$ will be reached with probability~$1$. In this situation we introduce the random variable $T_{ij}$ conditional to $X_0=i$ as the \emph{first-passage time} from the state $i$ to the state $j$. Then, $f_{ij}^n$ is the probability mass function of~$T_{ij}$ and $F^n_{ij}$ its cumulative distribution function. Whenever $F_{ij}^\infty<1$, the random variable~$T_{ij}$ is a defective random variable as there is a chance that there is no first passage at all.
If~$F_{jj}^\infty=1$, we say that the state~$j$~is~\emph{recurrent} and if $F_{jj}^\infty<1$ we say that the state~is~\emph{transient}.


Two states are said to be in the same \emph{class} if there is path joining both. A path is constructed through the stochastic matrix $P$, two states can be joined if $P_{i,j}\neq 0$. In our case there is only one class.
In fact the class is irreducible, as there is no exit, all states communicate.
One can show that if $j$ is recurrent, $F^\infty_{jj}$, the every other state $i$ in that class is recurrent, $F_{ii}^\infty$; moreover, $F_{ji}^\infty=F_{ij}^\infty=1$.
If $\{N_{jj}(t)\}_{t\geq 0}$ is the \emph{counting process }for occurrences of state $j$ up to time $t$ in a Markov chain with initial state $j$, one has that the state $j$ is recurrent iff $\lim_{t\to\infty}N_{jj}(t)=\infty$ with probability $1$, which in turn is equivalent to
$\lim_{t\to\infty} \operatorname{E}(N_{jj}(t))=\infty$
that happens if
$\lim_{t\to\infty}\sum_{1\leq n\leq t}P_{jj}^n=\infty$,
i.e., for a recurrent state the counting process is \emph{renewal process.} Then, for the \emph{expectation mean values of the first-passage time }we have either $\operatorname{E}(T_{jj})<\infty$ or $\operatorname{E}(T_{jj})=+\infty$.

A recurrent state is said \emph{positive recurrent} if $\operatorname{E}(T_{jj})<\infty$, that is in mean the time of return is finite, and \emph{null recurrent} if $\operatorname{E}(T_{jj})=\infty$, i.e. the mean of the return time is infinite.
If there is only one class of states, as is in our case, all the states are either all positive recurrent, all null recurrent or all transient.

A vector $\boldsymbol{\pi}=(\pi_0,\pi_1,\ldots)$ with $\pi_i\geq 0$ for $i\in\N_0$ and $\sum_{i=0}^\infty \pi_i=1$, is known as a \emph{probability vector}, and can be understood as describing an initial state with probabilities $\pi_i$ of being in the $i$-th state. Observe that a probability vector $\boldsymbol{\pi}$ satisfies $\boldsymbol{\pi} \, \1=1$.
Assuming a probability vector $\boldsymbol{\pi}(0)$ as the initial state, after one transition the new probability of being at the $k$-th state will be $ \pi_k(1)=\sum_{i=0}^\infty\pi_i(0)P_{ik}$. That is, the new probability vector will be
${\boldsymbol{\pi}}(1)=\boldsymbol{\pi} (0)P$. Notice that ${\boldsymbol{\pi}}(1) \1=\boldsymbol{\pi}(0)P \, \1=\boldsymbol{\pi}(0) \, \1=1$.
Thus, after $n$ transitions the probability vector will be $\boldsymbol{\pi}(n)=\boldsymbol{\pi}(0 )P^n$,
which essentially is the Chapman--Kolmogorov equation $P^{n+m}_{ij}=\sum_{k=0}^\infty P^{n}_{ik}P^{m}_{kj}$.
A steady state, is an invariant probability vector $\boldsymbol{\pi}$, i.e., it does not evolve in the discrete time so that $\boldsymbol{\pi}=\boldsymbol \pi P$.
For an irreducible Markov chain, if there exists a steady state, this state is unique, 
the Markov chain is positive recurrent, and
\begin{align*}
\pi_j=\frac{1}{\operatorname{E}(T_{jj})}.
\end{align*}
Conversely, if the chain is positive recurrent, then there exists a steady sate.

The \emph{period} of a given state $j$ in a Markov chain is defined as the
$\operatorname{gcd}\{n\in\N_0: P^n_{jj}>0\}$ (gcd stands for greatest common divisor). For our banded transitions matrices, if $P_{jj}>0$ then $d=1$. Aperiodic Markov chains correspond precisely to period $1$, and
\begin{align*}
\lim_{n\to\infty}P_{jj}^n
&=\frac{1}{\operatorname{E}(T_{jj})}. \end{align*}
Moreover, Blackwell's theorem~\cite{Gallager} for delayed renewal processes ensures that
\begin{align*}
\lim_{n\to\infty}P_{ij}^n
=\frac{1}{\operatorname{E}(T_{jj})} ,  && i \in \N_0.
\end{align*}
A Markov chain that is recurrent positive and aperiodic is said ergodic. Therefore, for our Markov chains with $P_{jj}>0$ both concepts are equivalent. For an ergodic Markov chain, the steady state probability $\pi_j$ is, with probability $1$, the time-average rate of that state~$\frac{1}{\operatorname{E}(T_{jj})}$.


\section{Multiple orthogonal polynomials}\label{Section:multiple}

Multiple orthogonality 
is a very close topic to that of simultaneous rational approximation (simultaneous Padé aproximants) to systems of Cauchy transforms of measures. The history of simultaneous rational approximation starts in 1873 with the well known article~\cite{He} in where Charles Hermite proved the transcendence of Euler's constant $e$. Later on, along the years 1934-35, Kurt Mahler delivered in the University of Groningen several lectures~\cite{Ma} where he settled down the foundations of this theory. In the mean time, two Mahler's students, Coates and Jager, made important contributions in this respect (see~\cite{Co} and~\cite{Ja}).

There are two formulations of multiple orthogonality, first of all we have so called type~I and type~II. Both are equivalent in the sense of duality. If we set a problem involving orthogonality conditions regarding several measures (or several weights) we call it a type~II problem. Under this view, the fundamental objects are polynomials. The dual objects for these polynomials are linear forms considered in the type~I version of multiple orthogonality. Due to the fact that perhaps type~II version is more ``natural'', research interest has been centered in this characterization (cf.~\cite{abv}). Very few times there has been explicit calculations of linear forms that appear naturally in the type~I characterization (see~e.g.~\cite{leurs_vanassche}).


Some of the classical properties are those related to zeros distribution and properties, in particular the interlacing property and confinement (very classical results in standard orthogonality theory) are studied in~\cite{HVa,filg,k}. We will be interested in the case where there is an upper bound for the zeros of the polynomials.


Another well known classical result in orthogonality, along with recurrence relations (and a consequence of them) is the Christoffel--Darboux formula for the reproductive kernel (also known as CD kernel). Barry Simon~\cite{simon} makes an extensive review in the non-multiple case dealing with the real line and the unit circle among others. Studying the asymptotic behavior of this kernel is essential in the study of universality classes in random matrix theory due to its connection with eigenvalue distribution.

Multiple orthogonal polynomials have also different expressions for the Christoffel--Darboux kernel, coming from the different ways of defining them from the recurrence relation.
Sorokin and Van Iseghem~\cite{sorokin} derived a formula that can be applied to multiple orthogonal polynomials, also did Jonathan Coussement and Walter Van Assche in~\cite{Coussement__VanAssche}, and Luís Cotrim in~\cite{tesis}. Daems and Kuijlaars derived a~CD formula for the mixed multiple case~\cite{daems-kuijlaars,daems-kuijlaars2} using Riemann--Hilbert approach in the context of non-intersecating Brownian motions. Later on,  the article~\cite{afm} reproduces the same result with an algebraic approach and in the framework 
of mixed multiple orthogonal polynomial sequences. The~CD formula derived in~\cite{CD} suits particularly well to our problem because it is expressed only in terms of a unique polynomial sequence.

Let ${\mathcal{M}}(\Delta)$ denote all the finite Borel measures which have support, with infinitely many points in the interval $\Delta\subset \R$, where they do not change sign.
A weight on $\Delta$ is a real integrable function defined on $\Delta$ which does not change its sign on $\Delta$.
For a finite Borel measure $\mu \in {\mathcal{M}}(\Delta)$ we consider a system of weights $\vec w=(w_1,\ldots,w_p)$ on $\Delta$, with $p \in {\mathbb{N}}$, and a multi-index $\vec \nu=(\nu_1,\ldots, \nu_p) \in \N_0^p$, and denote $|\vec \nu|=\nu_1+\cdots+\nu_p$. Then, there exist polynomials, $A_{\vec \nu,1},\ldots, A_{\vec \nu,p}$, not all identically equal to zero, which satisfy the following orthogonality relations
\begin{align}\label{tipoI}
\int_{\Delta} x^{j} \sum_{a=1}^{p} A_{\vec \nu, a}(x)w_{a} (x)\d\mu (x)&=0, & \deg
A_{\vec \nu, a}&\leq\nu_{a}-1,& j&\in\{0,\ldots, |\vec \nu|-2\}.
\end{align}
Analogously, there exists a polynomial $B_{\vec{\nu}}$ not identically equal to zero, such that
\begin{align}\label{tipoII}
\int_{\Delta} B_{\vec\nu}(x) w_{a} (x) x^{j} \d\mu (x)&=0, & \deg B_{\vec\nu} &\leq|\vec \nu|,& j&=0,\ldots, \nu_a-1, \quad a=1,\ldots,p.
\end{align}

These families of polynomials are, respectively called, type~I and type~II multiple orthogonal polynomials, with respect to the combination $(\mu, \vec w, \vec \nu)$ of the measure $\mu$, the system of weights $\vec w$ and the multi-index $\vec \nu$. When $p=1$ both definitions coincide with standard orthogonal polynomials on the real line. The existence of a system of polynomials $(A_{{\vec\nu},1},\ldots, A_{{\vec\nu},p})$ and a polynomial $B_{\vec \nu }$ defined from~\eqref{tipoI} and~\eqref{tipoII} respectively, is ensured because in both cases, finding the coefficients of the polynomials is equivalent to solving a system of $|\vec\nu|$ linear homogeneous equations with $|\vec \nu|+1$ unknown coefficients.

From the theory of orthogonal polynomials we know that when $p=1$ each polynomial~$A_1 \equiv B$ has exactly degree
 $|\vec \nu|=\nu_1$; for $p>1$ that is not true in general. For instance, if $\vec w=(w_1,w_1,\ldots,w_1)$ the solution linear space has dimension bigger than one, and we can find two solutions which are linearly independent. Hence, there is at least an $a \in \{1,\ldots,p\}$ such that $\deg A_{\vec{\nu },a} < \nu_a-1$ and $\deg B< |\vec \nu|$. Given a measure $\mu \in {\mathcal{M}}(\Delta)$ and a system of weights $\vec w$ on $\Delta$ a multi-index $\vec \nu$ is called type~I or type~II normal if $\deg A_{\vec{\nu },a}$ must be equal to $\nu_a-1$, $a=1,\ldots,p$, or $\deg B$ must be equal to $|\vec \nu|-1$, respectively. When for a pair $(\mu, \vec w)$ all the multi-indices are type~I normal (respectively, type~II normal), then the pair is called type~I perfect (respectively, type~II perfect).


\subsection{The moment matrix}

Within this $\S$\ref{Section:multiple} we follow~\cite{afm}.
In this multiple scenario, the moment matrix is defined in terms of a measure $\mu \in {\mathcal{M}}(\Delta)$ and a system of weights $\vec w$ on
$\Delta \subset {\mathbb{R}}$, as well as corresponding compositions $\vec
n=(n_{1},\ldots, n_{p})\in\N^{p}$; we denote $|\vec n|:=n_1+\cdots+n_p$.
Let us recall the generalized Euclidean division. Given $i \in \N_0 
$
there are unique nonnegative integers $q(i),a(i),r(i)$, such that the partition
\begin{align}\label{i}
i&=q(i)|\vec n|+n_{1}+\cdots+n_{a(i)-1}+r(i), & 0&\leq r(i)<n_{a(i)},
\end{align}
holds, in terms of which we define
\begin{align}\label{k}
k(i)&:=q(i) n_{a(i)}+r(i),& 0&\leq r(i)<n_{a(i)}.
\end{align}
Combining~\eqref{i} and~\eqref{k} we can obtain a formula that given a $k$ and $a$ return back the corresponding
\begin{align}\label{iak}
i(k,a)=\left\lfloor\frac{k}{n_a}\right\rfloor \left(|\vec n|-n_a\right)+n_1+\cdots +n_{a-1}+k .
\end{align}
Let ${\mathbb{R}}^{\infty}$ denote the vector space of all sequences with elements in ${\mathbb{R}}$. An element $\lambda \in {\mathbb{R}}^{\infty}$ may be interpreted as a column semi-infinity vector as follows
\begin{align*}
\lambda&= (\lambda^{(0)},\lambda^{(1)},\ldots)^\top,& \lambda^{ ( { j} )} &\in \R,& j&\in\N_0.
\end{align*}
We consider the set $\{e_j\}_{j\geq0} \subset {\mathbb{R}}^{\infty}$ with
\begin{align*}
e_j=(\overset{j}{\overbrace{0,0,\ldots,0}},1,0,0,\ldots)^{\top}.
\end{align*}
Analogously, we denote by $(\R^p)^\infty$ the set of all sequences of vectors with $p$ components and observe that each sequence which belongs to $(\R^p)^\infty$ can also be understood as semi-infinity column vector: given the vector sequence $(\vec v_0,\vec v_1,\ldots)$ with $\vec v_j=(v_{j,1},\ldots,v_{j,p})^\top$ we have the corresponding sequence in $\R^\infty$ given~by $(v_{0,1}, \ldots , v_{0,p},v_{1,1}, \ldots , v_{1,p} , \ldots)$; i.e., ${\mathbb{R}}^{\infty} \cong({\mathbb{R}}^p)^{\infty}$. Therefore, we also consider the set $\{e_a(k)\}_{\substack{a=1,\ldots, p \\ k=0,1,\ldots}}\subset ({\mathbb{R}}^p)^{\infty}$ where for each pair $(a,k)\in \{1,\ldots,p\}\times \N_0$, $e_{a}(k)=e_{i(k,a)}$, and the function $i(a,k) \in \N_0$ is given~by~\eqref{iak}.

The monomial vectors
\begin{align}\label{chion}
\rchi_a&:=\sum_{k=0}^\infty e_a(k) x^k,&
\rchi_a^{(l)}&=\begin{cases}
x^{k(l)},& a=a(l),\\
0, &a\neq a(l),
\end{cases}
\end{align}
are important in the construction of the moment matrix.
These vectors may be understood as sequences of monomials according to the
composition $\vec n$ introduced previously.
We also define the following \emph{vector of undressed linear forms of type~I }
\begin{align}
\xi_{\vec n}&:=\sum_{a=1}^p \rchi_a w_a, &
\xi_{\vec n}^{(l)}&= w_{a(l)}x^{k(l)},& a=a(l),
\end{align}
which is a sequence of weighted monomials for each given composition $\vec n$. Sometimes, when we want to stress the dependence on the composition, we write
$\rchi_{\vec n,a}, \rchi_{\vec n}$ and $\xi_{\vec n}$.
We also consider
\begin{align*}
\rchi:=
\begin{pNiceMatrix}
1,x,x^2,\ldots 
\end{pNiceMatrix}^\top
\end{align*}

\begin{defi}
The moment matrix $g\equiv g_{\vec w}\equiv g_{\vec n}$ (we use these alternative notations depending on what is needed to be emphasized) is given~by
\begin{align}
\label{compact.g}
g\equiv g_{\vec w}\equiv g_{\vec n}:=\int_\Delta \rchi(x)(\xi_{\vec n}(x))^\top \d \mu(x).
\end{align}
\end{defi}
For each $j\in \N_0$ there exists a unique generalized Euclidean decomposition
\begin{align*}
 \quad j=q|\vec n|+n_{1}+ \cdots +n_{a-1}+r.
\end{align*}
Hence, taking $k= q n_a+r$ the coefficients $g_{i,j}\in {\mathbb{R}}$
of the moment matrix are
\begin{align}\label{explicit g2}
\begin{aligned}
g_{i,j}&=\int_\Delta x^{i+k ( { j } )}w_{a ( { j } )}(x)\d \mu(x) .
\end{aligned}
\end{align}

\subsection{The Gauss--Borel factorization}
Here we follow~\cite{afm}. Given a perfect combination $(\mu, \vec w)$ we consider

\begin{defi}
The Gauss--Borel factorization of a semi-infinite moment matrix $g_{\vec n}$, determined by $(\mu, \vec w)$, is the
problem of finding the solution of
\begin{align}\label{facto}
g&=S^{-1} H \tilde S^{-\top},
\end{align}
With $S,\tilde S$ lower unitriangular semi-infinite matrices
\begin{align*}
	S&=\left(\begin{NiceMatrix}[columns-width = auto]
		1 & 0 &\Cdots &\\
		S_{1,0 } & 1&\Ddots&\\
		S_{2,0} & S_{2,1} & \Ddots &\\
		\Vdots & \Ddots& \Ddots&
	\end{NiceMatrix}\right), & \tilde S&=
	\left(\begin{NiceMatrix}[columns-width = auto]
		1 & 0 &\Cdots &\\
		\tilde S_{1,0 } & 1&\Ddots&\\
		\tilde S_{2,0} & \tilde S_{2,1} & \Ddots &\\
		\Vdots & \Ddots& \Ddots&
	\end{NiceMatrix}\right),
\end{align*}
and $H$ a semi-infinite diagonal matrix
\begin{align*}
H=\diag
\begin{pNiceMatrix}
H_0,H_1,\ldots 
\end{pNiceMatrix},
\end{align*}
with $H_l\neq 0$, $l\in\N_0$.

In terms of these matrices we construct the polynomials
\begin{align} \label{defmops}
B^{(l)}&:=x^l+\sum_{i=0}^{l-1} S_{l,i}x^{i},&A^{(l)}_a&:=\frac{1}{H_l}\sum_{\substack{i=0,1,\ldots,l\\ a(i)=a}}\tilde S_{l,i} x^{k(i)} .
\end{align}
For $A^{(l)}_a$ the sum is taken for a given $a\in\{1,\ldots,p\}$ over those $i$ such that $a=a(i)$ and $i\leq l$.
\end{defi}

This factorization makes sense whenever all the principal minors of $g$ do not vanish,~i.e.,~if
$\det g^{[l]}\neq 0 $, 
$l=1,2,\ldots $,
and in our case it is true because $(\mu, \vec w)$ is perfect. With the use of the coefficients of the matrices $S$ and $\tilde S$ we construct multiple orthogonal polynomials.

For that aim, given $i\in\N_0 $ and $\vec n\in\N^{p} $
we introduce the degree multi-index $\vec \nu\in \N_0^{p}$ with
\begin{align}\label{mult:u}
\vec \nu(i)&=(\nu_1(i),\ldots,\nu_p(i)), &
\nu_b(i)=\begin{cases}
(q(i)+1)n_b, & b\in\{1,\ldots,a(i)-1\},\\
q(i)n_{a(i)}+r(i),& b=a(i),\\
q(i)n_{b},& b\in\{ a(i)+1,\ldots, p \} ,
\end{cases}
\end{align}
that satisfy
\begin{align}\label{formulitas}
k(i)&=\nu_{a(i)}(i),& | \vec\nu(i) |&=i,&
\vec \nu(i+l \, |\vec n|)&=\vec\nu(i)+l \, \vec n,
\end{align}

In the following table we gather together these numbers and vectors for two simple~cases:


\smallskip

$\phantom{ol}\hspace{-1.225cm}$\begin{tabular}{|c|c|c|c|c|c|c|} \hline
\multicolumn{7}{|c|}{$\vec n=(1,1)$ } \\
\hline
$ l$ & Euclidean & $a(l)$ & $q(l)$ & $r(l)$ & $k(l)$ & $\vec\nu(l)$ \\
\hline\hline
0 & 0 $\times$ 2+0+0 & 1 & 0 & 0 & 0 & (0,0) \\
\hline
1 & 0 $\times$ 2+1+0 & 2 & 0 & 0 & 0 & (1,0) \\
\hline
2 & 1 $\times$ 2+0+0 & 1 & 1 & 0 & 1 & (1,1) \\
\hline
3 & 1 $\times$ 2+1+0 & 2 & 1 & 0 & 1 & (2,1) \\
\hline
4 & 2 $\times$ 2+0+0 & 1 & 2 & 0 & 2 & (2,2) \\
\hline
5 & 2$\times$ 2+1+0 & 2 & 2 & 0 & 2 & (3,2) \\
\hline
6 & 3 $\times$ 2+0+0 & 1 & 3 & 0 & 3 & (3,3) \\
\hline
7 & 3 $\times$ 2+1+0 & 2 & 3 & 0 & 3 & (4,3) \\
\hline
8 & 4 $\times$ 2+0+0 & 1 & 4 & 0 & 4 & (4,4) \\
\hline
9 & 4 $\times$ 2+1+1 & 2 & 4 & 0 & 4 & (5,4) \\
\hline
10 & 5 $\times$ 2+0+0 & 1 & 5 & 0 & 5 & (5,5) \\
\hline
11 & 5 $\times$ 2+1+0 & 2 & 5 & 0 & 5 & (6,5) \\
\hline
12 & 6$\times$ 2+0+0 & 1 & 6 & 0 & 6 & (6,6) \\
\hline
\end{tabular}
\begin{tabular}{|c|c|c|c|c|c|c|}
\hline
\multicolumn{7}{|c|}{$\vec n=(3,2)$ } \\
\hline
$ l$ & Euclidean & $a(l)$ & $q(l)$ & $r(l)$ & $k(l)$ & $\vec\nu(l)$ \\
\hline\hline
0 & 0 $\times$ 5+0+0 & 1 & 0 & 0 & 0 & (0,0) \\
\hline
1 & 0 $\times$ 5+0+1 & 1 & 0 & 1 & 1 & (1,0) \\
\hline
2 & 0 $\times$ 5+0+2 & 1 & 0 & 2 & 2 & (2,0) \\
\hline
3 & 0 $\times$ 5+3+0 & 2 & 0 & 0 & 0 & (3,0) \\
\hline
4 & 0 $\times$ 5+3+1 & 2 & 0 & 1 & 1 & (3,1) \\
\hline
5 & 1 $\times$ 5+0+0 & 1 & 1 & 0 & 3 & (3,2) \\
\hline
6 & 1 $\times$ 5+0+1 & 1 & 1 & 1 & 4 & (4,2) \\
\hline
7 & 1 $\times$ 5+0+2 & 1 & 1 & 2 & 5 & (5,2) \\
\hline
8 & 1 $\times$ 5+3+0 & 2 & 1 & 0 & 2 & (6,2) \\
\hline
9 & 1 $\times$ 5+3+1 & 2 & 1 & 1 & 3 & (6,3) \\
\hline
10 & 2 $\times$ 5+0+0 & 1 & 2 & 0 & 6 & (6,4) \\
\hline
11 & 2 $\times$ 5+0+1 & 1 & 2 & 1 & 7 & (7,4) \\
\hline
12 & 2 $\times$ 5+0+2 & 1 & 2 & 2 & 8 & (8,4) \\
\hline
\end{tabular}

\
\begin{center}
\begin{tikzpicture}
\begin{axis}[legend style={at={(1,1)},anchor=north west},
enlargelimits,
axis lines = center,
xmin=-0.2,xmax=9,
ymin=-0.4,ymax=6,
xlabel = $\nu_1$,
disabledatascaling, axis equal,
ylabel = {$\nu_2$},
grid style={line width=1pt, draw=Bittersweet!10},
major grid style={line width=.2pt,draw=Bittersweet!50},
minor tick num=1,grid=both,
axis line style={latex'-latex'},Bittersweet]
]
\draw[BrickRed,ultra thick,fill=BrickRed] (0,0)--(1,0) circle [radius=1.5pt] --(1,1) circle [radius=1.5pt] --(2,1) circle [radius=1.5pt] --(2,2) circle [radius=1.5pt] --(3,2) circle [radius=1.5pt] --(3,3) circle [radius=1.5pt] --(4,3) circle [radius=1.5pt] --(4,4) circle [radius=1.5pt] --(5,4) circle [radius=1.5pt] --(5,5) circle [radius=1.5pt] --(6,5) circle [radius=1.5pt] 
--(6,6) circle [radius=1.5pt] --(7,6) circle [radius=1.5pt]--(7,7) circle [radius=1.5pt] ;
\addlegendimage{ultra thick,color=BrickRed}
\addlegendentry{$\vec n=(1,1)$}
\draw[NavyBlue,ultra thick,fill=NavyBlue,opacity=0.3] (0,0)--(1,0) circle [radius=1.5pt];
\draw[NavyBlue,ultra thick,fill=NavyBlue] (1,0)--(2,0)circle [radius=1.5pt] --(3,0)circle [radius=1.5pt] --(3,1)circle [radius=1.5pt] --(3,2)circle [radius=1.5pt] --(4,2)circle [radius=1.5pt] --(5,2)circle [radius=1.5pt] --(6,2)circle [radius=1.5pt] --(6,3)circle [radius=1.5pt] --(6,4)circle [radius=1.5pt] --(7,4)circle [radius=1.5pt] --(8,4)circle [radius=1.5pt] --(9,4)circle [radius=1.5pt] ;
\addlegendimage{ultra thick,color=NavyBlue}
\addlegendentry{$\vec n=(3,2)$}
\end{axis}
\draw (3.5,-0.4) node
{\begin{minipage}{0.4\textwidth}
\begin{center}\small
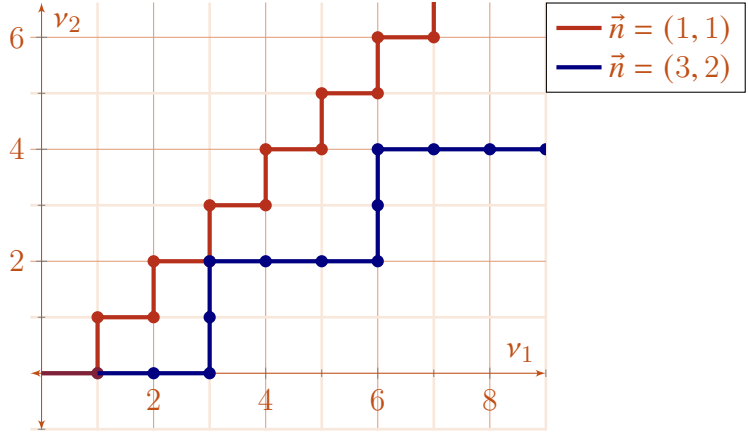
\textbf{Degree laders}, with stepline in red
\end{center}
\end{minipage}};
\end{tikzpicture}
\end{center}
Thus, the connection $\vec\nu(l)$ can be understood as a ladder, with dimensions in each direction given~by the components of $\vec n$. For $p=2$, $n_1$ is the dimension of each step and $n_2$ the height of the steps.

\begin{pro}
The polynomials $B^{(l)}(x)$ are type~II multiple orthogonal polynomials with multi-index $\vec\nu(l)$
\begin{align*}
B^{(l)}(x)&=B_{\vec\nu(l)}(x).
\end{align*}
Moreover, the set $\big\{A^{(l)}_a\big\}_{a=1}^p$ are type~I multiple orthogonal polynomials with multi-index $\vec\nu(l)$
\begin{align*}
A^{(l)}_a&=A_{\vec\nu(l+1),a}.
\end{align*}
These type~I multiple orthogonal polynomials satisfy the normalization
\begin{align*}
\int _\Delta x^{\nu_{a(l+1)}} \sum_{b=1}^{p} A_{\vec \nu(l+1),b}(x)w_{b}(x) \d\mu (x)=1.
\end{align*}
\end{pro}

\begin{proof}
From the Gauss--Borel factorization~\eqref{facto} we get that $Sg=H\tilde S^{-\top}$ is an upper triangular semi-infinite matrix, i.e.
\begin{align}\label{orth-2-1}
\sum_{i=0}^{l} S_{l,i} \, g_{i,j}&=0,& j&\in\{0,1,\ldots,l-1\},
\end{align}
With the aid of~\eqref{explicit g2}
we express~\eqref{orth-2-1} as follows
\begin{align*}
\int_\Delta \Big(\sum_{i=0}^l S_{l,i}x^{i}\Big)w_{a ( { j } )}(x)x^{k ( { j } )}\d \mu(x)&=0,& j&\in\{0,1,\ldots,l-1\}.
\end{align*}
Consequently we have
\begin{align*}
\int_\Delta B^{(l)}(x)w_{a ( { j } )}(x)x^{k ( { j } )}\d \mu(x)&=0,& j&\in\{0,1,\ldots,l-1\},
\end{align*}
and we need to check how the $l$ orthogonal relations are distributed for each weight $w_a(x)$.

Recalling~\eqref{i} for $i=l$,
\begin{align*}
l&=|\vec n|q(l)+n_1+\cdots+n_{a(l)-1}+r(l), & 0\leq r(l)<n_{a(l)},
\end{align*}
we see that these $l$ orthogonal relations are distributed as follows
\begin{align*}
\int_\Delta B^{(l)}(x)w_{b}(x)x^{k}\d \mu(x)&=0,& b&\in\{1,\ldots,a(l)-1\}
, & 0\leq k\leq(q(l)+1) n_b-1,\\
\int_\Delta B^{(l)}(x)w_{b}(x)x^{k}\d \mu(x)&=0,& b&=a(l)
, & 0\leq k\leq q(l)n_b+r(l)-1,\\
\int_\Delta B^{(l)}(x)w_{b}(x)x^{k}\d \mu(x)&=0,& b&\in\{a(l)+1,\ldots, p\}
, & 0\leq k\leq q(l)n_b-1,
\end{align*}
where we have subtracted $1$ on each of the three upper bounds for the degrees of the monomials~$x^k$, because we are starting at $k=0$
so that, using and~\eqref{defmops} and~\eqref{mult:u}, we get
\begin{align}\label{orthogonality}
\int _\Delta B^{(l)}(x)w_{a}(x)x^{k}\d \mu(x)&=0, & 0\leq k&\leq
\nu_{a}(l)-1, & a&\in\{1,\ldots,p\}.
\end{align}
We recognize these equations as those defining multiple orthogonal polynomials of type~II.
Moreover, from the normalization condition $ S_{ii}=1$ we get that the
polynomial $B_{\vec\nu}$ is monic with $\deg
B_{\vec\nu}=l=|\vec\nu|$.

From the Gauss--Borel factorization~\eqref{facto} we deduce that $g \, \tilde S^\top H=S$
is a lower unitriangular semi-infinite matrix. Consequently
\begin{align}\label{orth-1}
\sum_{j=0}^{l}g_{i,j}\frac{\tilde S_{l,j}}{H_l}&=0,& i&\in\{0,1,\ldots,l-1\},\\
\sum_{j=0}^{l}g_{l,j}\frac{\tilde S_{l,j}}{H_l}&=1.
\end{align}
Using~\eqref{mult:u} and~\eqref{defmops}, these equations read
\begin{align*}
\int_\Delta \sum_{a=1}^{p} \Big(\frac{1}{H_l}\sum_{\substack{j=0,\ldots,l \\ a ( { j } )=a}}\tilde S_{l,j} x^{k ( { j } )}\Big)
w_{a}(x)x^{i}\d \mu(x)&=0,& i&\in\{0,1,\ldots,l-1\},\\
\int_\Delta \sum_{a=1}^{p} \Big(\frac{1}{H_l}\sum_{\substack{j=0,\ldots,l \\ a ( { j } )=a}}\tilde S_{l,j} x^{k ( { j } )}\Big)w_{a}(x) x^{l}\d \mu(x)&=1, &
\end{align*}
and from~\eqref{defmops} we conclude that
\begin{align*}
\int_\Delta \sum_{a=1}^{p} A^{(l)}_{a}
(x)w_{a}(x)x^{i}\d \mu(x)&=0, &i&\in\{0,1,\ldots,l-1\},\\
\int_\Delta \sum_{a=1}^{p} A^{(l)}_{a}
(x)w_{a}(x)x^{l}\d \mu(x)&=1.
\end{align*}
The degrees of the polynomials $A^{(l)}_a(x)=\frac{1}{H_l}\sum_{\substack{j=0,\ldots,l\\ a ( { j } )=a}}\tilde S_{l,j} x^{k ( { j } )}$ are found by distributing the integer $l+1$ (as we are summing from $j=0$ up to $j=l$, we have $l+1$ terms).
Recalling the generalized Euclidean division~\eqref{i} for $i=l+1$,
\begin{align*}
l+1&=|\vec n|q(l+1)+n_1+\cdots+n_{a(l+1)-1}+r(l+1), & 0\leq r(l+1)<n_{a(l+1)},
\end{align*}
we see that
$\deg A^{(l)}_{a}\leq\nu_{a}(l+1)-1$, where we subtract a unity because the first monomial in the polynomials has degree zero.
Consequently, as $|\vec \nu(l+1)|=l+1$, we have $l-1=|\vec \nu(l+1)|-2$, and the orthogonalities can be written as
\begin{align*}
\int_\Delta \sum_{a=1}^{p} A^{(l)}_{a}
(x)w_{a}(x)x^{i}\d \mu(x)&=0, & \deg A^{(l)}_{a}&\leq\nu_{a}(l+1)-1 , & 0\leq i \leq |\vec \nu(l+1)|-2, \\
\int_{\Delta} \sum_{a=1}^{p} A^{(l)}_{a}
(x)w_{a}(x)x^{l}\d \mu(x)&=1.
\end{align*}
Hence, we are dealing with type~I multiple orthogonal polynomials which now happens to have a normalization of type~I.
\end{proof}

\subsection{Multiple biorthogonality}
Here we follow~\cite{afm}.
We introduce linear forms of type~I associated with multiple orthogonal polynomials as follows.
\begin{defi}
\label{def:linear forms of type~I }
Vector of type~II multiple orthogonal polynomials and type~I linear forms associated are
defined by
\begin{align}\label{linear forms of type~I S}
B:= \begin{pmatrix}
B^{(0)}\\
B^{(1)}\\
\vdots
\end{pmatrix}&= S \, \rchi,&
 Q:=\begin{pmatrix}
Q^{(0)}\\
Q^{(1)}\\
\vdots
\end{pmatrix}&=H^{-1} \tilde S \, \xi_{\vec n},&
\end{align}
\end{defi}
It can be immediately checked that
\begin{align}\label{linear.forms}
Q^{(l)}(x)&:= \sum_{a=1}^{p} A^{(l)}_{a}(x)w_{a}(x),
\end{align}

Moreover, we have that these linear forms of type~I are biorthogonal.

\begin{pro}[Biorthogonality]\label{proposition: biorthogonality}
The following multiple biorthogonality relations
\begin{align}\label{biotrhoganility}
\int_\Delta B^{(l)}(x) Q^{(k)}(x)\d \mu(x)&=\delta_{l,k},& l,k\in \N_0,
\end{align}
hold.
\end{pro}
\begin{proof}
Observe that
\begin{align*}
\int_\Delta B(x) (Q(x))^\top \d \mu(x)&=\int_\Delta S \, \rchi(x)(\xi_{\vec n}(x))^\top
\tilde S^{\top}H^{-1}\d \mu(x)&&\text{from~\eqref{linear forms of type~I S}} \\
&=S\Big(\int_\Delta \, \rchi(x)(\xi_{\vec n}(x))^\top\d \mu(x)\Big)\tilde S^{\top}H^{-1} \\
&=S \, g_{\vec n} \, \tilde S^{-1}H^{-1}&&\text{from~\eqref{compact.g}} \\
&=I , &&\text{from~\eqref{facto}}
\end{align*}
and so we get the desired result.
\end{proof}

 \subsection{Recursion relations and Jacobi matrix}\label{sec:recursion}
 Here we follow~\cite{afm}.
 Let us introduce some notation. As is of standard use, we have the unity matrix ${ \I } = \sum_{k=0}^\infty e_k e_k^\top$
 and the shift matrix $\Lambda:=\sum_{k=0}^\infty e_ke_{k+1}^\top$. We also define the projections $ \Pi^{[l]}:=\sum_{k=0}^{l-1} e_k e_k^\top$,
 and with the help of the set $\{e_a(k)\}_{\substack{a=1,\ldots,p\\ k=0,1,\ldots}}$, the projections
 $ \Pi_a:=\sum_{k=0}^{\infty}e_a(k) e_a(k)^\top$ with $\sum_{a=1}^p\Pi_a= { \I } $,

The moment matrix has a \emph{multi-Hankel symmetry} that implies the recursion relations (and, consequently, the Christoffel--Darboux formula).
We consider the shift operators defined by
\begin{align}
\Lambda_a & := \sum_{k=0}^\infty e_a(k)e_a(k+1)^\top .
\end{align}
Notice that
\begin{itemize}
\item $\Lambda_a$ leaves invariant the subspaces $\Pi_{a'}\R^\infty$, for $a'=1,\ldots,p$, and
$\Pi_{a'}\Lambda_a=\Lambda_a\Pi_{a'}$.
\item
$ \Lambda_a^{n_a}=\Pi_a\Lambda^{|\vec n|}$ and $\sum_{a=1}^p\Lambda_a^{n_a}=\Lambda^{|\vec n|}$.
\item The nonzero elements of $\Lambda_{a}$, which are 1, are distributed along the first and the $|\vec
n|-n_a+1$ diagonals over the main
diagonal.
\item The set of semi-infinite matrices $\{\Lambda_a^j\}_{\substack{a=1,\ldots,p\\j=1,2,\ldots}}$ is commutative.
\item We have the eigenvalue property
\begin{align}
\Lambda_a \, \rchi_{a'}= z \, \delta_{a,a'} \, \rchi_{a}.
\end{align}
\end{itemize}
Let us define the following \emph{multiple shift matrix}
\begin{align}
\Upsilon&:=\sum_{a=1}^{p}\Lambda_{a},
\end{align}
and introduce the integers
\begin{align*}
N_{a}&:=|\vec n|-n_{a}+1=\sum_{\substack{a'=1,\ldots,p\\a'\neq a}}n_{a'}+1,& a&=1,\ldots,p,& N&:=\max_{a=1,\ldots,p}N_{a}.
\end{align*}
\begin{rem}
Notice that, as $n_{1}\geq\cdots\geq n_{p}>0$, we obtain
 \begin{align*}
1&<N_{1}\leq\cdots\leq N_{p}.
 \end{align*}
\end{rem}
A direct computation shows
\begin{align*}
\Upsilon&=D_{0}\Lambda+D_{1}\Lambda^{N_{1}}+\cdots+D_{p}\Lambda^{N_{p}},
\end{align*}
where $D_{a}$, $a\in\{1, \ldots ,p\}$ are the following semi-infinite diagonal matrices:
\begin{align*}
D_{a}&=\diag
\begin{pNiceMatrix}
D_{a}(0),D_{a}(1),\ldots
\end{pNiceMatrix}
 ,\\
D_{a}(n)&:=\begin{cases}
1,& n=k|\vec n|+\sum\limits_{b=1}^a n_{b}-1,\quad k\in\N_0,\\
0,& n\neq k|\vec n|+\sum\limits_{b=1}^a n_{b}-1,\quad k\in\N_0,
\end{cases}
\end{align*}
and $D_{0}=\I-\sum\limits_{a=1}^{p}D_{a}$.

\begin{pro}[Multi-Hankel symmetry]
The moment matrix $g$ satisfies
\begin{align}\label{sym2}
\Lambda \, g =
g \Upsilon^\top.
\end{align}
\end{pro}
\begin{proof}
With the use of~\eqref{formulitas} and~\eqref{explicit g2} we get
\begin{align}\label{sym}
\Lambda \, g \, \Pi_{a}=
g\Lambda_a^\top,
\end{align}
and summing up in $b=1,\ldots,p$ we get the desired result.
\end{proof}

\begin{pro}\label{hess}
From the symmetry of the moment matrix one derives
\begin{align}\label{bigraded}
S \Lambda S^{-1}H=H\tilde S \Upsilon^\top\tilde S^{-1}.
\end{align}
\end{pro}
\begin{proof}
If we introduce~\eqref{facto} into~\eqref{sym2} we get
\begin{align*}
 \Lambda S^{-1}H \tilde S^{-\top}=S^{-1}H \tilde S^{-\top}\Upsilon^\top \Rightarrow S\Lambda S^{-1}H=H\big(\tilde S\Upsilon\tilde S^{-1}\big)^\top ,
\end{align*}
and the result follows.
\end{proof}

Thus, we define Jacobi matrix as
 \begin{align*}
 { J } := S\Lambda S^{-1}=\big(H^{-1}\tilde S \Upsilon\tilde S^{-1}H\big)^\top.
\end{align*}
This is a band matrix with all its superdiagonals equal to zero but for the first superdiagonal with coefficients equal to $1$, and only the first $N$ subdiagonals have possibly non zero coefficients, and the rest of subdiagonals have zero coefficients. Hence, this Jacobi matrix has the following aspect
\begin{align}\label{eq:Jacobi}
	{ J } =\left(\begin{NiceMatrix}[columns-width = .7cm]
		J_{0,0} &1&0&\Cdots\\
		\Vdots &\Ddots&\Ddots&\Ddots\\
		J_{N,0}&&&\\
		0&J_{N+1,1} &\Cdots&J_{N+1,N+1} &1\\
		\Vdots&\Ddots&\Ddots&&\Ddots&\Ddots
	\end{NiceMatrix}\right).
\end{align}

\begin{pro}[Linear recurrence  relations]
The Jacobi matrix, the type~II multiple orthogonal polynomials, and the corresponding type~I multiple orthogonal polynomials and linear forms of type~I fulfill the following eigenvalue properties
\begin{align}\label{eq:eigen_value}
 { J } B&= x B, & { J } ^\top A_a&= x A_a, & { J } ^\top Q&= xQ.
\end{align}
These eigenvalue properties, when expanded  component-wise gives corresponding   homogeneous recurrence relations of order $(N+1)$.
\end{pro}

\subsection{Christoffel--Darboux formula}

The Christoffel--Darboux (CD) kernel
is given by
\begin{align*}
K^{(n)}(x,y):=\sum_{m=0}^{n-1}B^{(m)}(y)Q^{(m)}(x),
\end{align*}
and as we know is a projection operator, fulfills a reproducing property and Christoffel--Darboux formulas.

\begin{pro}[Christoffel--Darboux formula on the sequence]\label{pro:CD}
For $n\geq N$, the following~CD formula is satisfied
 \begin{multline*}
(y-x)K^{(n)}(x,y)=Q^{(n-1)}(x)B^{(n)}(y)
-\begin{pNiceMatrix}
Q^{(n)}(x) & \Cdots & Q^{(n+N-1)}(x)
\end{pNiceMatrix}
 \\
\times
\begin{pNiceMatrix}[columns-width = 1.2cm]
J_{n,n-N}&\Cdots &&&J_{n,n-1}\\
0&J_{n+1,n-N+1}&\Cdots&&J_{n+1,n-1}\\
\Vdots&\Ddots&\Ddots&&\Vdots\\
&&&&\\
0&\Cdots&&0&J_{n+N-1,n-1}
\end{pNiceMatrix}
\begin{pNiceMatrix}
B^{(n-N)}(y)\\
\Vdots\\
B^{(n-1)}(y)
\end{pNiceMatrix}.
\end{multline*}
\end{pro}
For example, for $N=2$, the CD formula reads as follows
\begin{multline}\label{eq:CD_JP}
(\mu-\lambda)K^{(n)}(\lambda,\mu)=Q^{(n-1)}(\lambda)B^{(n)}(\mu)
 \\
-Q^{(n)}(\lambda)\big(J_{n,n-2}B^{(n-2)}(\mu) + J_{n,n-1}B^{(n-1)}(\mu))
-Q^{(n+1)}(\lambda)J_{n+1,n-1}B^{(n-1)}(\mu),
\end{multline}
Regularity in the previous CD formulas leads to the following result.
\begin{pro}
The following relations hold
\begin{multline}\label{eq:CD_regularity}
Q^{(n-1)}(x)B^{(n)}(x) =
\begin{pNiceMatrix}
Q^{(n)}(x) & \Cdots & Q^{(n+N-1)}(x)
\end{pNiceMatrix}\\
\times \begin{pNiceMatrix}[columns-width = 1.2cm]
J_{n,n-N}&\Cdots&&&J_{n,n-1}\\
0&J_{n+1,n-N+1}&\Cdots&&J_{n+1,n-1}\\
\Vdots&\Ddots&\Ddots&&\Vdots\\
&&&&\\
0&\Cdots&&0&J_{n+N-1,n-1}
\end{pNiceMatrix}
\begin{pNiceMatrix}
B^{(n-N)}(x)\\
\Vdots\\
B^{(n-1)}(x)
\end{pNiceMatrix}.
\end{multline}
\end{pro}

By taking limits on the CD formula and using~\eqref{eq:CD_regularity} we get confluent CD formulas for the diagonal CD kernels
\begin{pro}[Confluent Christoffel--Darboux formula on the sequence]\label{pro:dCD}
For $n\geq N$, the following confluent CD formula is satisfied
\begin{multline*}
K^{(n)}(x,x)={Q^{(n-1)}}'(x)B^{(n)}(x) 
-\begin{pNiceMatrix}
{Q^{(n)}}'(x) &{} \Cdots & {Q^{(n+N-1)}}'(x)
\end{pNiceMatrix}\\
\times \begin{pNiceMatrix}[columns-width = 1.2cm]
J_{n,n-N}&\Cdots&&&J_{n,n-1}\\
0&J_{n+1,n-N+1}&\Cdots&&J_{n+1,n-1}\\
\Vdots&\Ddots&\Ddots&&\Vdots\\
&&&&\\
0&\Cdots&&0&J_{n+N-1,n-1}
\end{pNiceMatrix}
\begin{pNiceMatrix}
B^{(n-N)}(x)\\
\Vdots\\
B^{(n-1)}(x)
\end{pNiceMatrix},
\end{multline*}
or
\begin{multline*}
K^{(n)}(x,x)=Q^{(n-1)}(x){B^{(n)}}'(x) 
-\begin{pNiceMatrix}
Q^{(n)}(x) &{} \Cdots & Q^{(n+N-1)}(x)
\end{pNiceMatrix}\\
\times \begin{pNiceMatrix}[columns-width = 1.2cm]
J_{n,n-N}&\Cdots&&&J_{n,n-1} \\
0&J_{n+1,n-N+1}&\Cdots&&J_{n+1,n-1}\\
\Vdots&\Ddots&\Ddots&&\Vdots\\
&&&&\\
0&\Cdots&&0&J_{n+N-1,n-1}
\end{pNiceMatrix}
\begin{pNiceMatrix}
{B^{(n-N)}}'(x)\\
\Vdots\\
B^{(n-1)}{}'(x)
\end{pNiceMatrix}.
\end{multline*}
\end{pro}

\section{Stochastic matrices, random walks and multiple orthogonal polynomials}\label{Section:Stochastic}

\subsection{Preliminaries}

The $\infty$-norm and $1$-norm of a semi-infinite matrix $M=(M_{i,j})_{i,j\in\N_0^2}$~are
\begin{align*}
\| M\|_\infty&=\sup_{i\in\N_0}\Big( \sum_{j=0}^{\infty} |M_{i,j}|\Big),&
\| M\|_1&=\sup_{j\in\N_0}\Big( \sum_{i=0}^{\infty} |M_{i,j}|\Big).
\end{align*}
Notice that the series above when $M$ is taken as  the Jacobi matrix $ { J } $, see \eqref{eq:Jacobi},  or its transpose~$ { J } ^\top$, reduces to a finite~sum. The semi-infinite matrix $M$ is said bounded if $\| M\|_\infty<\infty$. Moreover, $\|M\|_\infty=\|M^\top\|_1$.
Observe that~$P$, being stochastic, satisfies $\|P\|_\infty=1$, showing that the $\infty$-norm is particularly well fit to the notion of stochastic matrix.

For a bounded semi-infinite matrix $ { J } $, $\| { J } \|_\infty<\infty$, considered as an operator on a the  Banach space of complex sequences $\ell^p$, the resolvent set of $ { J } $ is
\begin{align*}
\rho( { J } )=\{z\in\C: \text {$ { J } -z \, { \I } $ has a bounded inverse} \},
\end{align*}
and the spectrum $\sigma( { J } )$ is the complement set of $\rho( { J } )$.
For $\lambda\in\rho( { J } )$ we have a well defined resolvent matrix $R( { J } ,z)=( { J } -z \, { \I } )^{-1}$.
The spectral radius is
\begin{align*}
r ( { J } )=\sup_{\lambda\in\sigma( { J } )}\{|\lambda|\},
\end{align*}
notice that $r ( { J } )\leq \| { J } \|_\infty$. The resolvent set $\rho( { J } ) $ is an open set and the resolvent matrix~$R(z)$, as a function of $z$, is holomorphic in the resolvent set with the following Neumann expansion 
\begin{align*}
R(z)=-\sum_{n=0}^\infty \frac{ { J } ^k }{z^{k+1}}.
\end{align*}
Moreover, the spectrum $\sigma( { J } )$ is a nonempty compact subset of
$\, \C$ and the Gel'fand spectral radius formula~holds
\begin{align*}
r ( { J } )=\lim_{n\to\infty}\sqrt[n]{\| { J } ^n \|}.
\end{align*}
for any norm. In particular, $r ( { J } )=\lim_{n\to\infty}\sqrt[n]{\| { J } ^n \|_1}=\lim_{n\to\infty}\sqrt[n]{\|( { J } ^\top )^n\|_\infty}=r( { J } ^\top)$.

\begin{defi}[Type~I and~II multiple stochastic matrices] 
\begin{enumerate}

\item
A semi-infinite stochastic matrix $P_{II}$ is said to be a multiple stochastic matrix of type~II if it has the form
\begin{align*}
	P_{II}=\left(\begin{NiceMatrix}[columns-width = 1.3cm]
		P_{0,0} &P_{0,1}&0&\Cdots\\
		\Vdots &\Ddots&\Ddots&\Ddots\\
		P_{N,0}&&&\\
		0&P_{N+1,1} &\Cdots&P_{N+1,N+1} &P_{N+1,N+2}\\
		\Vdots&\Ddots&\Ddots&&\Ddots&\Ddots
	\end{NiceMatrix}\right).
\end{align*}
\item A semi-infinite stochastic matrix $P_{I}$ is said to be a multiple stochastic matrix of type~I if it has the~form
\begin{align*}
	P_I=\left(\begin{NiceMatrix}[columns-width = 1.3cm]
		P_{0,0} &\Cdots&P_{0,N}&0&\Cdots\\
		P_{1,0} &\Ddots&&P_{1,N+1}&\Ddots\\
		0 &\Ddots&&\Vdots&\Ddots\\
		\Vdots&\Ddots&&P_{N+1,N+1}\\
		&&&P_{N+2,N+1}&\Ddots
		\\&&&&\Ddots
	\end{NiceMatrix}\right).
\end{align*}
\end{enumerate}
\end{defi}

\begin{rem}
The same definition holds for the semi-stochastic case.
\end{rem}

We now describe a method for deriving a stochastic matrix from a set of multiple biorthogonal polynomials. Given the corresponding Jacobi matrix $ { J } $ we will get a stochastic matrix $P$ using the linear recurrence relation in its eigenvalue form $ { J } B(x)=xB(x)$, and 
consequently using properties of the sequence of multiple orthogonal polynomials of type~II $\{B^{(l)}(x)\}_{l=0}^\infty$ .
These ideas extend to the transposed Jacobi matrix $ { J } ^\top$, and we will use the eigenvalue property $ { J } ^\top Q(x)=xQ(x)$ for the type~I linear form $Q(x)=\sum_{a=1}^pA_a(x)w_a(x)$.

\subsection{Type~II multiple stochastic matrices}
Let $Z(B^{(n)}):= \{\zeta_{n,k} \}_{k=1}^n\subset \C$ be the ordered set of zeros of the polynomial $B^{(n)}$, i.e. $B^{(n)}(\zeta_{n,k})=0$, with $|\zeta_{n,k}|\leq |\zeta_{n,k+1}|$, $k=1, \ldots, n$, and let $\mathscr Z(B):= \cup_{n=0}^\infty Z(B^{(n)})$ be the union of all these zero sets. Then, if
\begin{align*}
b(B):=\sup_{n\in\N_0}\big\{|\zeta_{n,k}|\big\}_{k=1}^n,
\end{align*}
the set $D (0,b(B))=\{z\in\C: |z|<b(B)\}$ is the smallest disk such that $\mathscr Z(B)\subset D(0,b(B))$.
\begin{lemma}
If the zeros of $B^{(n)}$ and $B^{(n+1)}$, $n\in\N_0$, interlace, then
\begin{align*}
b(B)=\lim_{n\to \infty} |\zeta_{n,n}|.
\end{align*}
Moreover, $b(B)\not\in\mathscr Z(B)$.
\end{lemma}

\begin{proof}
As the zeros interlace, the number $|\zeta_{n+1,n+1}|$ is greater that $|\zeta_{n,n}|$ and this leads to the result. If $b(B)$ is a zero of certain polynomial, say $B^{(n)}(x)$, then, given the interlacing property, it will be smaller than the greatest zero of the next polynomial $B^{(n+1)}(x)$. This is in contradiction with the fact that $b(B)$ is an upper bound for the zeros.
\end{proof}

\begin{lemma}

If the set of zeros $\mathscr Z(B)$ is bounded, i.e. $b(B)<+\infty$, and $b(B)\not\in\mathscr Z(B)$ we have $B^{(n)}(\lambda)>0$ for all $ \lambda\geq b(B)$.
\end{lemma}
\begin{proof}
As the zeros of $B^{(n)}(x)$ are smaller than $b$, the polynomial $B^{(n)}(x)$ does not change its sign in the set $(b(B),+\infty)$. Now, as the polynomials are monic we have $\lim\limits_{x\to+\infty}B^{(n)}(x)=+\infty$ and, consequently, $B^{(n)}(x)$ is a positive number if $x>b$.
\end{proof}
\begin{teo}[Multiple stochastic matrix of type~II]\label{pro:sigma_spectral}
Let us assume that
\begin{enumerate}
\item The Jacobi matrix $ { J } $ is nonnegative.
\item The set of zeros $\mathscr Z(B)$ is bounded, i.e. $b(B)<+\infty$, and $b(B)\not\in\mathscr Z(B)$.
\end{enumerate}
Then, for $\lambda\geq b(B)$ the diagonal matrix
\begin{align*}
\sigma_{II}&=
\diag
\begin{pNiceMatrix}
\sigma_{II,0},\sigma_{II,1},\ldots
\end{pNiceMatrix},
& \sigma_{II,n}:=\frac{1}{B^{(n)}(\lambda)},
\end{align*}
is such that
\begin{align}\label{eq:stochastic_II}
P_{II}&:=\frac{1}{\lambda}\sigma_{II} { J } \, \sigma_{II} ^{-1}, & P_{II,n,m}&=\frac{1}{\lambda}\frac{B^{(m)}(\lambda)}{B^{(n)}(\lambda)} J_{n,m},
\end{align}
is a multiple stochastic matrix of type~II.
\end{teo}
\begin{proof}
Notice that, by definition of $\sigma_{II}$, we have
\begin{align*}
\1= \sigma_{II} B(\lambda).
\end{align*}
As for $n\in\N_0 $ we have $\sigma_{II,n} >0$, the matrix $P$ is nonnegative and
\begin{align*}
P_{II} \, \1 = \frac{1}{\lambda}\sigma_{II} { J } \, \sigma_{II} ^{-1}\sigma_{II} B(\lambda)=\frac{1}{\lambda}\sigma_{II} { J } \, B(\lambda)=\sigma_{II} B(\lambda) = \1,
\end{align*}
as desired.
\end{proof}

\begin{rem}
An Angelescu system~\cite{angelesco} $(w_1\d \mu,\ldots , w_p\d\mu)$ is a set such that the convex hull of the support of
each measure $w_i\d\mu$ is a closed interval $[a_i,b_i]$ and all open intervals $(a_1,b_1), \ldots , (a_p,b_p)$ are disjoint, see~\cite{vanassche_ismail}.
If $\vec \nu$ is the index of a type~II multiple orthogonal polynomial $B_{\vec \nu}(x)$, then $B_{\vec \nu}(x)$ has $\nu_i$ zeros in $(a_i,b_i)$. Therefore, the zeros belong to $\cup_{i=0}^p(a_i,b_i)$, and if the corresponding Jacobi matrix is a non negative matrix with $b(B)=\max_{i=1, \ldots ,p} b_i<+\infty$, we can find a corresponding stochastic matrix as described above.
\end{rem}
\begin{rem}
 Another example are the algebraic Chebyshev (AT, the T is from the French transliteration Tchebycheff) 
 multiple orthogonal polynomials of type~II. Given a set of functions
$\{\varphi\}_{i=1}^s$ we say that it is a Chebyshev system in $[a,b]$ if the set is linearly independent and any linear combination $a_1\varphi_1+\cdots+a_s\varphi_s$ has at most $s-1$ zeros in $[a,b]$. The AT multiple orthogonal polynomials appear whenever $\{w_1,\ldots , x^{\nu_1-1}w_1, \ldots , w_p , \ldots , x^{\nu_p-1}w_p\}$ is a Chebyshev system in the support of~$\mu$, $[a,b]$ (cf.~\cite{nikishin_sorokin,vanassche_ismail}).

Suppose that $\vec \nu$ is such that for every $\vec \nu'$ with $\nu'_i\leq \nu_i$, $i\in\{1,\ldots p\}$ we are dealing with an AT-system. Then, $B_{\vec\nu}$ has $|\vec \nu|$ zeros on $(a,b)$. Thus, we can take $\lambda=b$ as $b$ is not a zero. Again, for such cases assuming that $b<+\infty$, all the zeros are confined and we may construct a stochastic matrix from a non negative Jacobi~matrix.
\end{rem}

\begin{rem}
A particular case of a non negative Jacobi matrix with a bounded set of zeros appear for Nikishin systems on star-like sets.
These systems have been studied for example in~\cite{abey,abey2}. These authors consider a set of multiple orthogonal polynomials $\{B_n\}_{n \in \N_0}$ associated with a~Nikishin system of $p$ measures supported on a star-like set of $p+1$ rays $S_+ = \{ z \in\C : z^{p+1}\in \R_+ \}$.
These polynomials satisfy a three-term recurrence relation of the form
\begin{align*}
zB_n=B_{n+1}(z)+a_n B_{n-p}(z) ,
\end{align*}
with $a_n>0$ for $n\geq p$.
Therefore, we have a Jacobi matrix of the form
\begin{align*}
	{ J } = \left(\begin{NiceMatrix}[columns-width = 0.1cm]
		0 &1&0&\Cdots\\
		\Vdots&\Ddots&\Ddots&\Ddots\\
		& & & & &\\
		0 & & & \\
		a_p&0&&&\\
		0&a_{p+1} &\\
		\Vdots&&\Ddots&\Ddots& \\
	\end{NiceMatrix}\right).
\end{align*}
\end{rem}

\subsection{Type~I multiple stochastic matrices}

Let $Z(Q^{(n)}):=\{\zeta_{n,k}\}_{k=1}^n\subset \C$ be the ordered set of zeros of the linear form $Q^{(n)}(x)$, $Q^{(n)}(\zeta_{n,k})=0$ and $\mathscr Z(Q):= \cup_{n=0}^\infty Z(Q^{(n)})$. Then, we denote the smallest upper bound of these zeros by
\begin{align*}
b(Q):=\sup_{n\in\N_0}\big\{|\zeta_{n,k}|\big\}_{k=1}^n.
\end{align*}
\begin{lemma}
Let us assume that
\begin{enumerate}
\item The set of zeros $\mathscr Z(Q)$ is bounded, i.e. $b(Q)<+\infty$ and $b(Q)\not\in\mathscr{Z}(Q)$.
\item $\lim\limits_{x\to+\infty}Q^{(n)}(x)=+\infty$
\end{enumerate}
Then, $Q^{(n)}(\lambda)>0$ for all $ \lambda\geq b(Q)$.
\end{lemma}
\begin{proof}
As the zeros of $Q^{(n)}(x)$ are smaller than $b$, the linear form $Q^{(n)}(x)$ does not change its sign in the set $(b(Q),+\infty)$. Now, as we suppose that $\lim\limits_{x\to+\infty}Q^{(n)}(x)=+\infty$ we have that~$Q^{(n)}(x)$ is positive if $x>b(Q)$.
\end{proof}

\begin{teo}[Multiple stochastic matrix of type~I]\label{pro:sigma_spectral_I}
Let us assume that:
\begin{enumerate}
\item The Jacobi matrix $ { J } $ is nonnegative.
\item The set of zeros $\mathscr Z(Q)$ is bounded, i.e. $b(Q)<+\infty$ and $b(Q)\not\in\mathscr{Z}(Q)$.
\item $\lim\limits_{x\to+\infty}Q^{(n)}(x)=+\infty$.
\end{enumerate}
Then, for $\lambda\geq \max (b(Q),0)$, the diagonal matrix
\begin{align*}
\sigma_I&=
\diag
\begin{pNiceMatrix}
\sigma_{I,0},\sigma_{I,1},\ldots
\end{pNiceMatrix},
& \sigma_{I,n}:=\frac{1}{Q^{(n)}(\lambda)}
\end{align*}
is such that
\begin{align*}
P_{I}&:=\frac{1}{\lambda}\sigma_I { J } ^\top\sigma_I^{-1}, & P_{I,n,m}=\frac{1}{\lambda}\frac{Q^{(m)}(\lambda)}{Q^{(n)}(\lambda)}J_{m,n}
\end{align*}
is a multiple stochastic matrix of type~I.
\end{teo}
\begin{proof}
Notice that, by definition of $\sigma_I$, we have
\begin{align*}
\1=\sigma_I Q(\lambda).
\end{align*}
As for $n\in\N_0 $ we have $\sigma_{I,n} >0$, the matrix $P$ is nonnegative and
\begin{align*}
P _I \, \1
=\frac{1}{\lambda}\sigma_I { J } ^\top\sigma_I^{-1}\sigma_I Q(\lambda)
=\frac{1}{\lambda}\sigma_I { J } \, Q(\lambda)=\sigma_I Q(\lambda) = \1 ,
\end{align*}
and the result follows.
\end{proof}

\begin{rem}
A family of examples satisfying ii) follows, (we use~\cite{nikishin_sorokin}). Let us assume that $\Delta\subset \R_+$ and that the systems of weights $\{w_a(x)\}_{a=1}^p$ is an AT-system in any closed interval in $\R_+$.\footnote{That is, an AT-system for any index $\vec\nu$, or equivalently
$\{ w_1(x) , \ldots , x^{\nu_1-1}w_1(x), \ldots ,w_p(x), \ldots $, $x^{\nu_p-1} w_p(x) \}$
is a T-system (Chebyshev system) of order $|\vec \nu|-1$, i.e., any linear combination of them vanishes at most at $|\vec \nu|-1$ points in that interval.}
Then, the linear form $Q_{\vec \nu}$ will have at most $|\vec \nu|-1$ zeros in $\R_+$. Moreover, being an AT-system on the support $\Delta$ we know that the linear form $Q_{\vec \nu}(x)$ has $|\vec{\nu}|-1$ zeros in the interior of the support, i.e.~$\mathring \Delta$. Consequently, there are no zeros of $Q_{\vec \nu}(x)$ in $\R_+\setminus \mathring\Delta$.
Thus, if $\Delta=[a,b]$ we can take $\lambda\geq b$.
We~will see that the Jacobi--Piñeiro multiple orthogonal polynomials fulfill all these conditions.
\end{rem}

As we have seen, we have two stochastic matrices $P_I$ and $P_{II}$, we say that they are dual stochastic matrices.
The corresponding Markov chains are said to be dual.

\subsection{Ratio asymptotics and dual stochastic matrices}

Here we require of the Poincaré theory for the ratio asymptotics of homogeneous linear recurrences (cf.~\cite{poincare,perron2,montel,Elaydi,Mate_Nevai, vanassche0,vanassche}).

From~\eqref{eq:eigen_value} we get

\begin{lemma}
The multiple orthogonal polynomials of type~II satisfy the following order $N+1$ homogeneous linear recurrence relation at $\lambda$
\begin{align}\label{eq:linear_recurrence}
\hspace{-.25cm}
J_{N+n,n} B^{(n)}(\lambda)+J_{N+n,n+1} B^{(n+1)}(\lambda)+\cdots +J_{N+n,N+n}B^{(N+n)}(\lambda)+B^{(N+n+1)}(\lambda)=\lambda B^{(N+n)}(\lambda)
\end{align}
and the linear forms of type~I satisfy the following $N+1$ order homogeneous linear recurrence relation, dual to~\eqref{eq:linear_recurrence},
\begin{align}\label{eq:dual_linear_recurrence}
Q^{(n)}(\lambda)+Q^{(n+1)}(\lambda)J_{n+1,n+1} +\cdots +Q^{(N+1+n)}(\lambda)J_{N+1+n,n+1}=\lambda Q^{(n+1)}(\lambda).
\end{align}
\end{lemma}
\begin{lemma}
Let us assume that
\begin{align}\label{eq:asymptotics_Jacobi}
\lim_{n\to\infty}J_{N+n,n}&=c_N\neq 0, &\lim_{n\to\infty}J_{N-1+n,n}&=c_{N-1}, &\cdots&&
\lim_{n\to\infty}J_{n,n}&=c_0,
\end{align}
 define the characteristic polynomial of~\eqref{eq:linear_recurrence}
\begin{align}\label{eq:characteristic_polynomial}
\varphi(r):=c_N +c_{N-1} r+\cdots +(c_0-\lambda)r^N+r^{N+1},
\end{align}
and introduce for~\eqref{eq:dual_linear_recurrence} the reciprocal polynomial
\begin{align}\label{eq:dual_charasteristic_polynomial}
 \varphi^*(r):=r^{N+1}\varphi\Big(\frac{1}{r}\Big)=1+(c_0-\lambda)r+c_1 r^2+\cdots+c_Nr^{N+1},
\end{align}
of $\varphi(r)$.
Then, if the ratio asymptotics for the multiple orthogonal polynomials of type~II
\begin{align*}
\lim_{n\to \infty}\frac{B^{(n+1)}(\lambda)}{B^{(n)}(\lambda)}=r_{II}(\lambda),
\end{align*}
holds, necessarily $r_{II}(\lambda)$ is a zero of the characteristic polynomial~\eqref{eq:characteristic_polynomial}. Moreover, if the ratio asymptotics for the linear forms of type~I
\begin{align*}
\lim_{n\to \infty}\frac{Q^{(n+1)}(\lambda)}{Q^{(n)}(\lambda)}=\frac{1}{r_I(\lambda)},
\end{align*}
is satisfied, then $ \frac{1}{r_I(\lambda)}$ is a root of the reciprocal polynomial $\varphi^*(r)$, i.e., $r_I(\lambda)$ is a zero of $\varphi(r)$.
\end{lemma}
\begin{rem}
The reciprocal polynomial $\varphi^{**}$ of $\varphi^*$ is the characteristic polynomial $\varphi$.
\end{rem}
\begin{rem}
The assumption~\eqref{eq:asymptotics_Jacobi} implies that the Jacobi matrix splits into
\begin{align*}
J=J_\infty +\delta J,
 \end{align*}
where
\begin{align*}
	J_\infty:=\left(\begin{NiceMatrix}[columns-width = .7cm]
		c_{0} &1&0&\Cdots\\
		\Vdots &\Ddots&\Ddots&\Ddots\\
		c_{N}&&&\\
		0&c_{N} &&c_0 &1\\
		\Vdots&\Ddots&\Ddots&&\Ddots&\Ddots
	\end{NiceMatrix}\right)
\end{align*}
is a banded Toeplitz matrix and
\begin{align*}
	\delta J:=\left(\begin{NiceMatrix}[columns-width =1cm]
		\delta J_{0,0} &0&\Cdots&\\
		\Vdots &\Ddots&\Ddots&\\
		\delta J_{N,0}&&&\\
		0&\delta J_{N+1,1} &&\delta J_{N+1,N+1} &\\
		\Vdots&\Ddots&\Ddots&&\Ddots&\Ddots
	\end{NiceMatrix}\right),
\end{align*}
is, as $\lim_{n\to\infty}\delta J_{n,n-k}=0$, for $k\in\{-1,0,1,\ldots, N\}$, a compact matrix.
\end{rem}

\begin{teo}\label{teo:I_and_II}
\begin{enumerate}
\item Both dual stochastic matrices, $P_{II}$ and $P_I$ are connected by
\begin{align*}
P_I&= \sigma P_{II}^\top \sigma^{-1}, & \sigma=\sigma_I\sigma_{II}.
\end{align*}
That is, the stochastic matrices coefficients fulfill
\begin{align}\label{eq:PI_vs_PII_large_n}
P_{I,n,n-k}=\frac{B^{(n-k)}(\lambda)Q^{(n-k)}(\lambda)}{B^{(n)}(\lambda)Q^{(n)}(\lambda)}P_{II,n-k,n},
\end{align}
with $m=n-k$, $k\in\{-1,0,1, \ldots , N\}$, where $N$ is the number of nonzero subdiagonals of the Jacobi matrix $ { J } $.
\item Let us assume, for the coefficients of the Jacobi matrix~\eqref{eq:Jacobi}, the large $n$ behavior given in~\eqref{eq:asymptotics_Jacobi}
and let $\mathscr R_\lambda=\{r_1(\lambda),\ldots,r_{N+1}(\lambda)\}$ be the characteristic roots of the characteristic polynomial~\eqref{eq:characteristic_polynomial} that we assume to have distinct absolute values.
Then, there are two positive characteristic roots $r_{II}(\lambda),r_I(\lambda)\in\mathscr R_\lambda$ such that
\begin{align}\label{eq:PI_vs_PII_large_n_2}
P_{I,n,n-k}&=\frac{(r_{I}(\lambda))^k}{(r_{II}(\lambda))^k}P_{II,n-k,n}, &k\in\{-1,0,1, \ldots , N\} .
\end{align}
\end{enumerate}
\end{teo}
\begin{proof}
\begin{enumerate}
\item[]$\phantom{ola}$
\item We have $P_{II}=\frac{1}{\lambda}\sigma_{II} { J } \, \sigma_{II} ^{-1}$ so that $\lambda \sigma_{II} ^{-1}P_{II}\sigma_{II}= { J } $. Thus, $P_{I}=\frac{1}{\lambda}\sigma_{I} { J } ^\top\sigma_{I} ^{-1}=
\sigma_{I} \sigma_{II} P_{II}^\top\sigma_{II}^{-1}\sigma_{I} ^{-1}$.
\item On the one hand, notice that the multiple orthogonal polynomials of type~II satisfy the following order $N+1$ homogeneous linear recurrence relation~\eqref{eq:linear_recurrence}.
As we have that the absolute values of the characteristic roots are distinct, Poincaré's theorem~\cite{poincare} ensures that
\begin{align*}
\lim_{n\to \infty}\frac{B^{(n+1)}(\lambda)}{B^{(n)}(\lambda)}=r_{II}(\lambda),
\end{align*}
for some characteristic root $r_{II}(\lambda)$, that is positive because $B^{(n)}(1)>0$, for $n\in\N_0$. On the other hand, 
the linear forms of type~I satisfy the $N+1$ order homogeneous linear recurrence relation~\eqref{eq:dual_linear_recurrence}
with dual characteristic polynomial given in~\eqref{eq:dual_charasteristic_polynomial}
with dual characteristic roots $\big\{\frac{1}{r_1(\lambda)},\ldots,\frac{1}{r_{N+1}(\lambda)}\big\}$. Notice that, as $c_N\neq 0$, $0$ is not a characteristic root. Then, according to Poincaré's theorem there is a positive characteristic root $r_I(\lambda)\in\mathscr R_\lambda$ such that
\begin{align*}
\lim_{n\to \infty}\frac{Q^{(n+1)}(\lambda)}{Q^{(n)}(\lambda)}=\frac{1}{r_I(\lambda)}.
\end{align*}
Hence,
\begin{align*}
\lim_{n\to \infty} \frac{B^{(n+1)}(\lambda)Q^{(n+1)}(\lambda)}{B^{(n)}(\lambda)Q^{(n)}(\lambda)}=
\frac{r_{II}(\lambda)}{r_{I}(\lambda)},
\end{align*}
and
\begin{align*}
\lim_{n\to \infty} \frac{B^{(n-k)}(\lambda)Q^{(n-k)}(\lambda)}{B^{(n)}(\lambda)Q^{(n)}(\lambda)}&=
\frac{r_{I}(\lambda)^k}{r_{II}(\lambda)^k}, & k&\in\{1,2\}.
\end{align*}
Therefore, Equation~\eqref{eq:PI_vs_PII_large_n} implies Equation~\eqref{eq:PI_vs_PII_large_n_2}.

\end{enumerate}
Which finishes the proof.
\end{proof}

\begin{rem}
Poincaré's result~\cite{poincare} says that in 
{general} the ratio will converge to the root with largest absolute value, and \emph{exceptionally} to one with smaller absolute value. Thus, in \emph{general}, one could expect that $r_{II}(\lambda)$ is the largest positive characteristic root and that $r_{I}(\lambda)$ is the smallest positive characteristic root. So we could expect in \emph{general} that
\begin{align*}
\frac{r_{II}(\lambda)}{r_{I}(\lambda)}\geq 1.
\end{align*}
\end{rem}

For the next results we need of the complete homogeneous symmetric polynomials, see~\cite{matrix,matrix2} where we also used them in the context of matrix orthogonality,
\begin{align*}
h_n(x,y)&:=\sum\limits_{m=0}^n x^{n-m}y^m,& n&\in\N, & h_0&=1,
\end{align*}
so that
\begin{align}\label{eq:symmetric}
h_{n-1}(x,y)&:=\frac{x^{n}-y^{n}}{x-y}.
\end{align}

Now, we show that the linear forms of type~I do indeed satisfy linear recurrence relations of order~$N$ instead of order $N+1$ as in~\eqref{eq:dual_linear_recurrence}.
\begin{teo}\label{teo:removing_roots}
Let us assume~\eqref{eq:asymptotics_Jacobi}
as well as the following ratio asymptotics for the multiple orthogonal polynomials of type~II
\begin{align*}
\lim_{n\to\infty}\frac{B^{(n+1)}(\lambda)}{B^{(n)}(\lambda)}=r_{II}(\lambda),
\end{align*}
for some root $r_{II}(\lambda)$ of the characteristic polynomial~\eqref{eq:characteristic_polynomial}.
Let us write the set of characteristic roots of $\varphi$ in the following form $\mathscr R_\lambda=\{r_1(\lambda),r_2(\lambda),\ldots, r_N(\lambda),r_{II}(\lambda)\}$, where the roots \emph{are not assumed} to be distinct.
Then, the linear forms of type~I satisfy the following   order  $N$  homogeneous linear recurrence
\begin{multline}\label{eq:new_linear_recurrence}
\begin{pNiceMatrix}
Q^{(n)}(\lambda) {} & {} \Cdots {} & Q^{(n+N-1)}(\lambda)\end{pNiceMatrix}
 \\
\times
\begin{pNiceMatrix}[columns-width = 1.2cm]
J_{n,n-N}&\Cdots&&&J_{n,n-1}\\
0&J_{n+1,n-N+1}&\Cdots&&J_{n+1,n-1}\\
\Vdots&\Ddots&\Ddots&&\Vdots\\
&&&&\\
0&\Cdots&&0&J_{n+N-1,n-1}
\end{pNiceMatrix}
\begin{pNiceMatrix}
\frac{B^{(n-N)}(\lambda)}{B^{(n)}(\lambda)}\\
\Vdots\\
\frac{B^{(n-1)}(\lambda)}{B^{(n)}(\lambda)}
\end{pNiceMatrix}
=Q^{(n-1)}(\lambda),
\end{multline}
in where the ratios $\big\{\frac{B^{(n-k)}(\lambda)}{B^{(n)}(\lambda)}\big\}_{k=1}^N$ are understood as given.
The characteristic polynomial of this order~$N$ homogeneous linear recurrence is
\begin{align*}
 \Phi^*(r)=-\frac{\varphi^*(r)}{1-r_{II}(\lambda)r}=-(1-rr_1(\lambda))\cdots (1-rr_N(\lambda)).
\end{align*}
Thus, if the multiplicity of $r_{II}(\lambda)$ is $1$, $\frac{1}{r_{II}(\lambda)}$ is not among the characteristic roots of~\eqref{eq:new_linear_recurrence}, and if the~multiplicity~of $r_{II}(\lambda)$ is $m$, then $\frac{1}{r_{II}(\lambda)}$ is still a characteristic root of~\eqref{eq:new_linear_recurrence} but with multiplicity~$m-1$.
\end{teo}
\begin{proof}
The new lower order linear recurrence relation~\eqref{eq:new_linear_recurrence} follows from
Equation~\eqref{eq:CD_regularity}.
The characteristic polynomial can be written
\begin{align*}
 \Phi^*(r)& =-1+
\begin{pNiceMatrix}
r {} &{} \Cdots& {} r^{N}\end{pNiceMatrix}
\begin{pNiceMatrix}
c_{N}&\Cdots&&c_2&c_1\\
0&c_N&\Cdots&&c_2\\
\Vdots&\Ddots&\Ddots&&\Vdots\\
&&&&\\
0&\Cdots&&0&c_N
\end{pNiceMatrix}
\begin{pNiceMatrix}
\frac{1}{r_{II}^{N}(\lambda)}\\
\Vdots\\
\frac{1}{r_{II}(\lambda)}
\end{pNiceMatrix}\\
 &=
-1+\frac{r}{r_{II}(\lambda)}
\sum_{k=1}^Nc_kh_{k-1}\Big(r,\frac{1}{r_{II}(\lambda)}\Big),
\end{align*}
in terms of the homogeneous completely symmetric polynomials $h_{k-1}(x,y)$. From~\eqref{eq:symmetric} we know~that
\begin{align*}
\frac{r}{r_{II}(\lambda)}h_{k-1}\Big(r,\frac{1}{r_{II}(\lambda)}\Big)=\frac{r}{r_{II}(\lambda)}\frac{r^k-\frac{1}{(r_{II}(\lambda))^k}}{r-\frac{1}{r_{II}(\lambda)}}=\frac{r^k-\frac{1}{(r_{II}(\lambda))^k}}{r_{II}(\lambda)-\frac{1}{r}}.
\end{align*}
Consequenly, we find
\begin{align*}
\Phi^*(r)&=-1+\sum_{k=1}^N c_k \frac{r^k-\frac{1}{(r_{II}(\lambda))^k}}{r_{II}(\lambda)-\frac{1}{r}}=-1+\frac{\sum_{k=1}^N c_kr^k-\sum_{k=1}^N c_k\frac{1}{(r_{II}(\lambda))^k}}{r_{II}(\lambda)-\frac{1}{r}}\\
&=-1+\frac{\Big(\frac{1}{r}\varphi^*(r)-(c_0-\lambda)-\frac{1}{r}\Big)-\Big(\frac{1}{(r_{II}(\lambda))^N}\varphi(r_{II}(\lambda))-(c_0-\lambda)-r_{II}(\lambda)\Big)}{r_{II}(\lambda)-\frac{1}{r}}\\
&
=-1+\frac{\frac{1}{r}\varphi^*(r)+r_{II}(\lambda)-\frac{1}{r}}{r_{II}(\lambda)-\frac{1}{r}} \\
& =-\frac{\frac{1}{r}\varphi^*(r)}{\frac{1}{r}-r_{II}(\lambda)}=-r^N\Big(\frac{1}{r}-r_1(\lambda)\Big)\cdots\Big(\frac{1}{r}-r_N(\lambda)\Big), 
\end{align*}
and the result follows.
\end{proof}

From these ideas we also get

\begin{coro}\label{cor:removing_roots}
Let us assume~\eqref{eq:asymptotics_Jacobi}
as well as the following ratio asymptotics for the linear forms of type~I
\begin{align*}
\lim_{n\to\infty}\frac{Q^{(n+1)}(\lambda)}{Q^{(n)}(\lambda)}=\frac{1}{r_{I}(\lambda)},
\end{align*}
for some root $r_{I}(\lambda)$ of the characteristic polynomial~\eqref{eq:characteristic_polynomial}.
Let consider the set of characteristic roots of $\varphi$, 
$\mathscr R_\lambda=\{ r_{I}(\lambda),r_1(\lambda),r_2(\lambda),\ldots, r_N(\lambda)\}$,
where the roots \emph{are not assumed} to be distinct.
Then, the multiple orthogonal polynomials of type~II satisfy the following order~$N$  homogeneous linear recurrence relation
\begin{multline}\label{eq:new_linear_recurrence_2}
\begin{pNiceMatrix}
\frac{Q^{(n)}(\lambda)}{Q^{(n-1)}(\lambda)} {} & {} \Cdots {} & \frac{Q^{(n+N-1)}(\lambda)}{Q^{(n-1)}(\lambda)}\end{pNiceMatrix}
 \\
\times
\begin{pNiceMatrix}[columns-width = 1.2cm]
J_{n,n-N}&\Cdots&&&J_{n,n-1}\\
0&J_{n+1,n-N+1}&\Cdots&&J_{n+1,n-1}\\
\Vdots&\Ddots&\Ddots&&\Vdots\\
&&&&\\
0&\Cdots&&0&J_{n+N-1,n-1}
\end{pNiceMatrix}
\begin{pNiceMatrix}
B^{(n-N)}(\lambda)\\
\Vdots\\
B^{(n-1)}(\lambda)
\end{pNiceMatrix}
=B^{(n)}(\lambda),
\end{multline}
in where the ratios $\big\{\frac{Q^{(n+k)}(\lambda)}{Q^{(n-1)}(\lambda)}\big\}_{k=0}^{N-1}$ are understood as given.
The characteristic polynomial of this order~$N$  homogeneous linear recurrence relation is
\begin{align*}
\Phi(r)=-\frac{\varphi(r)}{r-r_{I}(\lambda)}=-(r-r_1(\lambda))\cdots (r-r_N(\lambda)).
\end{align*}
Thus, if the multiplicity of $r_{I}(\lambda)$ is $1$, $r_{I}(\lambda)$ is not among the characteristic roots of~\eqref{eq:new_linear_recurrence_2}, and if the multiplicity of $r_{I}(\lambda)$ is $m>1$, then $r_{I}(\lambda)$ is still a characteristic root of~\eqref{eq:new_linear_recurrence_2} but with multiplicity~$m-1$.
\end{coro}
\begin{proof}
Again, relation~\eqref{eq:new_linear_recurrence} follows immediately from~\eqref{eq:CD_regularity}.
The characteristic polynomial reads in this case
\begin{align*}
\Phi(r)&=-r^N+
\begin{pNiceMatrix}
\frac{1}{r_{I}(\lambda)}{} &{} \Cdots& {} \frac{1}{r_{I}^{N}(\lambda)}
\end{pNiceMatrix}
\begin{pNiceMatrix}[columns-width = 1.2cm]
c_{N}&\Cdots&&c_2&c_1\\
0&c_N&\Cdots&&c_2\\
\Vdots&\Ddots&\Ddots&&\Vdots\\
&&&&\\
0&\Cdots&&0&c_N
\end{pNiceMatrix}
\begin{pNiceMatrix}
1\\
\Vdots\\
r^{N-1}
\end{pNiceMatrix}
\\&=r^N\left(
-1+
\begin{pNiceMatrix}
\frac{1}{r_{I}(\lambda)}{} &{} \Cdots& {} \frac{1}{r_{I}^{N}(\lambda)}
\end{pNiceMatrix}
\begin{pNiceMatrix}[columns-width = 1.2cm]
c_{N}&\Cdots&&c_2&c_1\\
0&c_N&\Cdots&&c_2\\
\Vdots&\Ddots&\Ddots&&\Vdots\\
&&&&\\
0&\Cdots&&0&c_N
\end{pNiceMatrix}
\begin{pNiceMatrix}
\frac{1}{r^{N}}\\
\Vdots\\
\frac{1}{r}
\end{pNiceMatrix}
\right)\\&
=r^N\left(-1+\frac{1}{r_{I}(\lambda)r}
\sum_{k=1}^Nc_kh_{k-1}\Big(\frac{1}{r_{I}(\lambda)},\frac{1}{r}\Big)\right)
=r^N\left(-1+\frac{1}{r_{I}(\lambda)r}
\sum_{k=1}^Nc_k\frac{\frac{1}{r_{I}^k(\lambda)}-\frac{1}{r^k}}{\frac{1}{r_{I}(\lambda)}-\frac{1}{r}}\right) \\
&=r^N\left(-1+\frac{1}{r-r_{I}(\lambda)}
\sum_{k=1}^Nc_k\Big(\frac{1}{r_{I}^k(\lambda)}-\frac{1}{r^k}\Big)\right) \\
&=-r^N+\frac{\frac{r^N}{r_{I}^N(\lambda)}\sum_{k=1}^Nc_kr_{I}^{N-k}(\lambda)-\sum_{k=1}^Nc_kr^{N-k}}{r-r_{I}(\lambda)}\\
&=-r^N+\frac{\frac{r^N}{r_{I}^N(\lambda)}(\varphi(r_{I}(\lambda))-(c_0-\lambda)r_{I}(\lambda)^N-r_{I}(\lambda)^{N+1})-(\varphi(r)-(c_0-\lambda)r^N-r^{N+1})}{r-r_{I}(\lambda)}\\
&=-r^N+\frac{r^N(r_{I}^N(\lambda)-r)-\varphi(r)+r^{N+1}}{r-r_{I}(\lambda)}=-\frac{\varphi(r)}{r-r_{I}(\lambda)}.
\end{align*}
As we wanted to prove.
\end{proof}



\subsection{Steady states, positive and null recurrence}

We now present a candidate for steady state and some conjectures on its existence and the relation with mass points.

\begin{teo}[Steady states]\label{teo:steady}
	
	Let us assume conditions in Theorems \ref{pro:sigma_spectral} and \ref{pro:sigma_spectral_I}. Then, 
\begin{enumerate}
\item The row vector
\begin{align*}
{\boldsymbol{\kappa}}_\lambda=(B^{(0)}(\lambda)Q^{(0)}(\lambda),B^{(1)}(\lambda)Q^{(1)}(\lambda),\ldots)
\end{align*}
is a nonnegative vector, which is a left eigenvector of both dual multiple stochastic matrices~$P_{II}$ and~$P_I$ with unit eigenvalue
\begin{align*}
\boldsymbol{\kappa}_\lambda P_{II}&=\boldsymbol{\kappa}_\lambda, & \boldsymbol{\kappa}_\lambda P_{I}&=\boldsymbol{\kappa}_\lambda.
\end{align*}
\item If $\boldsymbol \kappa_\lambda\in\ell_1$, i.e.
\begin{align*}
\|\boldsymbol{\kappa}_\lambda\|_1=\sum_{k=0}^\infty B^{(k)}(\lambda)Q^{(k)}(\lambda)<\infty,
\end{align*}
then
\begin{align*}
\boldsymbol \pi_\lambda=\frac{\boldsymbol{\kappa}_\lambda}{\|\boldsymbol{\kappa}_\lambda\|_1},
\end{align*}
is a steady state for both dual Markov chains. In this situation the random walk is positive recurrent, and whenever the chain is aperiodic is also ergodic.
\item If
\begin{align*}
\lim_{n\to \infty} \frac{B^{(n+1)}(\lambda)Q^{(n+1)}(\lambda)}{B^{(n)}(\lambda)Q^{(n)}(\lambda)}<1
\end{align*}
we have $\boldsymbol \kappa_\lambda\in\ell_1$ and
\begin{align*}
{\boldsymbol {\pi}}_\lambda=\frac{\boldsymbol{\kappa}_\lambda}{\|\boldsymbol{\kappa}_\lambda\|_1}
\end{align*}
is a steady state.
\end{enumerate}
\end{teo}
\begin{proof}
\begin{enumerate}
\item We have the eigenvalue property of the Jacobi matrix
\begin{align*}
\lambda(Q(\lambda))^\top= (Q(\lambda))^\top J
\end{align*}
so that
\begin{align*}
(Q(\lambda))^\top\sigma_{II}^{-1}= (Q(\lambda))^\top \sigma_{II}^{-1}\frac{1}{\lambda} \sigma_{II}J\sigma_{II}^{-1}.
\end{align*}
Thus, as $(Q(\lambda))^\top\sigma_{II}^{-1}=\boldsymbol \kappa_\lambda$, we get one the assertions. For the other we recall
\begin{align*}
\lambda(B(\lambda))^\top= (B(\lambda))^\top J^\top
\end{align*}
so that
\begin{align*}
(B(\lambda))^\top\sigma_{I}^{-1}= (B(\lambda))^\top \sigma_{I}^{-1}\frac{1}{\lambda }\sigma_{I}J^\top\sigma_{I}^{-1}.
\end{align*}
and now, we notice $(B(\lambda))^\top\sigma_{I}^{-1}=\boldsymbol \kappa_\lambda$, and we get the other case.
\item
If $\boldsymbol \kappa_\lambda \in\ell_1$, as $\|\boldsymbol{\kappa}_\lambda \|_1= \boldsymbol{\kappa}_\lambda \, \1$, the row vector
$\boldsymbol \pi_\lambda=\frac{\boldsymbol{\kappa}_\lambda}{\|\boldsymbol{\kappa}_\lambda\|_1} $ is a probability vector, $\boldsymbol \pi_\lambda\, \1=1$. Therefore, we have found a steady state.
\end{enumerate}
Finally, applying d'Alembert's ratio test we see that iii) holds true.
\end{proof}

%

\begin{con}\label{conjecture:mass points and steady states}
The left eigenvector satisfies $\boldsymbol \kappa\in\ell_1$ if and only if $\lambda$ is a mass point of the measure, i.e. $\mu\big(\{\lambda\}\big)\neq 0$.
Hence, if the measure is absolutely continuous and has no singular part we have $\boldsymbol \kappa\not\in\ell_1$, and  the Markov chain if recurrent is null recurrent.
\end{con}
\paragraph{\emph{Motivation}}In terms of the diagonal CD kernel we have
\begin{align*}
\|\boldsymbol{\kappa}\|_1=\lim_{
n\to\infty}K^{(n)}(\lambda,\lambda),
\end{align*}
and for standard orthogonal polynomials (with only one component) we have Theorem~9.7 in~\cite{simon}, that ensures that given a measure $\d\mu$ on the real line with compact support, if $\mu(\{\lambda\})=0$, where $\mu(\{\lambda\}):=\lim\limits_{\delta \to 0}\mu\big((\lambda-\delta,\lambda+\delta)\big)$, then
\begin{align*}
\lim_{n\to\infty}K^{(n)}(\lambda,\lambda)=\frac{1}{\mu(\{\lambda\})}.
\end{align*}
Thus, this limit is finite only if $\lambda$ is a mass point, part of the support of the singular part of the measure. Otherwise,
\begin{align*}
\lim_{n\to\infty}K^{(n)}(\lambda,\lambda)=\infty.
\end{align*}
Thus, we conjecture that what is true for a birth and death Markov chain, that is the tridiagonal case, is also true in the multiple orthogonal scenario.

\begin{rem} Now we comment on ergodic states as discussed in~\cite{KmcG}. For this aim, let us put $\Delta=[0,1]$ and take $\lambda=1$.
The $\pi_n$ introduced in~\cite{KmcG} correspond to $B^{(n)}(1)Q^{(n)}(1)$ in the standard tridiagonal scenario, observe also that the corresponding vector $(\pi_0,\pi_1,\ldots)$ in~\cite{KmcG} is not a probability vector yet as it needs to be normalized. Recall also that for a birth and death Markov chain, i.e., a tridiagonal case, we have $Q^{(n)}(1)=\frac{B^{(n)}(1)}{H_n}$; so that, Karlin--McGregor's $\pi_n$ is $\frac{B^{(n)}(1)^2}{H_n}$ and
$1/\rho=\sum_{n=0}^\infty\pi_n=\lim_{n\to\infty}K^{(n)}(1,1)$,
being $\rho=C(1)$, the Christoffel function evaluated at~$1$.
Therefore, following~\cite{simon} for the standard situation described in~\cite{KmcG} the process is ergodic if and only if~$1$ is a mass point of the measure $\d\psi$, in the notation used in~\cite{KmcG}. This is what in~\cite{KmcG} is referred as a jump of $\psi$.
\end{rem}

%
%

 \begin{rem}
 Krein--Rutman~\cite{Krein_Rutman} seminal theorem on the existence of a positive eigenvector with eigenvalue given by the spectral radius for a compact positive operator, see~\cite{Karlin} for extensions to positive bounded operators, extended previous results in finite dimensions of Perron~\cite{perron1} and Frobenius~\cite{Frobenius,Frobenius2,Frobenius3}. In~\cite{Boukas_feinsilver_Fellouris} using Krein--Rutman theorem
if the stochastic matrix $P$ is compact and strongly positive, with $1$ as a simple eigenvalue, then the associated Markov chain is positive recurrent and the powers of $P$ converge entry-wise
 $\lim_{n\to\infty} P_{ij}^n = \pi_j$, with $\pi_j $ positive, summing to~$1$, providing the steady state of the chain. 
 \\
Moreover, many examples of positive compact Jacobi operators come from discrete orthogonal systems, where all the support points are mass points,~\cite{vanassche0}. Thus, this circle of ideas fit well with our conjectures regarding the existence of positive recurrent states and mass points of the measure.
 \end{rem}

\subsection{Karlin--McGregor representation formula}
We now extend the results of Karlin and McGregor concerning tridiagonal stochastic matrices~\cite{KmcG} to the multi-diagonal situation of multiple orthogonal polynomials.

We have seen in Theorems~\ref{pro:sigma_spectral} and~\ref{pro:sigma_spectral_I} that certain sets of multiple orthogonal polynomials give two families of stochastic matrices $P_{II}$
and $P_I$. Here $P_{II}$ models a random walk with allowed
jumps backward farther than near neighbors, and $P_{I}$ models a random walk with allowed
jumps forward farther than near neighbors.


\begin{teo}[KMcG representation formula]\label{teo:KMcG}
Let us assume the conditions requested in Theorem~\ref{pro:sigma_spectral}.
Then, for a random walk with Markov matrix $P_{II}$, the transition probability, after~$r$ transitions from state~$n$ to state $m$ is given~by
\begin{align}\label{eq:representation_II}
P_{II,nm}^r
=
\frac{
B_{\vec{\nu}(m)}(\lambda)
}{
B_{\vec{\nu}(n)}(\lambda)
}
\int_\Delta \frac{x^r}{\lambda^r}B_{\vec \nu(n)}(x) Q_{\vec \nu(m+1)}(x)\d \mu(x).
\end{align}
Let us assume the conditions requested in Theorem~\ref{pro:sigma_spectral_I}.
Then, for a random walk with probability matrix $P_{I}$, the transition probability, after $r$ transitions from state $n$ to state $m$ is given~by
\begin{align}\label{eq:representation_I}
P_{I,nm}^r=
\frac{
Q_{\vec \nu(m+1)}(\lambda)
}{
Q_{\vec \nu(n+1)}(\lambda)
}
\int_\Delta \frac{x^r }{\lambda^r }B_{\vec \nu(m)}(x) Q_{\vec \nu(n+1)}(x)\d \mu(x).
\end{align}
\end{teo}

\begin{proof} In terms of the Jacobi matrix $ { J } $ we have
\begin{align*}
P_{II,nm}^r= ( { J } ^r)_{n,m}\frac{B^{(m)}(\lambda)}{\lambda^rB^{(n)}(\lambda)}.
\end{align*}
But, $ { J } ^rB(x)=x^r B(x)$ so that $\sum_{m=0}^\infty( { J } ^r)_{n,m}B^{(m)}(x)=x^r B^{(n)}(x)$
and using biorthogonality~\eqref{biotrhoganility} we~get
\begin{align*}
( { J } ^r)_{n,m}=\int_\Delta x^r B^{(n)}(x) Q^{(m)}(x)\d \mu(x),
\end{align*}
and~\eqref{eq:representation_II} follows.

In terms of the transposed Jacobi matrix $ { J } ^\top$ we have
\begin{align*}
P_{I,nm}^r= (( { J } ^\top)^r)_{n,m}\frac{Q^{(m)}(\lambda)}{\lambda^rQ^{(n)}(\lambda)}
= ( { J } ^r)_{m,n}\frac{Q^{(m)}(\lambda)}{\lambda^rQ^{(n)}(\lambda)}
\end{align*}
and we obtain~\eqref{eq:representation_I}.
\end{proof}

\subsection{Generating functions and first passage distributions}
For this discussion see for instance~\cite{KmcG,feller}.
The generating functions of the probability $P_{ij}^n$ and first-passage-time probability $f_{ij}^n$  given by
\begin{align*}
P_{ij}(s)&=\sum_{n=0}^\infty P_{ij}^ns^n, & F_{ij}(s)&=\sum_{n=1}^\infty f_{ij}^n s^n,
\end{align*}
are connected by
\begin{align*}
P_{ij}(s)&=F_{ij}(s)P_{ij}(s), & i&\neq j,&
P_{jj}(s)&=1+F_{jj}(s)P_{jj}(s).
\end{align*}
That allows us to express the generating functions of the first time passage distributions after~$n$ transitions in terms of the generating functions for the transition probability after~$n$ transitions.

\begin{teo}\label{teo:KMcG2}
Let us assume the conditions requested in Theorem~\ref{pro:sigma_spectral}. Then, for $|s|<1$, the transition probability generating function reads~as
\begin{align*}
P_{II,nm}(s)&=\frac{B_{\vec\nu(m)}(\lambda)}{B_{\vec\nu(n)}(\lambda)}
\bigintssss_\Delta \frac{B_{\vec \nu(n)}(x) Q_{\vec \nu(m+1)}(x)}{1-\frac{s}{\lambda}x}\d \mu(x),
\end{align*}
while for the first passage generating function we have
\begin{align*}
F_{II,nm}(s)&=\frac{B_{\vec\nu(m)}(\lambda)}{B_{\vec\nu(n)}(\lambda)}
\dfrac{\bigintss_\Delta \dfrac{B_{\vec \nu(n)}(x) Q_{\vec \nu(m+1)}(x)}{1-\frac{s}{\lambda}x}\d \mu(x)}{\bigintss_\Delta \dfrac{B_{\vec \nu(m)}(x) Q_{\vec \nu(m+1)}(x)}{1-\frac{s}{\lambda}x}\d \mu(x)}, &
n\neq m, \\
F_{nn}(s)&=1-\frac{1}{\bigintss_\Delta \dfrac{B_{\vec \nu(n)}(x) Q_{\vec \nu(n+1)}(x)}{1-\frac{s}{\lambda}x}\d \mu(x)}.
\end{align*}
Let us assume the conditions in Theorem~\ref{pro:sigma_spectral_I}. Then, for
$|s|<1$, the transition probability generating function reads~as
\begin{align*}
P_{I,nm}(s)&=\frac{Q_{\vec\nu(m+1)}(\lambda)}{Q_{\vec\nu(n+1)}(\lambda)}
\bigintssss_\Delta \frac{B_{\vec \nu(m)}(x) Q_{\vec \nu(n+1)}(x)}{1-\frac{s}{\lambda}x}\d \mu(x),
\end{align*}
while for the first passage generating function we have
\begin{align*}
F_{I,nm}(s)&=\frac{Q_{\vec\nu(m+1)}(\lambda)}{Q_{\vec\nu(n+1)}(\lambda)}
\dfrac{\bigintss_\Delta \dfrac{B_{\vec \nu(m)}(x) Q_{\vec \nu(n+1)}(x)}{1-\frac{s}{\lambda}x}\d \mu(x)}{\bigintss_\Delta \dfrac{B_{\vec \nu(m)}(x) Q_{\vec \nu(m+1)}(x)}{1-\frac{s}{\lambda}x}\d \mu(x)},&
n\neq m, \\
F_{nn}(s)&=1-\frac{1}{\bigintss_\Delta \dfrac{B_{\vec \nu(n)}(x) Q_{\vec \nu(n+1)}(x)}{1-\frac{s}{\lambda}x}\d \mu(x)}.
\end{align*}
\end{teo}

\begin{proof}
Let us denote by $P_{II}(s)=\big(P_{II,nm}(s)\big)_{n,m\in\N_0}$ the semi-infinite matrix whose coefficients are the transition probability generating functions.
 Then, as $|s|<1$ and $\|P_{II}\|_\infty=1$, we know that for the first Neumann type expansion
 \begin{align*}
 P_{II}(s)&=\sum_{k=0}^\infty P_{II}^k s^k=\big(I-sP_{II}\big)^{-1},
 \end{align*}
 uniformly in norm.
 Now, as
 \begin{align*}
 P_{II} \sigma_{II}B(x)=\frac{1}{\lambda}\sigma_{II}J\sigma_{II}^{-1} \sigma_{II}B(x)=\frac{x}{\lambda}\sigma_{II}B(x),
 \end{align*}
we conclude that
\begin{align*}
\big(I-sP_{II}\big)\sigma_{II}B(x)=\Big(1-s\frac{x}{\lambda}\Big)\sigma_{II}B(x),
\end{align*}
and, consequently,
\begin{align*}
\big(I-sP_{II}\big)^{-1}\sigma_{II}B(x)=\frac{1}{1-s\frac{x}{\lambda}}\sigma_{II}B(x).
\end{align*}
Therefore,
\begin{align*}
\sum_{m=0}^\infty P_{II,nm}(s)\sigma_{II,m}B^{ ( m )}(x)
=\dfrac{1}{1- \frac{s}{\lambda}x}\sigma_{II,n} B^{(n)}(x).
\end{align*}
Now, using biorthogonality~\eqref{biotrhoganility} we get
\begin{align*}
P_{II, nm}(s)=\frac{\sigma_{II,n}}{\sigma_{II,m}}\int_{\Delta}\dfrac{1}{1- \frac{s}{\lambda}x} B^{(n)}(x)Q^{ ( { m } )}(x)\d \mu (x).
\end{align*}
For the type~I we proceed analogously. For $|s|<1$, the following first Neumann expansion converges uniformly
\begin{align*}
P_{I}(s)&=\sum_{k=0}^\infty P_{I}^k s^k=\big(I-sP_{I}\big)^{-1}.
\end{align*}
 Now, analogously to the previous discussion, we have
\begin{align*}
	P_{I} \sigma_{I}Q(x)=\frac{1}{\lambda}\sigma_{I}J^\top\sigma_{I}^{-1} \sigma_{I}Q(x)=\frac{x}{\lambda}\sigma_{I}Q(x),
\end{align*}
so that
\begin{align*}
	\big(I-sP_{I}\big)\sigma_{I}Q(x)=\Big(1-s\frac{x}{\lambda}\Big)\sigma_{I}Q(x)
\end{align*}
and, consequently,
\begin{align*}
	\big(I-sP_{I}\big)^{-1}\sigma_{I}Q(x)=\frac{1}{1-s\frac{x}{\lambda}}\sigma_{I}Q(x).
\end{align*}
Hence,
\begin{align*}
\sum_{m=0}^\infty P_{I,nm}(s)\sigma_{I,m}Q^{ ( m )}(x)
	=\dfrac{1}{1- \frac{s}{\lambda}x}\sigma_{I,n} Q^{(n)}(x)
\end{align*}
and biorthogonality~\eqref{biotrhoganility} leads to
\begin{align*}
P_{I, nm}(s)=\frac{\sigma_{I,n}}{\sigma_{I,m}}\int_{\Delta }\dfrac{1}{1- \frac{s}{\lambda}x} B^{ ( { m} )}(x)Q^{(n)}(x)\d \mu (x),
\end{align*}
as desired.
\end{proof}

\begin{lemma}\label{lemma:recurrent_n}
For both Markov chains, the $n$-state is recurrent or transient whenever
	\begin{align}\label{eq:recurrent}
	\bigintssss_\Delta \dfrac{B_{\vec \nu(n)}(x) Q_{\vec \nu(n+1)}(x)}{1-\frac{x}{\lambda}}\d \mu(x)
\end{align}
diverges or converges, respectively.

\end{lemma}
\begin{proof}
The limit
\begin{align*}
	F^\infty_{nn}=\lim_{s\to 1^-} F_{nn}(s)
\end{align*}
describes the probability that the $n$-th state is visited again, that is that the state is recurrent.
	According to the previous results, 
	\begin{align}\label{eq:recurrent_states}
		\lim_{s\to 1^-}F_{nn}(s)&=1-\frac{1}{\bigintss_\Delta \dfrac{B_{\vec \nu(n)}(x) Q_{\vec \nu(n+1)}(x)}{1-\frac{x}{\lambda}}\d \mu(x)} .
	\end{align}
	Thus, the $n$-th state is visited again whenever
the integral in \eqref{eq:recurrent}
	diverges to $+\infty$, so that
	\begin{align*}
		\lim_{s\to 1^-} F_{nn}(s)=1, 
	\end{align*}
	as we wanted to show.
\end{proof}

\begin{teo}\label{teo:recurrent_state}
Both dual Markov chains are recurrent if and only if the integral
\begin{align}\label{eq:integral}
\int_\Delta \frac{w_1(x)}{1-\frac{x}{\lambda}}\d \mu(x)
\end{align}
diverges. Both dual Markov chains are transient whenever the integral converges.
\end{teo}
\begin{proof}
As in our Markov chains there is only one class of states, one needs only to check if the state~$0$ is recurrent or transient.
Using the previous Lemma~\ref{lemma:recurrent_n} the result follows.
\end{proof}
\begin{rem}
Notice that, as $\lim\limits_{s\to 1^-} F_{nn}(s)\in[0,1]$, we know that
$\bigintss_\Delta \dfrac{B_{\vec \nu(n)}(x) Q_{\vec \nu(n+1)}(x)}{1-\frac{x}{\lambda}}\d \mu(x)\in\R_+$.
Hence, if $\Delta=[a,b]$, and we take $\lambda>b$, the integral in~\eqref{eq:recurrent} converges. Only if $\lambda\in [a,b]$ we can
achieve divergence.
\end{rem}


\begin{con} Whenever the integral~\eqref{eq:integral} converges both dual Markov chains are transient and~$\lambda$ can't be a mass point. If the integral~\eqref{eq:integral} diverges then, if $\lambda$ is a mass point both Markov chains are ergodic and if $\lambda$ is not a mass point then both Markov chains are null recurrent. That is, in the first case the expectation value of the return times $T_{jj}$ is finite and in the second, in average, it takes an infinite time for first return.
\end{con}

\section{Applications: The Jacobi--Piñeiro random walks}\label{Section:JP}

In this section we analyze an example of multiple orthogonal polynomials that fulfill the requirements of  Theorems \ref{pro:sigma_spectral} and \ref{pro:sigma_spectral_I} in the previous section. Therefore, we will have type~I and II multiple stochastic matrices and corresponding random walks.

The general context is as follows. In our notation we take $p=2$ with $\vec n=(1,1)$, we choose the step-line, and two Jacobi type weights
\begin{align*}
w_1(x)&=x^{\alpha} , &w_2(x)&=
x^{\beta},& \d\mu(x)&=(1-x)^{\gamma}\d x
\end{align*}
supported in $\Delta=[0,1]$. In order to ensure finiteness of the recursion coefficients and to avoid resonances, one takes $\alpha,\beta,\gamma>-1$ and $\alpha-\beta\not\in\Z$. The weight vectors are
\begin{align*}
\vec \nu(2n)&=(n+1,n), & \vec \nu(2n+1)&=(n+1,n+1), & n\in\N_0.
\end{align*}
These polynomials for $\gamma=0$ where considered for the first time by Luis Piñeiro in~\cite{pineiro} and in~\cite{nikishin_sorokin} the general situation was studied.

\subsection{Jacobi--Piñeiro multiple orthogonal polynomials of type~II}

Here we~follow Cou\-ssement and Van Assche~\cite{coussement} and Aptekarev, Branquinho and Van Assche~\cite{abv}.
In~\cite{abv}, using the Rodrigues formula, the polynomials of type~II where computed.

\begin{pro}[\cite{pineiro,nikishin_sorokin,coussement,abv}]
The monic Jacobi--Piñeiro multiple orthogonal polynomials of type~II are
\begin{align*}
B^{(2n)}(x)&=B_{(n,n)}(x), & B^{(2n+1)}(x)=B_{(n+1,n)}(x),
\end{align*}
where
\begin{align}\label{eq:typeII}
B_{(n,m)} (x)&= N_{n,m}\sum _{k=0}^n \sum_{j=0}^m B_{n,m}^{k,j}(x-1)^{j+k}x^{-j-k+m+n}, \\\notag
&= N_{n,m}
\sum _{k=0}^n
\sum_{j=0}^m \tilde B_{n,m}^{k,j}
\sum_{i=0}^{n+m-j-k}
\binom{n + m - k - j}{i} (-1)^{j+k} x^{k + j + i},
\end{align}
with
\begin{align*}
B_{n,m}^{k,j}&:=\frac{(\gamma +j+k+1)_{m-j}
(\gamma +k+m+1)_{n-k} (\alpha
-k+n+1)_k (\beta -j-k+m+n+1)_j}{j! k!
(m-j)! (n-k)!}, \\
\tilde B_{n,m}^{k,k} & :=
\frac{(\gamma +n+m-j-k+1)_{j}
(\gamma +n+m-k+1)_{k}
(\alpha+k+1)_{n-k}
(\beta +j+k+1)_{m-j}}{j! k! (m-j)! (n-k)!} , \\
N_{n,m}&:=\frac{1}{\displaystyle \sum _{k=0}^n \sum_{j=0}^mB^{k,j}_{n,m}}
=
\frac{n!m!}{(n+m+\alpha+\gamma+1)_n (n+m+\beta+\gamma+1)_m},
\end{align*}
where we have used the Pochhammer symbol $(z)_k=z(z+1)\cdots(z+k-1)$.
\end{pro}
\begin{proof}
These results have been proved in the references quoted. However, the value of $N_{n,m}$ requires a discussion.
Recalling the definition of the binomial function
\begin{align*}
\binom{z}{k}:=\frac{(z-k+1)_k}{k!}
\end{align*}
we immediately see that
\begin{align*}
B^{k,j}_{n,m}=\binom{\gamma+k+m}{m-j}\binom{\gamma+n+m}{n-k}\binom{\alpha+n}{k}\binom{\beta+n+m-k}{j}
\end{align*}
in agreement with \S3.3. in~\cite{abv}.
Thus, the normalization factor is
\begin{align*}
 N_{n,m}^{-1}
 & =
\sum _{k=0}^n \sum_{j=0}^mB^{k,j}_{n,m}
 = \sum _{k=0}^n\binom{\alpha+n}{k}\binom{\gamma+n+m}{n-k}\sum_{j=0}^m\binom{\beta+n+m-k}{j}\binom{\gamma+k+m}{m-j}
\\
& =
\sum _{k=0}^n\binom{\alpha+n}{k}\binom{\gamma+n+m}{n-k}\binom{\beta+\gamma+n+2m}{m} \\
& = \binom{\alpha+\gamma+2n+m}{n}\binom{\beta+\gamma+n+2m}{m}\\
& = \frac{(\alpha+\gamma+2n+m-n+1)_n(\beta+\gamma+n+2m-m+1)_m}{n!m!} \\
& =
\frac{(\alpha+\gamma+n+m+1)_n(\beta+\gamma+n+m+1)_m}{n!m!},
\end{align*}
where we have used the Chu--Vandermonde identity twice. A similar argument was used in~\cite{coussement}.
\end{proof}

\begin{coro}\label{cor:B values at 1 and 0}
The values of the multiple orthogonal polynomials at $1$ is
\begin{align}\label{eq:Ben1}
B_{(n,m)}(1)
& = \frac{(\gamma +1)_{m+n}}{(\alpha +\gamma +m+n+1)_n (\beta +\gamma +m+n+1)_m}.
\end{align}
Moreover, at the origin we have
\begin{align}\label{eq:Ben0}
B_{(n,m)}(0)=(-1)^{m+n}\frac{(\alpha+1)_n(\beta+1)_m}{(\alpha +\gamma +m+n+1)_n (\beta +\gamma +m+n+1)_m}.
\end{align}
\end{coro}
\begin{proof}
It follows from~\eqref{eq:typeII}.
\end{proof}

\begin{rem}
Thus, we explicitly see that the value at $1$ of the monic type~II multiple orthogonal polynomials is positive. This we already knew from the AT property of the Jacobi--Piñeiro weights. All the zeros of the polynomials lay in $(0,1)$ and its value at $+\infty$ is $+\infty$. Therefore, the value at~$1$ is positive. The  value at $-\infty$ is $(-1)^{m+n}\infty$, thus $B_{(n,n)}(0)$ is always a positive number and $B_{(n+1,n)}(0)$ is a negative number.
\end{rem}

\begin{rem}
Using the generalized hypergeometric function ${}_{3}F_{2}$ we have~\cite{vanassche_ismail}
\begin{multline*}
(1-x)^\gamma B_{(n,m)}(x)=
\frac{(-1)^{n+m} (\alpha+1)_n (\beta+1)_m}{(n+m+\alpha+\gamma+1)_{n}(n+m+\beta+\gamma+1)_{m}}
 \\
\times
{}_{3}F_{2}\left[{\begin{array}{c}-n-m-\gamma,\;\alpha+n+1 ,\;\beta +n+1\\\alpha +1,\;\beta+1\end{array}};x\right].
\end{multline*}
\end{rem}

\subsection{Jacobi--Piñeiro multiple orthogonal polynomials of type~I}

We will give in this section an explicit construction of the families of Jacobi--Piñeiro multiple orthogonal polynomials of type~I biorthogonal to the Jacobi--Piñeiro multiple orthogonal polynomials of type~II previously discussed. To our best knowledge, only in~\cite{leurs_vanassche} one can find explicit expressions for multiple orthogonal polynomials of type~I, in that case of Jacobi--Angelescu type. The explicit formulas for the Jacobi--Piñeiro multiple orthogonal polynomials of type~I given
in Theorem~\ref{theorem:JPI} is a new finding and have not been described before in the literature. Notice also that in~\cite{adler_van moerbeke_wang} complex integral expressions were presented for these polynomials.

The proof of the orthogonality of the Jacobi--Piñeiro multiple orthogonal polynomials of type~I that we will present here requires some preparation. The key result in the proof is Proposition~\ref{pro:cancellations}, that shows some interesting results, basic but nontrivial, about partial fraction expansions.

\begin{lemma}\label{lemma:finite_differences}
Given a polynomial $q(x)$ with $\deg q\leq n-2$
we have
\begin{align*}
\sum _{j=0}^{n-1}(-1)^{j}\binom {n-1}{j}q(\alpha+j)=0.
\end{align*}
\end{lemma}

\begin{proof}
Is a well known fact in the theory of finite differences that any polynomial $p(x)$
with $\deg p\leq n-2$ is such that
\begin{align*}
\sum _{j=0}^{n-1}(-1)^{j}\binom {n-1}{j}p ( { j } )=0.
\end{align*}
This can be shown, for example, by taken consecutive derivatives of the relation
\begin{align*}
(x+1)^{n-1}=\sum_{j=0}^{n-1}\binom{n-1}{j}x^j,
\end{align*}
up to the $(n-2)$-th derivative, evaluating at $x=1$ and taking linear combinations to get
\begin{align*}
\sum _{j=0}^{n-1}(-1)^{j}\binom {n-1}{j}j^k&=0, &k\in\{0,1, \ldots , n-2\}.
\end{align*}
Obviously $p(x)=q(\alpha+x)$ is a polynomial with $\deg q\leq n-2$.
\end{proof}

Let us introduce the function that collects all the singularities involved
\begin{align*}
F_{n}^\beta(\alpha):= \prod_{m=0}^{n}\frac{1}{\alpha-\beta-m}.
\end{align*}
\begin{lemma}
For $k\in\{0, \ldots , n-1\}$ the following relation is fulfilled
\begin{align}\label{eq:res}
\sum_{j= 0}^{n-1} \binom{n-1}{j} \frac{(-1)^{n-1-k+j}}{k! (n-1-k)!}\frac{1}{\alpha-\beta+j-k}
= -\binom{n-1}{k} (-1)^k F_{n-1}^\alpha(\beta+k).
\end{align}

\end{lemma}

\begin{proof}
Writing
\begin{align}
F_{n-1}^\beta(\alpha+j)=\frac{1}{\alpha-\beta+j-k}\prod_{m=0}^{k-1}\frac{1}{\alpha-\beta+j-m} \prod_{m=k+1}^{n-1}\frac{1}{\alpha-\beta+j-m}
\end{align}
we get the following value for the residue
\begin{align*}
\operatorname{Res}\big(F_{n-1}^\beta(\alpha+j),\beta-j+k\big)= \prod_{m=0}^{k-1}\frac{1}{k-m} \prod_{m=k+1}^{n-1}\frac{1}{k-m} =
\frac{1}{k!}\frac{(-1)^{n-1-k}}{ (n-1-k)!}.
\end{align*}
Consequently, the following partial fraction decomposition holds
\begin{align}\label{eq:simple_fractions}
F_{n-1}^\beta(\alpha+j)=\sum_{k=0}^{n-1} \frac{1}{k!}\frac{(-1)^{n-1-k}}{ (n-1-k)!}\frac{1}{\alpha-\beta+j-k},
\end{align}
that can also be written as follows
\begin{align*}
F_{n-1}^\alpha(\beta+k)=\sum_{j=0}^{n-1} \frac{1}{j!}\frac{(-1)^{n-1-j}}{ (n-1-j)!}\frac{1}{\beta-\alpha+k-j}.
\end{align*}
Thus,
\begin{multline*}
\sum_{j= 0}^{n-1} (-1)^{j} \binom{n-1}{j} \frac{\operatorname{Res}\big(F_{n-1}^\beta(\alpha+j),\beta-j+k\big)}{\alpha-\beta+j-k}
\\ =
(n-1)! \sum_{j= 0}^{n-1} \frac{(-1)^{n-1-k+j}}{j!(n-1-j)! k!(n-1-k)!} \frac{1}{\alpha-\beta+j-k}
\end{multline*}
and
\begin{align*}
-(-1)^k \binom{n-1}{k} F_{n-1}^\alpha(\beta+k)=-(n-1)!\sum_{j=0}^{n-1}
\frac{(-1)^{n-1-j+k}}{ k!(n-1-k)!j!(n-1-j)!}\frac{1}{\beta-\alpha+k-j}
\end{align*}
and~\eqref{eq:res} follows.
\end{proof}

\begin{defi}
Given any polynomial $p(\alpha)$ let us introduce
\begin{align*}
G_{n-1}^{\alpha,\beta}&:=\frac{1}{(n-1)!}\sum_{j=0}^{n-1}(-1)^j\binom{n-1}{j}p(\alpha+j) F_{n-1}^\beta(\alpha+j),\\
G_{1,n-1}^{\alpha,\beta}&:=\frac{1}{n!}\sum_{j=0}^{n}(-1)^j\binom{n}{j}p(\alpha+j) F_{n-1}^\beta(\alpha+j),\\
G_{2,n-1}^{\alpha,\beta}&:=-\frac{1}{(n-1)!}\sum_{j=0}^{n-1}(-1)^j\binom{n-1}{j}p(\alpha+j) F_{n}^\beta(\alpha+j).
\end{align*}
\end{defi}

\begin{pro}\label{pro:cancellations}
For $\deg p\leq 2n-2$ we have
\begin{align}\label{eq:G}
G_{n-1}^{\alpha,\beta}+G_{n-1}^{\beta,\alpha}=0,
\end{align}
and if $\deg p \leq 2n-1$ we have
\begin{align}\label{eq:H}
G_{1,n-1}^{\alpha,\beta}+G_{2,n-1}^{\beta,\alpha}=0.
\end{align}
\end{pro}

\begin{proof}
An Euclidean division leads to
\begin{align*}
p(\alpha) F^\beta(\alpha)=q(\alpha)+r(\alpha) F^\beta(\alpha) ,
\end{align*}
with a quotient polynomial $q(\alpha)$ such that $\deg q\leq 2n-2-n=n-2$ and a reminder polynomial such that $\deg r < n$.
The partial fraction decomposition~\eqref{eq:simple_fractions} gives
\begin{align*}
G_{n-1}^{\alpha,\beta} = \begin{multlined}[t][.9\textwidth]
\frac{1}{(n-1)!}\sum_{j= 0}^{n-1} \binom{n-1}{j} (-1)^{j} q(\alpha + j) \\+ \frac{1}{(n-1)!}\sum_{k=0}^{n-1}p(\beta+k)\sum_{j= 0}^{n-1} \binom{n-1}{j}
\frac{(-1)^{n-1-k+j}}{k! (n-1-k)!}\frac{1}{\alpha-\beta+j-k},
\end{multlined}
\end{align*}
and using
 Lemma~\ref{lemma:finite_differences} we conclude that
\begin{align*}
G_{n-1}^{\alpha,\beta} = \frac{1}{(n-1)!}\sum_{k=0}^{n-1}p(\beta+k)\sum_{j= 0}^{n-1} \binom{n-1}{j}
\frac{(-1)^{n-1-k+j}}{k! (n-1-k)!}\frac{1}{\alpha-\beta+j-k}.
\end{align*}
%
Finally,~\eqref{eq:res} leads to
\begin{align*}
G_{n-1}^{\alpha,\beta} = -\frac{1}{(n-1)!}\sum_{k=0}^{n-1}\binom{n-1}{k} (-1)^k p (\beta+k) F_{n-1}^\alpha(\beta+k)=-G_{n-1}^{\beta,\alpha} ,
\end{align*}
and~\eqref{eq:G} is proven.

For
\begin{align*}
G_{1,n-1}^{\alpha,\beta}+G_{2,n-1}^{\beta,\alpha}&=\begin{multlined}[t][.7\textwidth]
\frac{1}{n!}\sum_{j=0}^{n}(-1)^j
\binom{n}{j}p(\alpha+j)F_{n-1}^\beta(\alpha)\\
-\frac{1}{(n-1)!}\sum_{j=0}^{n-1}(-1)^j\binom{n-1}{j}p(\beta+j)F_{n}^\alpha(\beta+j),
\end{multlined}
\end{align*}
we perform a partial fraction expansion in each summand as a function of the independent variable $\alpha$ to obtain
\begin{multline*}
\phantom{o}
\hspace{-.5cm}G_{1,n-1}^{\alpha,\beta}+G_{2,n-1}^{\beta,\alpha}
=
\frac{1}{n!}\sum_{j=0}^{n}(-1)^j
\binom{n}{j}q(\alpha+j)+\sum_{j=0}^{n}\sum_{k=0}^{n-1} \frac{1}{n!}\binom{n}{j}\frac{(-1)^{n-1-k+j}p(\beta+k)}{ k!(n-1-k)!}\frac{1}{\alpha-\beta+j-k}
\\
-\sum_{k=0}^{n}
\sum_{j=0}^{n-1}\frac{1}{(n-1)!}\binom{n-1}{j} \frac{(-1)^{n-k+j}p(\beta+j)}{ k!(n-k)!}\frac{1}{\beta-\alpha+j-k},
\end{multline*}
where $\deg q =2n-1-n-1=n-2$. Thus, from Lemma~\ref{lemma:finite_differences} we get
\begin{align*}
G_{1,n-1}^{\alpha,\beta}+G_{2,n-1}^{\beta,\alpha}
&=
\begin{multlined}[t][.7\textwidth]
-\sum_{j=0}^{n}\sum_{k=0}^{n-1} \frac{(-1)^{n-k+j}p(\beta+k)}{ k!(n-1-k)!j!(n-j)!}\frac{1}{\alpha-\beta+j-k}
\\+\sum_{k=0}^{n}\sum_{j=0}^{n-1} \frac{(-1)^{n-k+j}p(\beta+j)}{ k!(n-k)!j!(n-1-j)!}\frac{1}{\alpha-\beta+k-j}=0,
\end{multlined}
\end{align*}
and the desired result~\eqref{eq:H} follows.
 \end{proof}

\begin{teo}\label{theorem:JPI}
The Jacobi--Piñeiro multiple orthogonal polynomials of type~I are
\begin{align*}
A_{(n,n),1}^{\alpha, \beta, \gamma}(x) &=
\begin{multlined}[t][0.75\textwidth]
A_{(n,n),2}^{\beta,\alpha, \gamma}(x)=
\frac{(2n+ \alpha + \gamma)_{n} (2n + \beta + \gamma)_{n}}{(n-1)!\Gamma( 2 n + \gamma)}
\\ \times \sum_{j= 0}^{n-1} (-1)^j \binom{n - 1}{j} \frac{\Gamma( 3 n - (j + 1) + \alpha + \gamma)}{\Gamma(n - j + \alpha)(\alpha - \beta - j)_n}x^{n - 1 - j} , \quad n\in\N,
\end{multlined}
\\
A_{(n+1,n),1}^{\alpha, \beta, \gamma} (x)&=\begin{multlined}[t][.75\textwidth]
\frac{\Gamma (3 n + 2 + \alpha + \gamma) ( 2 n + 1+ \beta + \gamma)_n}{n!\Gamma ( 2 n + 1+\gamma )}
\\ \times \sum_{j=0}^{n} (-1)^{j} \binom{n}{ j}
\frac{(2 n + 1+\alpha + \gamma )_{n-j}}{\Gamma (\alpha + n + 1 - j) (\alpha - \beta + 1 - j)_n}x^{n - j} , \quad n\in\N_0,
\end{multlined}\\
A_{(n+1,n),2}^{\alpha, \beta, \gamma} (x)&=\begin{multlined}[t][.75\textwidth]
(-1)^{n-1}\frac{\Gamma (3 n + 1 + \beta + \gamma) (2 n + 1 + \alpha + \gamma)_{n+1}}{(n-1)!\Gamma(2 n + 1 + \gamma)}
\\ \times \sum_{j=0}^{n-1} (-1)^j \binom{n - 1}{ j}
\frac{( 2 n + 1+\beta +\gamma )_{n-1-j}}{\Gamma (\beta + n - j) (\alpha - \beta - (n - 1) + j)_{n+1}}x^{n - 1 - j}, \quad n\in\N.
\end{multlined}
\end{align*}
\end{teo}
\begin{proof}
One can check that
\begin{align*}
\int_0^1A_{(n,n),1}^{\alpha, \beta, \gamma}(x) x^\alpha (1-x)^\gamma x^m\d x
=
C^{\alpha,\beta,\gamma}_{(n,n)}\sum_{j=0}^{n-1} (-1)^{j } \binom{n-1}{j} p^\gamma_m(\alpha+j)F^\beta_{n-1}(\alpha)
\end{align*}
with
\begin{align*}
C^{\alpha,\beta,\gamma}_{(n,n)}&:=\frac{(-1)^{n-1}}{(n-1)!}\frac{(2n + \alpha + \gamma)_n (2n+ \beta + \gamma)_n}{(1+ \gamma)_{2n-1} },\\
p^\gamma_m(\alpha)&:=\begin{cases} (2+\alpha + \gamma )_{2n-2}, &m=0,\\
 (2+m+\alpha + \gamma )_{2n-2-m}
(1+\alpha)_m, &m\in\{1, \ldots , 2n-3\}.
\end{cases}
\end{align*}

Notice that $C^{\alpha,\beta,\gamma}_{(n,n)}$ are symmetric in the interchange $\alpha \rightleftarrows\beta$ and that $p^\gamma_m(x)$ is a polynomial of degree $2n-2$.
Consequently, according~\eqref{eq:G}
\begin{multline*}
\int_0^1\big(A_{(n,n),1}^{\alpha, \beta, \gamma}(x) x^\alpha (1-x)^\gamma+A_{(n,n),2}^{\alpha, \beta, \gamma}(x) x^\beta (1-x)^\gamma\big) x^m\d x\\=C^{\alpha,\beta,\gamma}_{(n,n)}\Big(
\sum_{j=0}^{n-1} (-1)^{j } \binom{n-1}{j} p^\gamma_m(\alpha+j)F^\beta_{n-1}(\alpha)+
\sum_{j=0}^{n-1} (-1)^{j } \binom{n-1}{j} p^\gamma_m(\beta+j)F^\alpha_{n-1}(\beta)
\Big)=0
\end{multline*}
for $m\in\{0,1, \ldots , 2n-2\}$.

In terms of
\begin{align*}
\tilde C^{\alpha,\beta,\gamma}_{1,(n+1,n)}&:=(-1)^n\frac{ (2 n + 1+ \alpha + \gamma)_{n+1} (2 n + 1 + \beta + \gamma\big)_n}{
\big(1+ \gamma \big)_{2n}} ,\\
\tilde p_{m}^\gamma(\alpha)&:=
\begin{cases}
( 2 + j +\alpha + \gamma )_n, &m=0,\\
(1+j+\alpha )_m
(2 + j + m+ \alpha + \gamma )_{2n-m+1}, &m\in\{1, \ldots , 2n-1\},
\end{cases}
\end{align*}
and recalling~\eqref{eq:H} we obtain the following type~I multiple orthogonality relations
\begin{multline*}
\int_0^1\big(A_{(n+1,n),1}^{\alpha, \beta, \gamma}(x) x^\alpha (1-x)^\gamma+A_{(n+1,n),2}^{\alpha, \beta, \gamma}(x) x^\beta (1-x)^\gamma\big) x^m\d x\\=\tilde C^{\alpha,\beta,\gamma}_{(n,n)}\Big(
\frac{1}{n!}\sum_{j=0}^{n} (-1)^{j } \binom{n}{j} p^\gamma_m(\alpha+j)F^\beta_{n-1}(\alpha)-\frac{1}{(n-1)!}
\sum_{j=0}^{n-1} (-1)^{j } \binom{n-1}{j} p^\gamma_m(\alpha+j)F^\alpha_{n}(\beta)
\Big)=0
\end{multline*}
for $m\in\{0,1, \ldots , 2n-1\}$.
\end{proof}

\begin{coro}\label{cor:H}
We have
\begin{align*}
H_{2n}&=\frac{n!\Gamma ( 2 n + 1+\gamma )\Gamma (n + 1 +\alpha ) }{\Gamma (3 n + 2 + \alpha + \gamma) }
\frac{(\alpha - \beta + 1 )_n}{ (2 n + 1+\alpha + \gamma )_{n}(2 n + 1+ \beta + \gamma)_n},\\
H_{2n+1}&=\frac{n!\Gamma( 2 n+2 + \gamma)\Gamma(n +1+ \beta)}{\Gamma( 3 n +2+ \beta + \gamma)}
 \frac{(\beta - \alpha)_{n+1}}{(2n+2+ \beta + \gamma)_{n+1} (2n +2+ \alpha + \gamma)_{n+1}}.
\end{align*}
\end{coro}
\begin{proof}
The leading coefficients of the polynomials~$A_{(n+1,n),1}$ and $A_{(n+1,n+1),2}$ are
$H_{2n}^{-1}$ and $H_{2n+1}^{-1}$, respectively. Then, Theorem~\ref{theorem:JPI} gives the~result.
\end{proof}

\begin{rem}
In~\cite{leurs_vanassche} one finds also explicit expressions for type~I multiple orthogonal polynomials. The case considered by Leurs \& Van Assche are the Jacobi--Angelescu multiple orthogonal polynomials on an $r$-star. Despite that those Jacobi--Angelescu polynomials in~\cite{leurs_vanassche} and the Jacobi--Piñeiro polynomials constructed in our paper are different extensions of the Jacobi polynomials the explicit formulas for both families, which are different, intriguingly enough, look alike.
\end{rem}

According to Mathematica 12 we have

\begin{coro}\label{corollary:3F2_linear_form}
The type~I Jacobi--Piñeiro orthogonal polynomials can be expressed in terms of the generalized hypergeometric function
${}_{3}F_{2}$ as follows
\begin{align*}
A_{(n,n),1}^{\alpha, \beta, \gamma}(x)&
\begin{multlined}[t][.85\textwidth]
=A_{(n,n),2}^{\beta,\alpha, \gamma}(x)=
\frac{\Gamma (\alpha +\gamma +3 n-1) }{\Gamma (\alpha +n) \Gamma (\gamma +2 n) }
\frac{(\alpha +\gamma +2 n)_n (\beta +\gamma+2n)_n }{(\alpha-\beta)_n} \\
\times
\frac{z^{n-1} }{(n-1)! } \,
{}_{3}F_{2}\left[{\begin{array}{c}1-n,\;-\alpha-n+1 ,\;-\alpha +\beta -n+1\\-\alpha +\beta +1,\;-\alpha -\gamma -3 n+2\end{array}};\frac{1}{z}\right], \quad n\in\N,
\end{multlined} \\
 A_{(n+1,n),1}^{\alpha, \beta, \gamma}(x)&=
\begin{multlined}[t][.75\textwidth]
\frac{\Gamma (\alpha +\gamma +3 n+2) }{\Gamma (\alpha +n+1)
 \Gamma (\gamma +2 n+1)}\frac{(\alpha +\gamma +2 n+1)_n (\beta +\gamma +2 n+1)_n}{(\alpha -\beta +1)_{n} }
 \\
\times
\frac{z^n}{n!} \,
{}_{3}F_{2}\left[{\begin{array}{c}-\alpha-n,\;-\alpha+\beta -n ,\;-n\\-\alpha +\beta ,\;-\alpha -\gamma -3 n\end{array}};\frac{1}{z}\right], \quad n\in\N_0,
\end{multlined}
 \\
A_{(n+1,n),2}^{\alpha, \beta, \gamma}(x)&=
\begin{multlined}[t][.75\textwidth]
\frac{ \Gamma (\beta +\gamma +3 n) }{ \Gamma (\beta +n)\Gamma (\gamma +2 n+1) }\frac{(\alpha +\gamma +2 n+1)_{n+1}(\beta +\gamma +2 n+1)_n}{(\alpha -\beta -n+1)_{n+1}}
 \\
\times(-1)^{n-1}\frac{z^{n-1}}{(n-1)!} \, {}_{3}F_{2}\left[{\begin{array}{c}1-n,\;-\beta -n+1 ,\;\alpha-\beta-n+1\\\alpha -\beta+2 ,\;-\beta -\gamma -3 n+1\end{array}};\frac{1}{z}\right], \quad n\in\N.\end{multlined}
\end{align*}
\end{coro}

\begin{coro}
For $ n \in \N$, the linear forms of type~I with argument unity are
\begin{align*}
\begin{multlined}[t][.95\textwidth]
Q_{(n,n)}(1) =
\frac{(\alpha +\gamma +2 n)_n (\beta +\gamma+2n)_n }{(n-1)! \Gamma (\gamma +2 n)} \\
\times \bigg(\frac{\Gamma (\alpha +\gamma +3 n-1) }{\Gamma (\alpha +n) (\alpha-\beta)_n}
{}_{3}F_{2}\left[{\begin{array}{c}1-n,\;-\alpha-n+1 ,\;-\alpha +\beta -n+1\\-\alpha +\beta +1,\;-\alpha -\gamma -3 n+2\end{array}};1\right]
\\ +
\frac{\Gamma (\beta +\gamma +3 n-1) }{\Gamma (\beta +n) (\beta-\alpha)_n}
{}_{3}F_{2}\left[{\begin{array}{c}1-n,\;-\beta-n+1 ,\;-\beta +\alpha -n+1\\-\beta +\alpha +1,\;-\beta -\gamma -3 n+2\end{array}};1\right]\bigg),
\end{multlined}
 \\
\begin{multlined}[t][.95\textwidth] Q_{(n+1,n)}(1) =
\frac{(\alpha +\gamma +2 n+1)_n (\beta +\gamma +2 n+1)_n}{(n-1)!\Gamma (\gamma +2 n+1)}\\
\times
\bigg(
\frac{\Gamma (\alpha +\gamma +3 n+2) }{n\Gamma (\alpha +n+1)
(\alpha -\beta +1)_{n} }{}_{3}F_{2}\left[{\begin{array}{c}-\alpha-n,\;-\alpha+\beta -n ,\;-n\\-\alpha +\beta ,\;-\alpha -\gamma -3 n\end{array}};1\right]
\\
+\frac{ (-1)^{n-1}\Gamma (\beta +\gamma +3 n) (\alpha +\gamma +3 n+1)}{ \Gamma (\beta +n)(\alpha -\beta -n+1)_{n+1}}{}_{3}F_{2}\left[{\begin{array}{c}1-n,\;-\beta -n+1 ,\;\alpha-\beta-n+1\\\alpha -\beta+2 ,\;-\beta -\gamma -3 n+1\end{array}};1\right]
\bigg),
\end{multlined}
\end{align*}
with
\begin{align*}
Q_{(1,0)}(1)=\frac{\Gamma (1 + \beta + \gamma) }{\Gamma(1 + \gamma)\Gamma(1+\alpha)}.
\end{align*}
\end{coro}

\subsection{The Jacobi matrix}

The coefficients of the Jacobi matrix
\begin{align}\label{eq:Jacobi_Jacobi_Piñeiro}
	{ J } = \left( \begin{NiceMatrix}[columns-width = 0.5cm]
		b_{0,0}^{\alpha,\beta,\gamma} & 1 & 0 & 0 & 0 & \Cdots\\ 
		c_{1,0}^{\alpha,\beta,\gamma} &b_{1,0}^{\alpha,\beta,\gamma} & 1 & 0 & 0 & \Ddots \\ 
		d_{1,1}^{\alpha,\beta,\gamma} & c_{1,1}^{\alpha,\beta,\gamma} & b_{1,1}^{\alpha,\beta,\gamma} & 1 &0&\Ddots\\ 
		0& d_{2,1}^{\alpha,\beta,\gamma} &c_{2,1}^{\alpha,\beta,\gamma} & b_{2,1} ^{\alpha,\beta,\gamma}& 1 & \Ddots\\ 
		\Vdots&\Ddots&\Ddots&\Ddots&\Ddots&\Ddots
	\end{NiceMatrix}\right)
\end{align}
were determined in~\cite{coussement}. Inspired by~\cite{abv} we obtain, for $n = 0, 1, \ldots$,
\begin{align*}
b_ {n, n}^{\alpha - 1, \beta - 1, \gamma - 1} & = \frac{B_{n,n}}{\tilde B_{n,n}},&b_ {n + 1, n}^{\alpha - 1, \beta - 1, \gamma - 1}
& =
\frac{B_{n+1,n}}{\tilde B_{n+1,n}},
 \\
c_ {n+1, n+1}^{\alpha - 1, \beta - 1, \gamma - 1}
& =\frac{C_{n+1,n+1}}{\tilde C_{n+1,n+1}} , &
c_ {n+1, n}^{\alpha - 1, \beta - 1, \gamma - 1} &
= \frac{C_{n+1,n}}{\tilde C_{n+1,n}},\\
d_ {n+1, n+1}^{\alpha - 1, \beta - 1, \gamma - 1}
& = \frac{D_{n+1,n+1}}{\tilde D_{n+1,n+1}} , &
d_ {n+2, n+1}^{\alpha - 1, \beta - 1, \gamma - 1}
& = \frac{D_{n+2,n+1}}{\tilde D_{n+2,n+1}},
\end{align*}
with
\begin{align*}
& \begin{multlined}[t][\textwidth]
B_ {n, n} = \frac{(\alpha +n) (\beta +\gamma +2 n-1) (\alpha +\gamma +2 n-1)}{(\alpha+\gamma +3 n) (\beta +\gamma +3 n-1)}\\
+\frac{n (\gamma +2 n-1) (\alpha +\gamma +2n-1)}{(\beta +\gamma +3 n-2) (\beta +\gamma +3 n-1)}
+\frac{n (\gamma +2 n-1) (\beta +\gamma +2 n-2)}{(\alpha +\gamma +3 n-2) (\beta +\gamma +3 n-2)},
\end{multlined}
 \\
&\tilde B_{n,n} =\alpha+\gamma +3 n-1,
 \\
& \begin{multlined}[t][\textwidth]
B_{n+1,n} =(\beta ^2+(\gamma +3 n-1) \beta +2 n (\gamma +2 n)) \alpha ^2
+(2\gamma +5 n) (\beta ^2+(\gamma +3 n-1) \beta +2 n (\gamma +2 n)) \alpha \\
+(\gamma +2 n) (18 n^3+(14 \beta +15 \gamma +5) n^2+(2 \beta +\gamma+2) (2 \beta +3 \gamma -1) n \\
+(\beta +1) (\gamma +1) (\beta +\gamma -1)),
\end{multlined}
 \\
&\tilde B_{n+1,n} =(\alpha +\gamma +3 n) (\alpha +\gamma +3 n+1) (\beta +\gamma +3 n-1)(\beta +\gamma +3 n+1),\\
&\begin{multlined}[t][\textwidth] C_{n+1,n+1} =
n (\gamma +2 n+1) (\alpha +\gamma +2 n) (\beta +\gamma +2 n) \Big((\beta+n) (\beta +\gamma +2 n) \\+\frac{(\alpha -\beta +n+1) (\gamma +2 n) (\beta+\gamma +3 n+1)}{\alpha +\gamma +3 n}+\frac{(\beta +n) (\alpha +\gamma +2n+1) (\alpha +\gamma +3 n+1)}{\beta +\gamma +3 n+2}\Big),
\end{multlined}\\
&\tilde C_{n+1,n+1} =(\alpha +\gamma +3n+1)^2 (\alpha +\gamma +3 n+2) (\beta +\gamma +3 n) (\beta +\gamma +3 n+1)^2,\\
&\begin{multlined}[t][\textwidth]
C_{n+1,n}=(\gamma + 2 n) (-1 + \alpha + \gamma + 2 n) (-1 + \beta + \gamma +2 n) \bigg(\alpha^3 n (-1 + \beta + n) \\
+ \alpha^2 n (-1 + \beta + n) (-1 + 3 \gamma + 8 n)
+ \alpha \Big(\beta^3 (1 + n) + \gamma^3 (1 + n) + \beta^2 (1 + n) (-3 + 3 \gamma + 8 n) \\
+ 3 \gamma^2 (-1 + n + 4 n^2) + \gamma (2 + n (-13 - 9 n + 42 n^2)) \\
+\beta \big(2 + \gamma^2 (3 + 6 n) + \gamma (-6 + 9 n + 33 n^2) + n (-15 + n + 44 n^2)\big) + n \big(6 + n (-13 + n (-26 + 45 n))\big)\Big) \\+n \Big(\beta^3 (1 + n) + \beta^2 (1 + n) (-3 + 3 \gamma + 8 n) + (-1 + \gamma + 3 n) (\gamma + 3 n) (-2 + \gamma + 3 \gamma n + 6 n^2) \\+ \beta \big(2 + \gamma^3 + 3 \gamma^2 (1 + 4 n) + \gamma (-7 + 9 n + 42 n^2) + n (-17 + n (2 + 45 n))\big)\Big)\bigg),\end{multlined}
 \\
&\begin{multlined}[t][\textwidth]
\tilde C_{n+1,n} =
(-1 + \alpha + \gamma + 3 n) (\alpha + \gamma + 3 n)^2 (1 + \alpha + \gamma + 3 n) (-2 + \beta + \gamma + 3 n)
 \\\times(-1 + \beta + \gamma +3 n)^2 (\beta + \gamma + 3 n),
\end{multlined}
 \\
&\begin{multlined}[t][\textwidth]
D_{n+1,n+1} =( n + 1) (\alpha +n) (\alpha -\beta +n+1) (\gamma +2 n) (\gamma +2 n+1) (\alpha +\gamma +2 n-1) \\\times (\alpha +\gamma +2 n) (\beta +\gamma +2 n-1) (\beta +\gamma +2 n)\end{multlined}
 \\
&\begin{multlined}[t][\textwidth]
\tilde D_{n+1,n+1} =
(\alpha +\gamma +3 n-1) (\alpha +\gamma +3 n)^2 (\alpha +\gamma +3 n+1)^2(\alpha +\gamma +3 n+2) \\
\times(\beta +\gamma +3 n-1) (\beta +\gamma +3 n) (\beta+\gamma +3 n+1),
\end{multlined}
 \\
&\begin{multlined}[t][\textwidth]
D_{n+2,n+1} = n (\beta +n) (-\alpha +\beta +n+1) (\gamma +2 n+1) (\gamma +2 n+2) (\alpha+\gamma +2 n) \\
\times(\alpha +\gamma +2 n+1) (\beta +\gamma +2 n) (\beta +\gamma +2n+1),\end{multlined}
 \\
&\begin{multlined}[t][\textwidth]
\tilde D_{n+2,n+1} =(\alpha +\gamma +3 n+1) (\alpha +\gamma +3 n+2) (\alpha +\gamma +3 n+3) (\beta +\gamma +3 n) (\beta +\gamma +3 n+1)^2 \\\times(\beta +\gamma +3 n+2)^2 (\beta+\gamma +3 n+3) .
\end{multlined}
\end{align*}


\begin{lemma}\label{Lemma_positivity}\phantom{text}
Let us consider the coefficients of the Jacobi matrix $ { J } $ given in~\eqref{eq:Jacobi_Jacobi_Piñeiro} for the Jacobi--Piñeiro multiple orthogonal polynomials given in~\eqref{eq:typeII}. Then:
\begin{enumerate}
\item The coefficient $b_ {n, n}^{\alpha - 1, \beta - 1, \gamma - 1}$ is positive for $n\in\N_0$ and $\alpha,\beta,\gamma>0$
\item The coefficient $c_ {n, n}^{\alpha - 1, \beta - 1, \gamma - 1}$ is positive for $n\in\N$, $\alpha,\beta,\gamma>0$ and
$\alpha-\beta+1>0$.
\item The coefficient $c_ {n+1, n}^{\alpha - 1, \beta - 1, \gamma - 1}$ is positive for $n\in\N_0$ and $\alpha,\beta,\gamma>0$.
\item The coefficient $d_ {n, n}^{\alpha - 1, \beta - 1, \gamma - 1}$ is positive for $n\in\N$, $\alpha,\beta,\gamma>0$ and
 $\alpha-\beta+1>0$.
\item The coefficient $d_ {n+1, n}^{\alpha - 1, \beta - 1, \gamma - 1}$ is positive for $n\in\N$, $\alpha,\beta,\gamma>0$ and
 $-\alpha+\beta+1>0$.
\end{enumerate}
\end{lemma}

\begin{proof} 
\begin{enumerate}
\item
For $n\in\N$ we immediately see that the denominator $\tilde B_{n,n}$ is positive. The numerator $B_{n,n}$ for $n\in\N$ is the sum of three positive rational expressions in $\alpha,\beta,\gamma>0$.
For $n=0$ we have $b_ {0, 0}^{\alpha - 1, \beta - 1, \gamma - 1} =\frac{\alpha }{\alpha +\gamma }$.

\item
For $n\in\N$ the denominator $\tilde B_{n+1,n}$ is positive. The numerator $B_{n+1,n}$ requires more analysis. We need to check the positivity of
\begin{align*}
T=18 n^3+(14 \beta +15 \gamma +5) n^2+(2 \beta +\gamma+2) (2 \beta +3 \gamma -1) n+(\beta +1) (\gamma +1) (\beta +\gamma -1),
\end{align*}
that ensures the positivity of $B_{n,n+1}$.
To understand that this positivity is not obvious we write $T$ as follows
\begin{align*}
\begin{multlined}[t][.95\textwidth]
T =
18 n^3+(14 \beta +15 \gamma +5) n^2+(2 \beta +\gamma+2) (2 \beta +3 \gamma) n \\
+(\beta +1) (\gamma +1) (\beta +\gamma)
-(2 \beta +\gamma+2) n-(\beta +1) (\gamma +1).
\end{multlined}
\end{align*}
But further manipulation do show that is a positive term. Indeed,
\begin{align*}
T&=
\begin{multlined}[t][.91\textwidth]
18 n^3+(12 \beta +14 \gamma +3) n^2+(2 \beta +\gamma+2) (2 \beta +3 \gamma) n
 \\
+(\beta +1) (\gamma +1) (\beta +\gamma)
+(2 \beta +\gamma+2) n(n-1)-(\beta +1) (\gamma +1)
\end{multlined}\\
&=\begin{multlined}[t][.91\textwidth]
18 n^3+(12 \beta +14 \gamma +3) n^2+\big((\beta +\gamma+1) (2 \beta +3 \gamma)+2(\beta+1)(\beta+\gamma)\big) n \\
+(\beta +1) (\gamma +1) (\beta +\gamma)
+(2 \beta +\gamma+2) n(n-1)+(\beta +1) \gamma(n-1) -(\beta +1)
\end{multlined}
 \\
&=\begin{multlined}[t][.91\textwidth]
18 n^3+(11 \beta +14 \gamma +2) n^2+\big((\beta +\gamma+1) (2 \beta +3 \gamma)+2(\beta+1)(\beta+\gamma)\big) n \\
+(\beta +1) (\gamma +1) (\beta +\gamma)
+(2 \beta +\gamma+2) n(n-1)+(\beta +1) \gamma(n-1) +(\beta +1) (n^2-1),
\end{multlined}
\end{align*}
and we see that $T$ is a positive number whenever $n\in \N$ and $\alpha,\beta,\gamma>0$, and so is $b_ {n+1, n}^{\alpha - 1, \beta - 1, \gamma - 1}$. Now, for $n=0$ we find
\begin{align*}
b_ {1, 0}^{\alpha - 1, \beta - 1, \gamma - 1}
=\frac{\beta \alpha ^2+2 \beta \gamma \alpha +(\beta +1) \gamma (\gamma +1)}{(\alpha +\gamma ) (\alpha +\gamma +1) (\beta +\gamma +1)},
\end{align*}
which is again positive.
\item
For $n\in\N$, the positivity of the denominator $\tilde C_{n,n}$ for $\alpha,\beta,\gamma>0$ is obvious by inspection. For the numerator $C_{n,n}$, the positivity is ensured whenever $\alpha-\beta+1>0$.
\item
For $n\in\N$, the positivity of the denominator $\tilde C_{n,n}$ for $\alpha,\beta,\gamma>0$ is obvious.
The numerator $C_{n+1,n}$ after inspection of its long expression, is also seen to be positive after cheeking the positivity of all its summands for $n=1,2,\ldots$.
Moreover, for $n=0$ we have
\begin{align*}
c_ {1, 0}^{\alpha - 1, \beta - 1, \gamma - 1}
=\frac{\alpha \gamma }{(\alpha +\gamma )^2 (\alpha +\gamma +1)}.
\end{align*}
Therefore, the positivity~holds in this case, as well.
\item
For $n\in\{2,3,\ldots\}$, the positivity of the denominator $\tilde D_{n,n}$ for $\alpha,\beta,\gamma>0$ is immediate. For the numerator $D_{n,n}$, the positivity is ensured whenever $\alpha-\beta+1>0$.
For $n=1$ we have
\begin{align*}
d_ {1, 1}^{\alpha - 1, \beta - 1, \gamma - 1}=\frac{\alpha (\alpha -\beta +1) \gamma (\gamma +1)}{(\alpha +\gamma )(\alpha +\gamma +1)^2 (\alpha +\gamma +2) (\beta +\gamma +1)}
\end{align*}
and the result is proven.
\item
For $n\in\N$, the denominator $\tilde D_{n+1,n}$ for $\alpha,\beta,\gamma>0$ is positive. The numerator $D_{n+1,n}$ is positive whenever $-\alpha+\beta+1>0$.
\end{enumerate}
Which completes the proof.
\end{proof}


\begin{teo}\label{teo:positive_JP}
For the Jacobi--Piñeiro multiple orthogonal polynomials as in~\eqref{eq:typeII} the corresponding Jacobi matrix given in~\eqref{eq:Jacobi_Jacobi_Piñeiro} is a nonnegative matrix whenever $\alpha,\beta,\gamma>-1$ and
$ |\alpha-\beta|<1$.
\end{teo}
 \begin{proof}
 From Lemma~\ref{Lemma_positivity} (recall the shift of the parameters $\alpha\to\alpha-1,\beta\to\beta-1,\gamma\to\gamma-1$) we see that when $\alpha,\beta,\gamma>-1$:
\begin{enumerate}
\item The coefficients $b_ {n, n}^{\alpha , \beta , \gamma }$ and $c_ {n+1, n}^{\alpha , \beta , \gamma }$ are positive.
 \item The coefficients $c_ {n, n}^{\alpha , \beta , \gamma }$ and $d_ {n, n}^{\alpha , \beta , \gamma }$ are positive for $\alpha-\beta+1>0$.
 \item The coefficient $d_ {n+1, n}^{\alpha , \beta , \gamma }$ is positive if
 $-\alpha+\beta+1>0$.
 \end{enumerate}
 Hence, for $x=\alpha-\beta$, we need to fulfill the couple of inequalities $x+1>0$ and $-x+1>0$, that is $x\in (-1,1)$.
 \end{proof}

Next we present a picture illustrating the possible values of the couple of parameters $(\alpha,\beta)$, the filled region that represents the possible values is an infinite band.
\begin{center}
 \begin{tikzpicture}[arrowmark/.style 2 args={decoration={markings,mark=at position #1 with \arrow{#2}}},scale=1]
 \begin{axis}[axis lines=middle,axis equal,grid=both,xmin=-1.7, xmax=3,ymin=-1.5, ymax=3.5,
 xticklabel,yticklabel,disabledatascaling,xlabel=$\beta$,ylabel=$\alpha$,every axis x label/.style={
 at={(ticklabel* cs:1)},
 anchor=south west,
 },
 every axis y label/.style={
 at={(ticklabel* cs:1.0)},
 anchor=south west,
 },grid style={line width=.1pt, draw=Bittersweet!10},
 major grid style={line width=.2pt,draw=Bittersweet!50},
 minor tick num=4,
 enlargelimits={abs=0.09},
 axis line style={latex'-latex'},Bittersweet]
 \node[anchor = north east,Bittersweet] at (axis cs: 4.1,2.5) {$\mathbb R^2$} ;
 \draw [fill=DarkSlateBlue!30,opacity=.5,dashed,thick] (-1,-1)--(-1,0)--(3,4) node[above, Black,sloped, pos=0.5] {$\alpha=\beta+1$}--(5,5)--(5,4) --(0,-1) node[below, Black,sloped, pos=0.6] {$\alpha=\beta-1$}--(-1,-1);
 \draw [fill=DarkSlateBlue!30,opacity=.5,dashed,thick] (-1,-1) -- (4,4 )node[below, Black,sloped, pos=0.4] {$\alpha=\beta$};
\draw[thick,black] (axis cs:-1,0) circle[radius=2pt,opacity=0.2,fill]node[left,above ] {$-1$} ;
\draw[thick,black] (axis cs:0,-1)circle[radius=2pt,opacity=0.2,fill] node[right,below ] {$-1$} ;
\draw[thick,black] (axis cs:-1,-1)node[left ] {$(-1,-1)$} ;
 \end{axis}
 \draw (3.5,-.7) node
 {\begin{minipage}{0.45\textwidth}
 \begin{center}\small
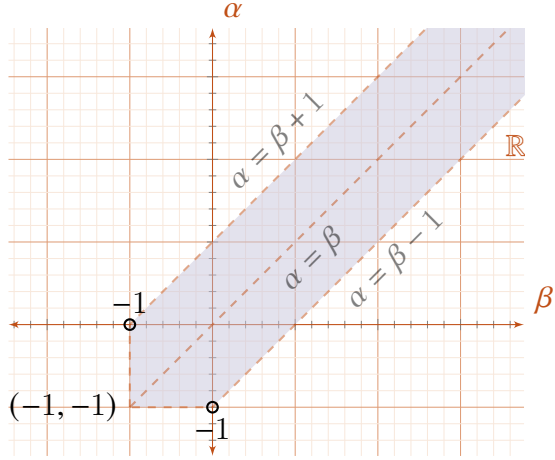
 \textbf{Parameter region for a non negative Jacobi--Piñeiro's Jacobi matrix}
 \end{center}
 \end{minipage}};
 \end{tikzpicture}
 \end{center}
The dashed lines are excluded of the allowed region for the parameters $\alpha$ and $\beta$, as those lines correspond to resonances, i.e., the difference $\alpha-\beta=\pm 1, 0$, over those semi-lines.

It is easy to see that, see for example~\cite{Coussement_coussment_VanAssche},
 \begin{coro}[Large $n$ limit of the Jacobi matrix]\label{coro:limit_Jacobi_Jacobi-Piñeiro}
 The following limits hold for large values of $n$ for the diagonals of the Jacobi matrix $J$
 \begin{align*}
 \lim_{n\to \infty} b_{n,n}&= \lim_{n\to \infty} b_{n+1,n}=\frac{2^2}{3^2}=3\kappa,\\
 \lim_{n\to \infty} c_{n,n}&= \lim_{n\to \infty} c_{n+1,n}=\frac{2^4}{3^5}=3\kappa^2,\\
 \lim_{n\to \infty} d_{n,n}&= \lim_{n\to \infty} d_{n+1,n}=\frac{2^6}{3^9}=\kappa^3.
 \end{align*}
 with $\kappa=\frac{4}{27}$.
 \end{coro}
\begin{rem}
Notice that the  large $n$ limit does not depend on $\alpha,\beta$ or $\gamma$.
\end{rem}

\begin{rem}
The Jacobi matrix can be split as follows
\begin{align*}
J=J_\infty+\delta { } J,
\end{align*}
with
\begin{align*}
	J_\infty &:= \left( \begin{NiceMatrix}[columns-width = 0.5cm]
		3\kappa & 1 & 0 & 0 & 0 & \Cdots\\ 
		3\kappa^2 &3\kappa& 1 & 0 & 0 & \Ddots \\ 
		\kappa^3& 3\kappa^2 & 3\kappa& 1 &0&\Ddots\\ 
		0& \kappa^3 &3\kappa^2 & 3\kappa& 1 & \Ddots\\ 
		\Vdots&\Ddots&\Ddots&\Ddots&\Ddots&\Ddots
	\end{NiceMatrix}\right), &
	\delta { } J &:= \left( \begin{NiceMatrix}[columns-width = 0.5cm]
		\delta { } b_{0,0} & 0 & 0 & 0 & 0 & \Cdots\\ 
		\delta { } c_{1,0}& \delta { } b_{1,0} & 0 & 0 & 0 & \Ddots \\ 
		\delta { } d_{1,1}& \delta { } c_{1,1}& \delta { } b_{1,1}& 0 &0&\Ddots\\ 
		0& \delta { } d_{2,1}& \delta { } c_{2,1}& \delta { } b_{2,1} & 0 & \Ddots\\ 
		\Vdots&\Ddots&\Ddots&\Ddots&\Ddots&\Ddots
	\end{NiceMatrix}\right), &\kappa&:=\frac{4}{27},
\end{align*}
in terms of a non-negative finite banded Toeplitz matrix $J_0$ and a perturbation matrix $\delta J$. In fact, as
\begin{align*}
\lim_{n\to \infty} \delta b_{n,n}= \lim_{n\to \infty} \delta b_{n+1,n}=
\lim_{n\to \infty} \delta c_{n,n}= \lim_{n\to \infty} \delta c_{n+1,n}=
\lim_{n\to \infty} \delta d_{n,n}= \lim_{n\to \infty} \delta d_{n+1,n}=0,
\end{align*}
we can ensure that $\delta J$ is a matrix of a compact operator, see~\cite{vanassche0}. The Toeplitz matrix $J_\infty$ was discussed in~\cite{Coussement_coussment_VanAssche}.
\end{rem}

 \subsection{Type~II Jacobi--Piñeiro's random walks}


Now, we explain how to turn stochastic the Jacobi matrix for the type~II given in~\eqref{eq:Jacobi_Jacobi_Piñeiro}. The zeros of the Jacobi--Piñeiro polynomials,
being an AT-system, are in $(0,1)$. Moreover, their density distribution of zeros fills that open interval and accumulate at the boundaries~\cite{Neuschel_ Van Assche}.
Thus, we will take $\lambda=1$ as the optimal choice in this case.

\begin{teo}\label{teo:JPII_stochastic}
Let us assume for the Jacobi--Piñeiro system that $\alpha,\beta,\gamma>-1$, $\alpha\neq \beta$ and $|\alpha-\beta|<1$. Then, the semi-infinite matrix
\begin{align}\label{eq:stochastic_Jacobi_Piñeiro}
P_{II}=
\left(\begin{NiceMatrix}[columns-width = .5cm]
P_{II,0,0} & P_{II,0,1}& 0 & 0 & 0 & \Cdots\\
P_{II,1,0}&P_{II,1,1}& P_{II,1,2}& 0 & 0 & \Ddots \\
P_{II,2,0}& P_{II,2,1}& P_{II,2,2} & P_{II,2,3}&0&\Ddots\\
0& P_{II,3,1}&P_{II,3,2}& P_{II,3,3} & P_{II,3,4}& \Ddots\\
\Vdots&\Ddots&\Ddots&\Ddots&\Ddots&\Ddots
\end{NiceMatrix}\right)
\end{align}
with coefficients given in terms of the coefficients of the Jacobi matrix~\eqref{eq:Jacobi_Jacobi_Piñeiro} and the multiple orthogonal polynomials of type~II evaluated at $z=1$, $B_{\vec\nu}(1)$, $\vec\nu=(n+1,n), (n,n)$ for $n =0,1,\ldots$, as~follows
\begin{align*}
P_{II,2n,2n+1} & = \frac{B_{(n+1,n)}(1)}{B_{(n,n)}(1)}, &
P_{II,2n+1,2n+2} & =\frac{B_{(n+1,n+1)}(1)}{B_{(n+1,n)}(1)}, \\
P_{II,2n,2n} & = b_{n,n}, &
P_{II,2n+1,2n+1} & = b_{n+1,n}, \\
P_{II,2n+2,2n+1} & = \frac{B_{(n+1,n)}(1)}{B_{(n+1,n+1)}(1)}c_{n+1,n+1}, &
P_{II,2n+1,2n}&=\frac{B_{(n,n)}(1)}{B_{(n+1,n)}(1)}c_{n+1,n}, \\
P_{II,2n+2,2n}&= \frac{B_{(n,n)}(1)}{B_{(n+1,n+1)}(1)}d_{n+1,n+1}, &
P_{II,2n+3,2n+1}&=\frac{B_{(n+1,n)}(1)}{B_{(n+2,n+1)}(1)}d_{n+2,n+1}.
\end{align*}
is a multiple stochastic matrix of type~II. Here the $b,c$ and $d$ coefficients are those in the Jacobi matrix~\eqref{eq:Jacobi_Jacobi_Piñeiro}.
\end{teo}
\begin{proof}
From the explicit expression~\eqref{eq:Ben1} we know that $B^{(n)}(1)$ is a strictly positive number, from the AT property for the system $\{ x^\alpha , x^\beta \}$.
Hence, using Theorem~\ref{pro:sigma_spectral} we can normalize at $x=1$ to get the stochastic Jacobi matrix using the factors $\sigma_{II,n}=\frac{1}{B^{(n)}(1)}$.
\end{proof}

The diagram for this Markov chain is
\begin{center}
\tikzset{decorate sep/.style 2 args={decorate,decoration={shape backgrounds,shape=circle,shape size=#1,shape sep=#2}}}
\begin{tikzpicture}[start chain = going right,
-Triangle, every loop/.append style = {-Triangle}]
\foreach \i in {0,...,5}
\node[state, on chain] (\i) {\i};
\foreach
\i/\txt in {0/$P_{01}$,1/$P_{12}$/,2/$P_{23}$,3/$P_{34}$,4/$P_{45}$}
\draw let \n1 = { int(\i+1) } in
(\i) edge[bend left,"\txt",color=Periwinkle] (\n1);
\foreach
\i/\txt in {0/$P_{10}$,1/$P_{21}$/,2/$P_{32}$,3/$P_{34}$,4/$P_{54}$}
\draw let \n1 = { int(\i+1) } in
(\n1) edge[bend left,above, "\txt",color=Mahogany,auto=right] (\i);

\foreach
\i/\txt in {0/$P_{20}$,1/$P_{31}$/,2/$P_{42}$,3/$P_{53}$}
\draw let \n1 = { int(\i+2) } in
(\n1) edge[bend left,color=RawSienna,"\txt"] (\i);

\foreach \i/\txt in {1/$P_{11}$,2/$P_{22}$/,3/$P_{33}$,4/$P_{44}$,5/$P_{55}$}
\draw (\i) edge[loop above,color=NavyBlue, "\txt"] (\i);
\draw (0) edge[loop left, color=NavyBlue,"$P_{00}$"] (0);

\draw[decorate sep={1mm}{4mm},fill] (11,0) -- (12,0);
\end{tikzpicture}
\begin{tikzpicture}
\draw (4,-1.8) node
{\begin{minipage}{0.8\textwidth}
		\begin{center}\small
			\textbf{Type II Jacobi--Piñeiro's Markov chain diagram}
		\end{center}
\end{minipage}};
\end{tikzpicture}

\end{center}

\begin{coro}\label{coro:stochastic_JPII}
The explicit expressions the type~II multiple stochastic matrix coefficients are
\begin{align*}
P_{II,2n,2n+1}&= \frac{(\gamma +2 n+1) (\alpha +\gamma +2 n+1) (\beta +\gamma
+2 n+1)}{(\alpha +\gamma +3 n+1) (\alpha +\gamma +3 n+2)(\beta +\gamma +3 n+1)}, \\
P_{II,2n+1,2n+2}&= \frac{(\gamma +2 n+2) (\alpha +\gamma +2 n+2) (\beta +\gamma
+2 n+2)}{(\alpha +\gamma +3 n+3) (\beta +\gamma +3 n+2)
(\beta +\gamma +3 n+3)},\\
P_{II,2n,2n}=& \begin{multlined}[t][0.75\textwidth]
\frac{(\alpha +n) (\alpha +\gamma +2 n-1)(\beta +\gamma +2 n-1)}{ (\alpha+\gamma +3 n-1)(\alpha+\gamma +3 n)(\beta +\gamma +3 n-1)}\\
+\frac{n (\gamma +2 n-1) (\alpha +\gamma +2n-1)}{(\alpha+\gamma +3 n-1)(\beta +\gamma +3 n-2) (\beta +\gamma +3 n-1)}
\\+\frac{n (\gamma +2 n-1) (\beta +\gamma +2 n-2)}{(\alpha +\gamma +3 n-2) (\alpha+\gamma +3 n-1)(\beta +\gamma +3 n-2)},
\end{multlined}\\
P_{II,2n+1,2n+1}=& \begin{multlined}[t][.75\textwidth]
\frac{n (\beta +n) (-\alpha +\beta +n)}{(\beta -\alpha )(\beta +\gamma +3 n)}+\frac{n (\beta +n) (-\alpha +\beta
+n) (\alpha +\gamma +3 n+1)}{(\beta -\alpha )(\alpha -\beta +1) (\beta +\gamma +3 n) (\beta +\gamma +3
n+1)}\\-\frac{(n+1) (\alpha +n+1) (\alpha -\beta +n+1) (\beta +\gamma +3 n+2)}{(\beta -\alpha )(\alpha -\beta +1)
(\alpha +\gamma +3 n+2) (\alpha +\gamma +3 n+3)}\\+\frac{(n+1) (\beta +n+1) (-\alpha +\beta
+n+1)}{(\beta -\alpha )(\beta +\gamma +3 n+3)},
\end{multlined}\\
P_{II,2n,2n-1}&= \begin{multlined}[t][.75\textwidth]
\frac{n (\alpha -\beta +n) (\gamma +2 n-1)}{(\alpha +\gamma
+3 n-1) (\alpha +\gamma +3 n) (\alpha +\gamma +3
n+1)}\\+\frac{n (\beta +n) (\beta +\gamma +2 n)}{(\alpha
+\gamma +3 n) (\alpha +\gamma +3 n+1) (\beta +\gamma +3
n)}\\+\frac{n (\beta +n) (\alpha +\gamma +2 n+1)}{(\alpha
+\gamma +3 n+1) (\beta +\gamma +3 n) (\beta +\gamma +3
n+1)},
\end{multlined}\\
P_{II,2n+1,2n}&=\begin{multlined}[t][0.75\textwidth]
\frac{n(\alpha -\beta -n) (\beta +n) (\alpha +\gamma +3
n+1) }{(\alpha -\beta +1) (\beta +\gamma +3 n) (\beta +\gamma +3
n+1) (\beta +\gamma +3 n+2)}\\+\frac{(n+1) (\alpha -\beta +n+1)(\alpha +n+1) }{(\alpha -\beta +1) (\alpha +\gamma +3 n+2) (\alpha
+\gamma +3 n+3) }
\end{multlined}\\
P_{II,2n,2n-2}&= \frac{n (\alpha -\beta +n)(\alpha +n) }{(\alpha +\gamma +3
n-1) (\alpha +\gamma +3 n) (\alpha +\gamma +3 n+1)},\\
P_{II,2n+1,2n-1}&= \frac{n (-\alpha +\beta +n)(\beta +n)}{(\beta +\gamma +3 n)
(\beta +\gamma +3 n+1) (\beta +\gamma +3 n+2)}.
\end{align*}
\end{coro}

\begin{coro}[Large $n$ limit]\label{coro:large_n_limit_II}
The limit, for large $n$, for the stochastic matrix coefficients are
\begin{align*}
&\lim_{n\to \infty}P_{II,2n+1,2n+2}=\lim_{n\to \infty}P_{II,2n,2n+1}=\frac{8}{27}\approx 0.296,\\
&\lim_{n\to \infty}P_{II,2n+1,2n+1}=\lim_{n\to \infty}P_{II,2n,2n}=\frac{4}{9}\approx 0.444,\\
&\lim_{n\to \infty}P_{II,2n+1,2n}=\lim_{n\to \infty}P_{II,2n,2n-1}=\frac{2}{9}\approx 0.222,\\
&\lim_{n\to \infty}P_{II,2n+1,2n-1}=\lim_{n\to \infty}P_{II,2n,2n-2}=\frac{1}{27}\approx 0.037.
\end{align*}
\end{coro}
The corresponding transition diagram for large $n$ is
\begin{center}
\begin{tikzpicture}[bullet/.style={circle,inner sep=0.7ex},x=2cm,auto,bend angle=40]
\draw[->] (-3.5,0) -- (2.5,0);
\path (-3,0) (-2,0 ) node[bullet,draw=blue] (-2a) {} (-1,0 ) node[bullet,draw=blue] (-a) {}(0,0) node[bullet,fill=red] (0) {} (1,0) node[bullet,draw=blue] (a) {};
\foreach \Y [count=\X starting from -3] in {-3a,-2a,-a}
{\draw (\X,0.05) -- (\X,-0.1) node[midway,below left]{$n\X$};}
\foreach \Y [count=\X starting from 1] in {a,2a}
{\draw (\X,0.05) -- (\X,-0.1) node[below]{$n+\X$};}
\draw[-{Stealth[bend]},line width=0.6mm] (0) to[bend left] node{$\frac{8}{27}$} (a);
\draw[-{Stealth[bend]},line width=0.22mm] (0) to[bend left,above] node{$\frac{2}{9}$} (-a);
\draw[-{Stealth[bend]},line width=0.04mm] (0) to[bend left] node{$\frac{1}{27}$} (-2a);
\draw[-{Stealth[bend]},line width=0.85mm] (0) to[in=140,out=40,looseness=20,] node{$\frac{4}{9}$} (0);
\end{tikzpicture}
\begin{tikzpicture}
	\draw (4,-1.8) node
	{\begin{minipage}{0.8\textwidth}
			\begin{center}\small
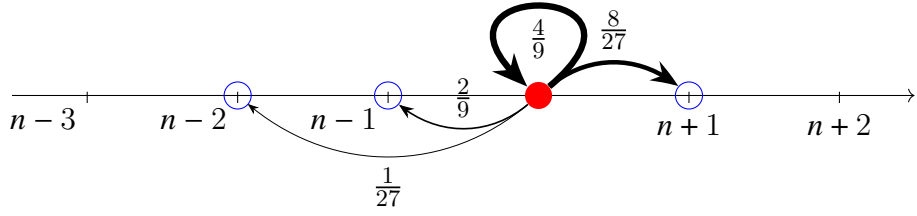
				\textbf{Asymptotic type II Jacobi--Piñeiro's Markov chain diagram}
			\end{center}
	\end{minipage}};
\end{tikzpicture}
\end{center}



\begin{rem}\label{rem:toeplitz_compact}
Observe that we have the splitting
\begin{align*}
P_{II}=T_{II}+\delta \, P_{II}
\end{align*}
with
\begin{align*}
	\hspace{-.35cm}
	T_{II} &:= \frac{1}{27}\left( \begin{NiceMatrix}[columns-width = 0.5cm]
		12 & 8 & 0 & 0 & 0 & \Cdots\\ 
		6 &12& 8& 0 & 0 & \Ddots \\ 
		1& 6& 12& 8&0&\Ddots\\ 
		0&1 &6 & 12& 8& \Ddots\\ 
		\Vdots&\Ddots&\Ddots&\Ddots&\Ddots&\Ddots
	\end{NiceMatrix}\right), &
	\delta { } JP_{II} &:=\left( \begin{NiceMatrix}[columns-width = 0.5cm]
		\delta { } P_{0,0} & \delta { } P_{0,1} & 0 & 0 & 0 & \Cdots\\ 
		\delta { } P_{1,0}& \delta P_{1,1} & \delta { } P_{1,2} & 0 & 0 & \Ddots \\ 
		\delta { } P_{2,0}& \delta { } P_{2,1}& \delta { } P_{2,2}& \delta { } P_{2,3}&0&\Ddots\\ 
		0& \delta { } P_{3,1}& \delta { } P_{3,2}& \delta { } P_{3,3} & \delta { } P_{3,4}& \Ddots\\ 
		\Vdots&\Ddots&\Ddots&\Ddots&\Ddots&\Ddots
	\end{NiceMatrix}\right),
\end{align*}
where $\delta P_{n,n+k}\xrightarrow[n\to \infty]{}0$ for $k\in\{-2,-1,0,1\}$, so that $\delta P_{II}$ represents a compact operator, and $T_{II}$ is a Toeplitz operator (notice that this Toeplitz matrix is a semi-stochastic matrix, but not stochastic because of the first two rows, in where there is a loss of probability $\frac{1}{27}$ in the second and of $\frac{7}{27}$ in the first).
\end{rem}

\subsection{Type~I Jacobi--Piñeiro's random walks}

In the next theorem we show how to turn stochastic the Jacobi matrix for the type~I given in~\eqref{eq:Jacobi_Jacobi_Piñeiro}.

\begin{teo}\label{teo:JPI_stochastic}
Let us assume for the Jacobi--Piñeiro system that $\alpha,\beta,\gamma>-1$, $\alpha\neq \beta$ and $|\alpha-\beta|<1$. Then, the semi-infinite matrix
\begin{align}\label{eq:stochastic_Jacobi_Piñeiro_I}
	P_I=\left( \begin{NiceMatrix}[columns-width = 0.5cm]
		P_{I,0,0} & P_{I,0,1}& P_{I,0,2} & 0 & 0 & \Cdots\\
		P_{I,1,0}&P_{I,1,1}& P_{I,1,2}& P_{I,1,3}& 0 & \Ddots \\
		0& P_{I,2,1}& P_{I,2,2} & P_{I,2,3} & P_{I,2,4}&\Ddots\\
		0 & 0&P_{I,3,2}& P_{I,3,3} & P_{I,3,4}& \Ddots\\
		\Vdots&\Ddots&\Ddots&\Ddots&\Ddots&\Ddots
	\end{NiceMatrix}\right)
\end{align}
with coefficients expressed in terms of the coefficients of the Jacobi matrix~\eqref{eq:Jacobi_Jacobi_Piñeiro} and the linear forms of type~I evaluated at $x=1$, $Q_{\vec\nu}(1)$ with $\vec\nu=(n+1,n), (n,n)$, for $n = 0, 1, \ldots$, as~follows
\begin{align*}
P_{I,2n,2n+2}&= \frac{Q_{(n+2,n+1)}(1)}{Q_{(n+1,n)}(1)}d_{n+1,n+1}, &P_{2n+1,2n+3}&=\frac{Q_{(n+2,n+2)}(1)}{Q_{(n+1,n+1)}(1)}d_{n+2,n+1},\\
P_{I,2n,2n+1}&= \frac{Q_{(n+1,n+1)}(1)}{Q_{(n+1,n)}(1)}c_{n+1,n}, &P_{I,2n+1,2n+2}&=\frac{Q_{(n+2,n+1)}(1)}{Q_{(n+1,n+1)}(1)} c_{n+1,n+1},\\
P_{I,2n,2n}&=b_{n,n}, &
P_{I,2n+1,2n+1}&=b_{n+1,n},\\
P_{I,2n+2,2n+1}&= \frac{Q_{(n+1,n+1)}(1)}{Q_{(n+2,n+1)}(1)} &
P_{I,2n+1,2n}&=\frac{Q_{(n+1,n)}(1)}{Q_{(n+1,n+1)}(1)},
\end{align*}
is a multiple stochastic matrix of type~I. 
\end{teo}

\begin{proof}
 According to~\cite{nikishin_sorokin} the system
 \begin{align*}
 \{x^\alpha , \ldots , x^{\nu_1-1+\alpha},x^\beta , \ldots , x^{\nu_2-1+\beta}\}
 \end{align*}
is a Chebyshev system in any closed interval of the positive semiaxis $\R_+=\{x\in\R: x>0\}$. Consequently, the linear form $Q_{\vec \nu}(x)= A_{\vec \nu,1}(x)x^\alpha+A_{\vec \nu,2}(x)x^\beta $ has at most $|\vec \nu|-1$ zeros in any closed interval $[a,b]\subset \R_+$. Thus, 
 the maximum number of zeros in $\R_+$ will be $|\vec \nu|-1$.
As $\{x^\alpha,x^\beta\}$ conforms an AT-system on $[0,1]$ it has $|\vec \nu|-1$ zeros in its interior, the open interval $(0,1)$, see~\cite{nikishin_sorokin,coussement,vanassche_ismail}. Therefore, $Q_{\vec \nu}(x)$ has no zeros for $x\geq 1$.

We now analyze the behavior of the linear form $Q_{\vec \nu}(x)$ for $x\to +\infty$. For $\vec \nu=(n+1,n)$ we have for $x\to +\infty$
\begin{align*}
A_{(n+1,n),1}(x)x^\alpha&=A_{(n+1,n),1,0}x^{n+\alpha}+O(x^{n+\alpha-1}), &
A_{(n+1,n),2}(x)x^\beta &=
A_{(n+1,n),2,0}x^{n+\beta-1}+O(x^{n+\beta-2}),
\end{align*}
 where $A_{(n+1,n),1,0}$ and $A_{(n+1,n),2,0}$ are the conductor coefficients, accompanying the leading terms, of the polynomials $A_{(n+1,n),1}(x)$ and $A_{(n+1,n),2}(x)$, respectively. Observing that
 \begin{align*}
 n+\alpha-(n+\beta-1)=\alpha-\beta+1>0,
 \end{align*}
 we see that for $x\to+\infty$ we have
 \begin{align*}
 Q_{(n+1,n)}(x)=A_{(n+1,n),1,0}x^{n+\alpha}+O(x^{n+\alpha-1}).
 \end{align*}
 According to Theorem~\ref{theorem:JPI} we have
 \begin{align*}
 A_{(n+1,n),1,0}=
 \frac{(2 n + 1+\alpha + \gamma )_{n}}{\Gamma (\alpha + n + 1) (\alpha - \beta + 1)_n}>0,
 \end{align*}
and, consequently, for $n\in\N$
\begin{align*}
\lim_{x\to+\infty}Q_{(n+1,n)}(x)=+\infty.
\end{align*}
For $n=0$, we have $Q_{(1,0)}=\frac{\Gamma(2+\alpha+\gamma)}{\Gamma(1+\gamma)}x^\alpha$, that is always positive. Hence $Q_{(n+1,n)}(x)>0$ for $x\geq 1$.

 For $\vec \nu=(n,n)$ we have for $x\to +\infty$
\begin{align*}
A_{(n,n),1}(x)x^\alpha&=A_{(n,n),1,0}x^{n-1+\alpha}+O(x^{n-2+\alpha}), &
A_{(n,n),1}(x)x^\beta &=
A_{(n,n),2,0}x^{n-1+\beta}+O(x^{n-2+\beta}).
\end{align*}
Hence, the dominant behavior at $+\infty$ of the linear form $Q_{(n,n)}(x)$ depends on whether $\alpha\lessgtr \beta$. Let us assume, in the first place, that $\alpha >\beta$. Then,
 for $x\to+\infty$ we have
\begin{align*}
Q_{(n,n)}(x)=A_{(n,n),1,0}x^{n+\alpha-1}+o(x^{n+\alpha-1}).
\end{align*}
According to Theorem~\ref{theorem:JPI} we have
\begin{align*}
A_{(n,n),1,0}= \frac{\Gamma( 3 n - 1+ \alpha + \gamma)}{\Gamma(n + \alpha)(\alpha - \beta )_n},
\end{align*}
that is positive for $\alpha >\beta$. Thus, for $n\in\{2,3,\dots\}$ we have
\begin{align*}
\lim_{x\to +\infty }Q_{(n,n)}(x)=+\infty.
\end{align*}
Finally, when $\beta >\alpha$,
for $x\to+\infty$ we have
\begin{align*}
Q_{(n,n)}(x)=A_{(n,n),2,0}x^{n+\beta-1}+o(x^{n+\beta-1}).
\end{align*}
Recalling that $A_{(n,n),2}^{\alpha,\beta}=A_{(n,n),1}^{\beta,\alpha}$ we use the previous result interchanging $\alpha$ and $\beta$ to get for $n\in\{2,3,\dots\}$ that
\begin{align*}
\lim_{x\to +\infty }Q_{(n,n)}(x)=+\infty,
\end{align*}
for $\beta>\alpha$. For $n=1$ we have $Q_{(1,1)}(x)=\frac{(2+\alpha+\gamma)(2+\beta+\gamma)}{\Gamma(1+\gamma)}(x^\alpha+x^\beta)$, which is always positive. Therefore, $Q_{(n,n)}(x)>0$ for $x\geq 1$. 

 We conclude that
\begin{align*}
\sigma_{I,l}=\frac{1}{Q^{(l)}(1)}>0,
\end{align*}
so that
\begin{align*}
P_I=\sigma_I { J } ^\top\sigma_I^{-1}
\end{align*}
is a multiple stochastic matrix of type~I, i.e.,
$P_I \, \1
= \1$.
\end{proof}

A byproduct of the previous proof  and Corollary \ref {corollary:3F2_linear_form} is
\begin{coro}
For $x\geq 1$, $\alpha,\beta,\gamma>-1$, $\alpha\neq \beta$ and $|\alpha-\beta|<1$, the generalized hypergeometric functions ${}_3F_2$  fulfills the following inequalities
	\begin{align*}
		\begin{multlined}[t][.95\textwidth]
			\frac{\Gamma (\alpha +\gamma +3 n-1) }{\Gamma (\alpha +n) (\alpha-\beta)_n}x^{\alpha-\beta}
			{}_{3}F_{2}\left[{\begin{array}{c}1-n,\;-\alpha-n+1 ,\;-\alpha +\beta -n+1\\-\alpha +\beta +1,\;-\alpha -\gamma -3 n+2\end{array}};\frac{1}{x}\right]
			\\ +
	\frac{\Gamma (\beta +\gamma +3 n-1) }{\Gamma (\beta +n) (\beta-\alpha)_n}
			{}_{3}F_{2}\left[{\begin{array}{c}1-n,\;-\beta-n+1 ,\;-\beta +\alpha -n+1\\-\beta +\alpha +1,\;-\beta -\gamma -3 n+2\end{array}};\frac{1}{x}\right]>0,\quad n \in\N,
		\end{multlined}
		\\
		\begin{multlined}[t][.95\textwidth] 
			\frac{\Gamma (\alpha +\gamma +3 n+2) }{n\Gamma (\alpha +n+1)
				(\alpha -\beta +1)_{n} }		x^{\alpha-\beta}{}_{3}F_{2}\left[{\begin{array}{c}-\alpha-n,\;-\alpha+\beta -n ,\;-n\\-\alpha +\beta ,\;-\alpha -\gamma -3 n\end{array}};\frac{1}{x}\right]
			\\
			>(-1)^{n}\frac{ \Gamma (\beta +\gamma +3 n) (\alpha +\gamma +3 n+1)}{ \Gamma (\beta +n)(\alpha -\beta -n+1)_{n+1}}{}_{3}F_{2}\left[{\begin{array}{c}1-n,\;-\beta -n+1 ,\;\alpha-\beta-n+1\\\alpha -\beta+2 ,\;-\beta -\gamma -3 n+1\end{array}};\frac{1}{x}\right], \quad n\in\{2,3,\dots\}.
		\end{multlined}
	\end{align*}
\end{coro}

The diagram for this Markov chain is
\begin{center}
\tikzset{decorate sep/.style 2 args={decorate,decoration={shape backgrounds,shape=circle,shape size=#1,shape sep=#2}}}
\begin{tikzpicture}[start chain = going right,
-Triangle, every loop/.append style = {-Triangle}]
\foreach \i in {0,...,5}
\node[state, on chain] (\i) {\i};
\foreach
\i/\txt in {0/$P_{01}$,1/$P_{12}$/,2/$P_{23}$,3/$P_{34}$,4/$P_{45}$}
\draw let \n1 = { int(\i+1) } in
(\i) edge[bend left,"\txt",below,color=Periwinkle] (\n1);
\foreach
\i/\txt in {0/$P_{10}$,1/$P_{21}$/,2/$P_{32}$,3/$P_{34}$,4/$P_{54}$}
\draw let \n1 = { int(\i+1) } in
(\n1) edge[bend left,below, "\txt",color=Mahogany,auto=right] (\i);

\foreach
\i/\txt in {0/$P_{02}$,1/$P_{13}$/,2/$P_{24}$,3/$P_{35}$}
\draw let \n1 = { int(\i+2) } in
(\i) edge[bend left,color=MidnightBlue,"\txt"](\n1);

\foreach \i/\txt in {1/$P_{11}$,2/$P_{22}$/,3/$P_{33}$,4/$P_{44}$,5/$P_{55}$}
\draw (\i) edge[loop below, color=NavyBlue,"\txt"] (\i);
\draw (0) edge[loop left, color=NavyBlue,"$P_{00}$"] (0);

\draw[decorate sep={1mm}{4mm},fill] (11,0) -- (12,0);
\end{tikzpicture}
\begin{tikzpicture}
	\draw (4,-1.8) node
	{\begin{minipage}{0.8\textwidth}
			\begin{center}\small
				\textbf{Type I Jacobi--Piñeiro's Markov chain diagram}
			\end{center}
	\end{minipage}};
\end{tikzpicture}
\end{center}

 \begin{pro}[Large $n$ limit for the dual Jacobi--Piñeiro Markov chains]\label{pro:JP_stochastic_dual}
The large $n$ limit of the Jacobi--Piñeiro stochastic matrices of type~I and II are the same after transposition, i.e.,
\begin{align*}
\lim_{n\to\infty} P_{I,n,n+k}
&=\lim_{n\to\infty}P_{II,n+k,n}, & k\in\{-2,-1,0,1\}.
\end{align*}
 \end{pro}
\begin{proof}
According to Theorem~\ref{teo:I_and_II} we need to show that
\begin{align}\label{eq:other_relations}
\frac{B^{(n-k)}(1)Q^{(n-k)}(1)}{B^{(n)}(1)Q^{(n)}(1)}&\underset{n\to \infty}{\longrightarrow} 1, &k&=2,1,-1.
\end{align}
%
From~\eqref{eq:Ben1} we directly deduce that
\begin{align*}
\frac{B^{(2n+1)}(1)}{B^{(2n)}(1)}&=
\frac{(\gamma+1+2n)(\alpha +\gamma +2n+1)(\beta +\gamma +2n+1)}{(\alpha +\gamma +3n+1)(\beta +\gamma +3n+1)(\beta +\gamma +3n+2)}\xrightarrow[n\to\infty]{}\frac{8}{27},\\
\frac{B^{(2n+2)}(1)}{B^{(2n+1)}(1)}&=
\frac{(\gamma+2+2n)(\beta +\gamma +2n+2)(\alpha +\gamma +2n+2)}{(\beta +\gamma +3n+3)(\alpha +\gamma +3n+2)(\alpha +\gamma +3n+3)}\xrightarrow[n\to\infty]{}\frac{8}{27}.
\end{align*}
Now, following the spirit of Theorem~\ref{teo:removing_roots}, from~\eqref{eq:CD_JP}
we get\begin{align*}
Q^{(n-1)}(1)B^{(n)}(1)=Q^{(n)}(1)\big(J_{n,n-2}B^{(n-2)}(1)+
J_{n,n-1}B^{(n-1)}(1))+Q^{(n+1)}(1)J_{n+1,n-1}B^{(n-1)}(1),
\end{align*}
so that we have for the linear forms of type~I the following lower degree homogeneous linear recurrence
\begin{align*}
-Q^{(n-1)}(1)+a_nQ^{(n)}(1)+b_nQ^{(n+1)}(1)=0,
\end{align*}
with
\begin{align*}
a_n&=J_{n,n-2}\frac{B^{(n-2)}(1)}{B^{(n)}(1)}+
J_{n,n-1}\frac{B^{(n-1)}(1)}{B^{(n)}(1)}\xrightarrow[n\to\infty]{}
\frac{4^3}{27^3}\frac{27^2}{8^2}+3\frac{4^2}{27^2}\frac{27}{8}=\frac{7}{27},\\
b_n&=J_{n+1,n-1}\frac{B^{(n-1)}(1)}{B^{(n)}(1)}\xrightarrow[n\to\infty]{}\frac{4^3}{27^3}\frac{27}{8}=\frac{8}{729},
\end{align*}
where we have used the previous result, $\lim_{n\to\infty}\frac{B^{(n+1)}(1)}{B^{(n)}(1)}=\frac{8}{27}$, and Corollary~\ref{coro:limit_Jacobi_Jacobi-Piñeiro}. The characteristic polynomial is
\begin{align*}
-1+\frac{7}{27}r+\frac{8}{729}r^2=\frac{8}{729}(r+27)\Big(r-\frac{27}{8}\Big).
\end{align*}
Therefore, from Poincaré's theorem, having its characteristic roots $\{-27,\frac{27}{8}\}$ distinct absolute value, as the linear forms of type~I are positive at $1$, we get
\begin{align*}
\lim_{n\to\infty}\frac{Q^{(n+1)}(1)}{Q^{(n)}(1)}=\frac{27}{8}
\end{align*}
and~\eqref{eq:other_relations} is satisfied.
\end{proof}

The corresponding asymptotic transition diagram for large $n$ is
\begin{center}
\begin{tikzpicture}[bullet/.style={circle,inner sep=0.7ex},x=2cm,auto,bend angle=40]
\draw[->] (-2.5,0) -- (3.5,0);
\path (-2,0 ) (-2a) {} (-1,0 ) node[bullet,draw=blue] (-a) {}(0,0) node[bullet,fill=red] (0) {} (1,0) node[bullet,draw=blue] (a) {} (2,0 ) node[bullet,draw=blue] (2a){} (3,0);
\foreach \Y [count=\X starting from -2] in {-2a,-a}
{\draw (\X,0.05) -- (\X,-0.1) node[midway,below left]{$n\X$};}
\foreach \Y [count=\X starting from 1] in {a,2a,3a}
{\draw (\X,0.05) -- (\X,-0.1) node[below]{$n+\X$};}
\draw[-{Stealth[bend]},line width=0.6mm] (0) to[bend left] node{$\frac{8}{27}$} (-a);
\draw[-{Stealth[bend]},line width=0.22mm] (0) to[bend left,below] node{$\frac{2}{9}$} (a);
\draw[-{Stealth[bend]},line width=0.04mm] (0) to[bend left] node{$\frac{1}{27}$} (2a);
\draw[-{Stealth[bend]},line width=0.85mm] (0) to[in=140,out=40,looseness=20,] node{$\frac{4}{9}$} (0);
\end{tikzpicture}
\begin{tikzpicture}
	\draw (4,-1.8) node
	{\begin{minipage}{0.8\textwidth}
			\begin{center}\small
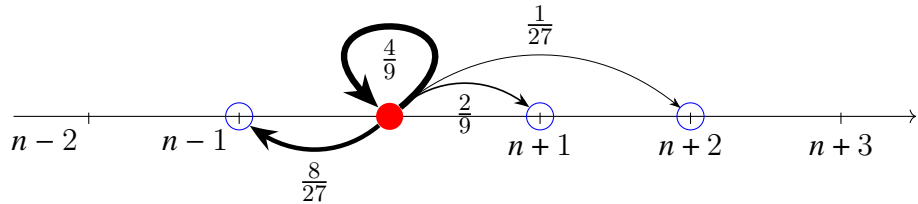
				\textbf{Asymptotic type I Jacobi--Piñeiro's Markov chain diagram}
			\end{center}
	\end{minipage}};
\end{tikzpicture}
\end{center}

\begin{pro}[Recurrent and transient Jacobi--Piñeiro random walks]\label{pro:recurrentJP}
Both dual Jacobi--Piñeiro random walks are recurrent whenever $-1<\gamma<0$ and transient for $\gamma\geq 0$.
\end{pro}
\begin{proof}
Is a direct consequence of Theorem \ref{teo:recurrent_state} as the divergence of the integral coincides with the divergence of $\int_a^1(1-x)^{\gamma-1}\d x$, for $0<a<1$, that happens for $\gamma<0$.
\end{proof}

\begin{con}[Recurrent Jacobi--Piñeiro Markov chains are null recurrent]
As there are no mass points, an according to our previous conjectures, $\boldsymbol{\kappa}_\lambda\not\in\ell_1$ and the Markov chains when recurrent ($-1<\gamma<0$) are null recurrent, i.e. the expected return times are infinite, and consequently not~ergodic.
\end{con}

\subsection{Two type~II examples: recurrent and transient random walks}
For $\alpha = -\frac{1}{4}, \beta = \gamma= -\frac{1}{2}$, that gives a recurrent random walk, we get the following transition matrix coefficients
\begin{align*}
P_{II,2n,2n+1}&= \frac{4(4n+1)(8n+1)}{3(144n^2+72n+5)}, &
P_{II,2n+1,2n+2}&=\frac{(2n+1)(8n+5)}{6(9n^2+9n+2)}, & n \geq 0 ,\\
P_{II,2n,2n} &=\frac{96n^2+16n-17}{6\left(36n^2+3n-5\right)}, &
P_{II,2n+1,2n+1} &= \frac{32n^2+32n+7}{72n^2+78n+20}, & n \geq 1, \\
P_{II,2n,2n-1}&=\frac{576n^3-552n^2+94n+7}{6(432n^3-360n^2+51n+7)},
&
P_{II,2n+1,2n}&=\frac{192 n^3 + 104 n^2 - 22 n - 11}{8 (108 n^3 + 45 n^2 - 12 n - 5)}, & n \geq 1, \\
P_{II,2n,2n-2}&= \frac{4 n (4 n+1)}{3 (12 n-7) (12 n+1)},&
P_{II,2n+1,2n-1}&=\frac{8 n^2-6 n+1}{24 (9 n^2-1)}, & n \geq 1,
\end{align*}
with $P_{II,0,0} =\frac{3}{5}$ and $P_{II,1,0} = \frac 2 5$.
Hence, the corresponding transition matrix looks as~follows
\begin{align*}
	\hspace*{-1cm}P_{II}
	&=\left(
	\begin{NiceMatrix}[columns-width = 0.5cm]
		\frac{3}{5} & \frac{2}{5} & 0 & 0 & 0 &0&0& \Cdots \\
		\frac{4}{15} & \frac{19}{60} & \frac{5}{12} & 0 & 0 &0&0& \Ddots \\
		\frac{4}{39}& \frac{25}{156}& \frac{95}{204} & \frac{60}{221} &0&0&0& \Ddots \\
		0& \frac{1}{64} & \frac{263}{1088} & \frac{71}{170} & \frac{13}{40} &0&0& \Ddots \\
		0&0&\frac{24}{425}& \frac{173}{850}&\frac{133}{290}&\frac{204}{725} &0&\Ddots \\
		0&0&0&\frac{1}{40}&\frac{271}{1160}& \frac{199}{464} & \frac{5}{16} & \Ddots \\
		\Vdots&\Ddots&\Ddots&\Ddots&\Ddots&\Ddots&\Ddots&\Ddots
	\end{NiceMatrix}\right)\approx\left(\scriptsize\begin{NiceMatrix}[columns-width = 0.5cm]
		0.6000& 0.4000 & 0 & 0 & 0 &0&0& \Cdots \\
		0.2667 & 0.3167 & 0.4167 & 0 & 0 &0&0& \Ddots \\
		0.1026& 0.1603& 0.4657& 0.2715 &0&0&0& \Ddots \\
		0& 0.0156 & 0.2417 &0.4176& 0.3250 &0&0& \Ddots \\
		0&0&0.0565& 0.2035&0.4586&0.2814 &0&\Ddots \\
		0&0&0&0.0250&0.2336& 0.4289 & 0.3125 & \Ddots \\
		\Vdots&\Ddots&\Ddots&\Ddots&\Ddots&\Ddots&\Ddots&\Ddots
	\end{NiceMatrix}\right).
\end{align*}
The approximation given in terms of numbers with decimals have four significant digits, i.e., if we think of probabilities percentages
 the number $0.4647$ represents a $46.47\%$ of probability and the error is less that $0.01\%$.

For $ \alpha=- \frac{1}{4} ,\beta =- \frac{1}{2},\gamma= \frac{1}{2}$, that gives a transient random walk, we get the following transition matrix coefficients
\begin{align*}
P_{II,2n,2n+1}&= \frac{2(2n+1)(8n+5)}{3(36n^2+27n+5)}, & P_{II,2n+1,2n+2}&=\frac{(4n+5)(8n+9)}{3(36n^2+63n+26)}, & n \geq 0, \\
P_{II,2n,2n} &=\frac{32n^2+16n+1}{72n^2+30n+2}, &P_{II,2n+1,2n+1}&=\frac{96n^2+128n+25}{6(36n^2+51n+13)}, & n \geq 1 ,\\
P_{II,2n,2n-1}&=\frac{576 n^3 + 120 n^2 - 74 n - 5}{6 (432 n^3 + 360 n^2 + 87 n + 5)}, &
P_{II,2n+1,2n}&=\frac{192 n^3 + 328 n^2 + 146 n + 17}{8(3n+1)(3n+2)(12n+13)} , & n \geq 1 , \\
P_{II,2n,2n-2}&= \frac{4n(4n+1)}{3\left(144n^2+72n+5\right)},& P_{II,2n+1,2n-1}&=\frac{8n^2-6n+1}{24(9n^2+9n+2)}, & n \geq 1,
\end{align*}
and $P_{II,0,0} =\frac{1}{3}$ and $P_{II,1,0} = \frac 2 3$.
The corresponding transition matrix is
\begin{align*}
	\hspace*{-1.2cm}P_{II}=\left( \begin{NiceMatrix}[columns-width = 0.5cm]
		\frac{1}{3} & \frac{2}{3}& 0 & 0 & 0 &0&0& \Cdots\\ 
		\frac{4}{39} &\frac{25}{78} & \frac{15}{26} & 0&0 &0&0& \Ddots \\ 
		\frac{20}{663}& \frac{617}{5304}& \frac{49}{104} & \frac{13}{34} &0&0&0&\Ddots\\
		0& \frac{1}{160} &\frac{683}{4000}& \frac{83}{200} & \frac{51}{125} &0&0& \Ddots\\
		0&0&\frac{24}{725}& \frac{47}{290}&\frac{23}{50}&\frac{10}{29} &0&\Ddots\\
		0&0&0&\frac{1}{64}&\frac{451}{2368}& \frac{95}{222} & \frac{325}{888} &\Ddots \\
		\Vdots&\Ddots&\Ddots&\Ddots&\Ddots&\Ddots&\Ddots&\Ddots
	\end{NiceMatrix}\right)\approx
	\left( \scriptsize\begin{NiceMatrix}[columns-width = 0.5cm]
		0,3333 & 0.6666& 0 & 0 & 0 &0&0& \Cdots\\ 
		0.1026 &0.3205& 0.5769& 0&0 &0&0& \Ddots \\ 
		0.0302&0.1163& 0.4712 & 0.3824&0&0&0&\Ddots\\
		0& 0.0062 &0.1707& 0.4150 & 0.4080 &0&0& \Ddots\\
		0&0&0.0331& 0.1621&0.4600&0.3448&0&\Ddots\\
		0&0&0&0.0156&0.1905&0.4279&0.3660 &\Ddots \\
		\Vdots&\Ddots&\Ddots&\Ddots&\Ddots&\Ddots&\Ddots&\Ddots
	\end{NiceMatrix}\right).
\end{align*}

Inspection of both transition matrices, we see that the probabilities to go to the left are bigger in the recurrent situation, with $\gamma=-\frac{1}{2}$, than in the transient case example with~$\gamma=\frac{1}{2}$.

The splitting described in Remark~\ref{rem:toeplitz_compact} the asymptotic semi-stochastic Toeplitz matrix is
\begin{align*}
	T_{II} &\approx\left( \begin{NiceMatrix}[columns-width = 0.5cm]\small
		0.4444 & 0.2963 & 0 & 0 & 0 &0&0& \Cdots\\ 
		0.2222&0.4444& 0.2963& 0 & 0 &0&0& \Ddots \\ 
		0.0370& 0.2222& 0.4444& 0.2963&0&0&0&\Ddots\\ 
		0&0.0370 &0.2222& 0.4444& 0.2963&0&0& \Ddots\\ 
		0&0&0.0370 &0.2222& 0.4444& 0.2963&0& \Ddots\\
		0&0&0&0.0370 &0.2222& 0.4444& 0.2963& \Ddots\\
		\Vdots&\Ddots&\Ddots&\Ddots&\Ddots&\Ddots&\Ddots&\Ddots
	\end{NiceMatrix}\right).
\end{align*}
The compact correction for the recurrent Markov chain, $ \alpha=- \frac{1}{4} ,\beta =- \frac{1}{2},\gamma= -\frac{1}{2}$, is
 \begin{align*}
	\delta P_{II} &\approx\left( \begin{NiceMatrix}[columns-width = 0.5cm]\small
		0.1556 & 0.1037& 0 & 0 & 0 &0&0& \Cdots\\ 
		0.0444&-0.1278& 0.1204& 0 & 0 &0&0& \Ddots \\ 
		0.0655& -0.0620& 0.0212& -0.0248&0&0&0&\Ddots\\ 
		0&-0.0214&0.0195& -0.0268& 0.0287&0&0& \Ddots\\ 
		0&0&0.0194 &-0.0187& 0,0142& -0.0149&0& \Ddots\\
		0&0&0&-0.0120&0.0114& -0.0156& 0.0162& \Ddots\\
		\Vdots&\Ddots&\Ddots&\Ddots&\Ddots&\Ddots&\Ddots&\Ddots
	\end{NiceMatrix}\right),
\end{align*}
and for the transient Markov chain, $ \alpha=- \frac{1}{4} ,\beta =- \frac{1}{2},\gamma= \frac{1}{2}$ is
 \begin{align*}
	\delta P_{II} &\approx\left( \begin{NiceMatrix}[columns-width = 0.5cm]\small
		-0.1111& 0.3704& 0 & 0 & 0 &0&0& \Cdots\\ 
		-0.1197&-0.1239& 0.2806& 0 & 0 &0&0& \Ddots \\ 
		-0.0069& -0.1059& 0.0267&0.0861&0&0&0&\Ddots\\ 
		0&-0.0308&-0.0515& -0.0294& 0.1117&0&0& \Ddots\\ 
		0&0&-0.0039&-0.0602& 0.0156&0.0485 &0& \Ddots\\
		0&0&0&-0.0214&-0.0318&-0.0165&0.0697& \Ddots\\
		\Vdots&\Ddots&\Ddots&\Ddots&\Ddots&\Ddots&\Ddots&\Ddots
	\end{NiceMatrix}\right),
\end{align*}

\subsection{Two type~I examples: recurrent and transient random walks}
For $\alpha = -\frac{1}{4}, \beta = \gamma= -\frac{1}{2}$, that gives a recurrent random walk, we get the following approximate transition matrix in decimal form with a precision of four significant digits (now is not possible to find closed rational expressions and several sums involving the Euler's Gamma function are required)
\begin{align*}
	P_{I}\approx\left( \begin{NiceMatrix}[columns-width = 0.5cm]
		0.6000& 0.2531& 0.1469& 0 & 0 &0&0&0& \Cdots\\ 
		0.4215&0.3167& 0.2419& 0.0199 & 0 &0&0&0& \Ddots \\ 
		0& 0.2760& 0.4657& 0.2036&0.0547&0&0&0&\Ddots\\
		0& 0 &0.3223&0.4176& 0.2341&0.0260&0&0& \Ddots\\
		0&0&0& 0.2826&0.4586&0.2110&0.0478&0&\Ddots\\
		0&0&0&0&0.3115& 0.4289 & 0.4289&0.0291&\Ddots \\
		\Vdots&\Ddots&\Ddots&\Ddots&\Ddots&\Ddots&\Ddots&\Ddots&\Ddots&\Ddots
	\end{NiceMatrix}\right).
\end{align*}

For $ \alpha=- \frac{1}{4} ,\beta =- \frac{1}{2},\gamma= \frac{1}{2}$, that gives a transient random walk, we get the following approximate expression for the transition matrix 
\begin{align*}
	P_{I}\approx\left( \begin{NiceMatrix}[columns-width = 0.5cm]
		0.3333& 0.3198& 0.3469& 0 & 0 &0&0&0& \Cdots\\ 
		0.2138&0.3205& 0.4289& 0.0368& 0 &0&0&0& \Ddots \\ 
		0& 0.1565& 0.4711& 0.2726&0.0998&0&0&0&\Ddots\\
		0& 0 &0.2395&0.4150& 0.3061&0.0394&0&0& \Ddots\\
		0&0&0& 0.2160&0,4600&0.2542&0.0697&0&\Ddots\\
		0&0&0&0&0.2583& 0.4279& 0.2746&0.0391&\Ddots \\
		\Vdots&\Ddots&\Ddots&\Ddots&\Ddots&\Ddots&\Ddots&\Ddots&\Ddots
	\end{NiceMatrix}\right).
\end{align*}

The splitting described in Remark~\ref{rem:toeplitz_compact} the asymptotic semi-stochastic Toeplitz matrix is
\begin{align*}
	T_{I} &\approx\left( \begin{NiceMatrix}[columns-width = 0.5cm]\small
		0.4444 & 0.2222 & 0.0370 & 0 & 0 &0&0& 0&\Cdots\\ 
		0.2963&0.4444& 0.2222& 0.0370& 0 &0&0&0& \Ddots \\ 
		0& 0.2963& 0.4444& 0.2222&0.0370 &0&0&0&\Ddots\\ 
		0&0 &0.2963& 0.4444& 0.2222&0.0370 &0& 0&\Ddots\\ 
		0&0&0&0.2963& 0.4444& 0.2222&0.0370 &0& \Ddots\\
		0&0&0&0 &0.2963& 0.4444& 0.2222& 0.0370 &\Ddots\\
		\Vdots&\Ddots&\Ddots&\Ddots&\Ddots&\Ddots&\Ddots&\Ddots&\Ddots
	\end{NiceMatrix}\right).
\end{align*}
The compact correction for the recurrent Markov chain, $ \alpha=- \frac{1}{4} ,\beta =- \frac{1}{2},\gamma= -\frac{1}{2}$, is
\begin{align*}
	\delta P_{I} &\approx\left( \begin{NiceMatrix}[columns-width = 0.5cm]\small
		0.1556 & 0.0308 & 0.1099& 0 & 0 &0&0& 0&\Cdots\\ 
		0.1252&-0.1278&0.0197& -0.0172& 0 &0&0&0& \Ddots \\ 
		0& -0.0203& 0.0212& -0.0186&0.0177 &0&0&0&\Ddots\\ 
		0&0 &0.0260& -0.0268& 0.0119&-0.0111&0& 0&\Ddots\\ 
		0&0&0&-0.0137& 0.0142& -0.0112&0.0108 &0& \Ddots\\
		0&0&0&0 &0.0152& -0.0156& 0.0083& -0.0079&\Ddots\\
		\Vdots&\Ddots&\Ddots&\Ddots&\Ddots&\Ddots&\Ddots&\Ddots&\Ddots
	\end{NiceMatrix}\right).
\end{align*}
and for the transient Markov chain, $ \alpha=- \frac{1}{4} ,\beta =- \frac{1}{2},\gamma= \frac{1}{2}$ is
\begin{align*}
	\delta P_{I} &\approx\left( \begin{NiceMatrix}[columns-width = 0.5cm]\small
		-0.1111& 0.0976 & 0.3098& 0 & 0 &0&0& 0&\Cdots\\ 
		-0.0825&-0.1239&0.2067& -0.0003& 0 &0&0&0& \Ddots \\ 
		0& -0.1398& 0.0267& 0.0504&0.0628&0&0&0&\Ddots\\ 
		0&0 &-0.0568& -0.0294& 0.0839&0.0024&0& 0&\Ddots\\ 
		0&0&0&-0.0803& 0.0156& 0.0320&0.0327&0& \Ddots\\
		0&0&0&0 &-0.0380& -0.0165& 0.0524& 0.0021&\Ddots\\
		\Vdots&\Ddots&\Ddots&\Ddots&\Ddots&\Ddots&\Ddots&\Ddots&\Ddots
	\end{NiceMatrix}\right).
\end{align*}
Let us mention that the convergence is really slow, and that in order to in the compact matrix correction of less a $0.1\%$ probability one needs to go the row $150$, and for a probability less that $0.01\%$ beyond the row $500$.

\section*{Conclusions and outlook}

The close relation among stochastic processes and orthogonal polynomials is well known since the 1930's. In particular, the Karlin--McGregor representation formula is an exceptional result illustrating these links. This formula gives the iterated probabilities and first passage probabilities in terms of an integral involving orthogonal polynomials determined by the measure 
fixed by the Markov matrix of the chain~\cite{KmcG}. This rich network of interrelations has lead to a fruitful collaboration of both research communities. In particular, see~\cite{Grunbaum1,Kovchegov}, there has been attempts to go beyond the birth and death Markov chains and processes (with tridiagonal stochastic matrices). In this paper we have shown how the theory of multiple orthogonal polynomials of types I and II supplies a framework to extend these connections to Markov chains beyond birth and death chains (with multidiagonal stochastic matrices) and illustrated the method for constructing random walks associated with Jacobi--Piñeiro multiple orthogonal polynomials.

The technique and strategy presented in this paper may be applied to other examples. In particular, we are working in the following cases. The Gauss hypergeometric multiple orthogonal polynomials discussed in~\cite{lima_loureiro} provides a random walk with the same asymptotics as our Jacobi--Piñeiro random walks, and also provides examples of Markov chains with ignored states (semi-stochastic) which are uniform. In this direction, we believe that the Nikishin star system studied in~\cite{abey,abey2} will be useful. If we look to mixed multiple orthogonal polynomials we will get banded stochastic matrices, with $N$ subdiagonals and~$M$ superdiagonals, describing more general Markov chains that the ones connected with non mixed multiple orthogonality. In this respect, the results on the zeros of mixed systems in~\cite{fidalgo} will allow to apply our normalization technique whenever the Jacobi matrix is non negative and weights are of an~AT type. We are studying such mixed multiple Jacobi type orthogonal polynomials on the step line.

More speculative avenues are given within multivariate orthogonal polynomials, see~\cite{multivariate,ariznabarreta_manas,ariznabarreta_manas2} and Chapter~2 in~\cite{Ismail2} by Yuan Xu; whenever a non-negative Jacobi matrix is provided, one needs to seek for methods in order to get stochastic or semi-stochastic matrices from~it. This will probably provide Markov matrices with increasing size blocks along the diagonal. 
Notice that already Karlin and McGregor~\cite{KmcG_multivariate} and Milch~\cite{Milch} discussed interesting examples of multivariate Hahn and Krawtchouk polynomials related to growth of birth and death processes. See also~\cite{Fernandez-de_la_Iglesia} for quasi birth and death process in relation with bivariate orthogonal polynomials.

\medskip


\begin{center}
\sc Acknowledgements
\end{center}

We would like to thank W~Van~Assche for interesting conversations on $2$-othogonality and its relations with multiple orthogonality, and to FA~Grünbaum for inspiring conversations regarding the Karlin--McGregor representation formula and its possible extensions. Finally, we reckon JMR~Parrondo for comments and encouragement in different moments of the work.

\appendix

\section*{Appendix: An algorithmic approach for type~II multiple stochastic matrices}

Let us assume that the triple $(\vec n,\vec w,\mu)$ is such that the Jacobi matrix is a nonnegative bounded semi-infinite matrix.
A Jacobi matrix is subject to the two following transformations
\begin{enumerate}

\item \emph{\textbf{Scaling the independent variable.}}
One can rescale the independent variable, so that variable
$\hat x=\lambda^{-1} x$, for some $\lambda\in\R\setminus\{ 0 \}$, and get new polynomials $\hat B(\hat x)=B(\lambda \hat x)$.
Consequently,
\begin{align*}
 \hat { J } \hat B(\hat x)&=\hat x \hat B(\hat x), & { J } &=\lambda \hat { J } .
\end{align*}
\item \emph{\textbf{Scaling the dependent variables. }}
Given a diagonal matrix $\sigma=\diag(\sigma_0,\sigma_1,\ldots)$ we can rescale dependent variables, and get a new vector of polynomials $\hat B= \sigma B$ which amounts to
\begin{align*}
 \hat { J } \hat B( x)&= x \hat B( x), & \hat { J } _{k,l}&= \frac{\sigma_k}{\sigma_l} { J } _{k,l}.
\end{align*}
\end{enumerate}

\begin
{nteo}[Scaling transformation]
\label{pro:sigma_algortihm}
Given a nonnegative bounded Jacobi matrix $ { J } $ there always exist scaling transformations of the independent and dependent variables, unique up to a overall factor, such that the transformed Jacobi matrix $P$ is stochastic, i.e. $0\leq P_{i,j}\leq 1$ and $\sum\limits _{j=0}^{i+1} P_{i,j}=1$, for all $i\in\N_0$.

If $\lambda \geq\| { J } \|_\infty$, we have
\begin{align*}
P=\frac{1}{\lambda} \sigma { J } \, \sigma^{-1}
\end{align*}
with $\sigma=\diag(\sigma_0,\sigma_1,\ldots)$. Moreover,
if $\sigma_0>0$ we have
\begin{align*}
\sigma_0\geq \sigma_1\geq \sigma_2\geq \cdots> 0,
\end{align*}
and if $\sigma_0<0$ we have
\begin{align*}
\sigma_0\leq \sigma_1\leq \sigma_2\leq \cdots< 0.
\end{align*}
All $\frac{\sigma_n}{\sigma_{n-1}}$, $n\in\N$, are expressed as continued fractions of the coefficients of the matrix $ { J } $. Moreover, $\sigma_n=f_n ( { J } )\sigma_0$, for certain rational functions on the coefficients of the Jacobi matrix.
\end{nteo}

\begin{proof}
Given a nonnegative bounded Jacobi matrix, we scale the independent variable $\hat x=\frac{x}{\lambda}$, with $\lambda\geq \| { J } \|_\infty$, so that the new Jacobi matrix is
\begin{align*}
\hat { J } =\frac{ { J } }{\lambda},
\end{align*}
with $\| \hat { J } \|_\infty \leq 1$ ; i.e., $ \hat { J } $ is a semi-stochastic matrix, that is a non negative matrix fulfilling
\begin{align*}
\sum_{j=0}^\infty \hat { J } _{i,j}&\leq 1, & i&\in\N_0.
\end{align*}
Now we transform the dependent variables to get a stochastic matrix $P$, which is a nonnegative matrix band matrix such that
\begin{align*}
\sum_{j=0}^\infty P_{i,j}&= 1, & i&\in\N_0.
\end{align*}
We have to determine a sequence $\{\sigma_k\}_{k\in\N_0}\subset \R_+$ such that
\begin{align*}
\sum_{j=0}^\infty \frac{\sigma_i}{\sigma_j} \hat { J } _{i,j}=1
\end{align*}
supposing that
\begin{align*}
\sum_{j=0}^\infty \hat { J } _{i,j}\leq 1
\end{align*}
with $ \hat { J } _{i,j}\geq 0$, and consequently $0\leq \hat { J } _{i,j}\leq 1$. In particular, as $ { J } _{i,i+1}=1$ we have $\hat { J } _{i,i+1}=\| { J } \|_\infty^{-1}\neq 0$, $i\in\N_0$

For $i=0$, we have the constraint $\hat { J } _{0,0}+\hat { J } _{0,1}\leq 1$ and we need to find $\sigma_0$ and $\sigma _1$ such that
\begin{align*}
\hat { J } _{0,0}+ \frac{\sigma_0}{\sigma_1}\hat { J } _{0,1}=1,
\end{align*}
that we solve for $ \frac{\sigma_1}{\sigma_0}$ to get
\begin{align*}
\frac{\sigma_1}{\sigma_0}=\frac{\hat { J } _{0,1}}{1-\hat { J } _{0,0}}\leq \frac{\hat { J } _{0,1}}{\hat { J } _{0,1}}=1.
\end{align*}
Thus, the first condition determines the ratio $ \frac{\sigma_1}{\sigma_0}$, which satisfies
$ 0<\frac{\sigma_1}{\sigma_0}\leq 1$ ($\frac{\sigma_1}{\sigma_0}\neq 0$ because $\hat { J } _{0,1}\neq 0$).

The next condition~reads
\begin{align*}
\frac{\sigma_1}{\sigma_0}\hat { J } _{1,0}+\hat { J } _{1,1}+ \frac{\sigma_1}{\sigma_2}\hat { J } _{1,2}=1
\end{align*}
so that
\begin{align*}
\frac{\sigma_2}{\sigma_1}=\frac{\hat { J } _{1,2}}{1-\frac{\sigma_1}{\sigma_0}\hat { J } _{1,0}-\hat { J } _{1,1}}\leq\frac{\hat { J } _{1,2}}{1-\hat { J } _{1,0}-\hat { J } _{1,1}}\leq\frac{\hat { J } _{1,2}}{\hat { J } _{1,2}}.
\end{align*}
Thus, the second condition determines uniquely the ratio $\frac{\sigma_2}{\sigma_1}$ that satisfies $ 0<\frac{\sigma_2}{\sigma_1}\leq 1$ ($\frac{\sigma_2}{\sigma_1}\neq 0$ because $\hat { J } _{1,2}\neq 0$).

In fact, by this procedure we are going to find all the ratios $\frac{\sigma_{n+1}}{\sigma_n}$ satisfying $ 0<\frac{\sigma_{n+1}}{\sigma_n}\leq 1$. Let us prove it using induction. Assume that all ratios $\frac{\sigma_{j+1}}{\sigma_j}$, $j\in\{0, \ldots , n-1\}$ have been determined, and that all of them fulfill $0<\frac{\sigma_{j+1}}{\sigma_j}\leq 1$.
If we assume that $n<N$, then we~have
\begin{align*}
\frac{\sigma_n}{\sigma_0}\hat { J } _{n,0}+ \frac{\sigma_n}{\sigma_1}\hat { J } _{n,1}+\cdots+\hat { J } _{n,n}+\frac{\sigma_n}{\sigma_{n+1}}\hat { J } _{n,n+1}=1
\end{align*}
so that
\begin{align*}
\frac{\sigma_{n+1}}{\sigma_n}=\frac{\hat { J } _{n,n+1}}{1-\frac{\sigma_n}{\sigma_0}\hat { J } _{n,0}- \frac{\sigma_n}{\sigma_1}\hat { J } _{n,1}-\cdots-\hat { J } _{n,n}}.
\end{align*}
But
\begin{align*}
\frac{\sigma_n}{\sigma_0}=\frac{\sigma_n}{\sigma_{n-1}}\frac{\sigma_{n-1}}{\sigma_{n-2}}\cdots\frac{\sigma_1}{\sigma_0}\leq 1,\\
\frac{\sigma_n}{\sigma_1}=\frac{\sigma_n}{\sigma_{n-1}}\frac{\sigma_{n-1}}{\sigma_{n-2}}\cdots\frac{\sigma_2}{\sigma_1}\leq 1,
\end{align*}
and, in general, we deduce that $\frac{\sigma_n}{\sigma_i}\leq 1$ for $i\in\{0,\ldots,n-1\}$ so that
\begin{align*}
\frac{\sigma_{n+1}}{\sigma_n}=\frac{\hat { J } _{n,n+1}}{1-\frac{\sigma_n}{\sigma_0}\hat { J } _{n,0}- \frac{\sigma_n}{\sigma_1}\hat { J } _{n,1}-\cdots-\hat { J } _{n,n}}\leq \frac{\hat { J } _{n,n+1}}{1-\hat { J } _{n,0}-\hat { J } _{n,1}-\cdots-\hat { J } _{n,n}}\leq \frac{\hat { J } _{n,n+1}}{\hat { J } _{n,n+1}}=1.
\end{align*}
If we have $n\geq N$, then we~have
\begin{align*}
\frac{\sigma_n}{\sigma_{n-N}}\hat { J } _{n,n-N}+ \frac{\sigma_n}{\sigma_{n-N+1}}\hat { J } _{n,n-N+1}+\cdots+\hat { J } _{n,n}+\frac{\sigma_n}{\sigma_{n+1}}\hat { J } _{n,n+1}=1
\end{align*}
so that
\begin{align*}
\frac{\sigma_{n+1}}{\sigma_n}=\frac{\hat { J } _{n,n+1}}{1-\frac{\sigma_n}{\sigma_{n-N}}\hat { J } _{n,n-N}- \frac{\sigma_n}{\sigma_{n-N+1}}\hat { J } _{n,n-N+1}-\cdots-\hat { J } _{n,n}},
\end{align*}
and the previous argument applies, so that we deduce $0<\frac{\sigma_{n+1}}{\sigma_n}\leq 1$.
Thus, the coefficients~of the semi-stochastic Jacobi matrix $ \hat { J } _{i,j}$ determine in a recursive way the~coefficients~$0<\frac{\sigma_{n+1}}{\sigma_n}\leq 1$.
\end{proof}

\end{document}